\documentclass[a4paper,reqno]{amsart}

\usepackage[utf8]{inputenc}
\usepackage{amssymb}
\usepackage{enumitem}
\usepackage{mathrsfs}
\usepackage{mathtools}
\usepackage{exscale}
\usepackage{relsize}
\usepackage{stmaryrd}
\usepackage{color}
\setlist[enumerate]{label=\textnormal{(\roman*)}}

\usepackage{hyperref}

\newtheorem{theorem}{Theorem}[section]

\newtheorem{lemma}[theorem]{Lemma}
\newtheorem{proposition}[theorem]{Proposition}
\newtheorem{claim}[theorem]{Claim}
\theoremstyle{definition}
\newtheorem{definition}[theorem]{Definition}
\newtheorem{remark}[theorem]{Remark}

\numberwithin{equation}{section}

\newcommand\ud{\, \mathrm{d}}

\def\llbracket{[\![}
\def\rrbracket{]\!]}
\usepackage{mathtools}

\begin{document}

    \parindent=8pt

    \title[Wave equations with inverse-square potential]
    {Soliton resolution for energy critical wave equation with inverse-square potential in radial case}

    \author[X. ~Li]{Xuanying Li}
    \address{School of Mathematical Sciences,
        University of Science and Technology of China, Hefei 230026, Anhui, China}
    \email{lxy0418@mail.ustc.edu.cn}

    \author[C. ~Miao]{Changxing Miao}
    \address{Institute of Applied Physics and Computational Mathematics, Beijing, 100089, China}
    \email{$miao_changxing@iapcm.ac.cn$}

    \author[L. ~Zhao]{Lifeng Zhao}
    \address{School of Mathematical Sciences,
        University of Science and Technology of China, Hefei 230026, Anhui, China}
    \email{zhaolf@ustc.edu.cn}


    \begin{abstract}

        In this paper, we establish the soliton resolution for energy critical wave equation with inverse square potential in radial case and in all dimensions $N\geq3$. The structure of the radial linear operator  $\mathcal{L}_a :=-\Delta +\frac{a}{|x|^2}=A^*A$,  is essential for the channel of energy, where $A$ is a first order differential operator and $A^*$ is its adjoint operator. Modulation and analysis of the multi-solitons are performed in the function spaces $\dot{H}^1_a(\Bbb R^N)\times L^2(\Bbb R^N)$ associated with $\mathcal{L}_a$.

    \end{abstract}

    \maketitle

 	\tableofcontents

\section{Introduction}

    We consider the  nonlinear wave equation with inverse-square potential
    \begin{equation}\label{weq}
        \left\{
        \begin{aligned}
            \partial_{t}^{2}u-\Delta u+\frac{a}{|x|^2}u&=|u|^{\frac{4}{N-2}}u\\
            \vec{u}|_{t=0}=(u_0,u_1)&\in\dot{H}^{1}(\mathbb{R}^{N})\times L^{2}(\mathbb{R}^{N})
        \end{aligned}
        \right.
    \end{equation}
    We denote by $\mathcal{L}_a$ the operator
    \begin{align}
        \mathcal{L}_{a}:=-\Delta+\frac{a}{|x|^2},
    \end{align}
    which arises often in complicated mathematical and physical backgrounds, for example the Dirac equation with Coulomb potential, and the study of perturbations of spacetime metrics such as Schwartzshild and Reissner-Nordstr\"om (\cite{BPS, KSWW, VZ}). When $a>-\frac{(N-2)^2}{4}$ and $N\geq3$, we consider $\mathcal{L}_a$ as the Friedrichs extension of the quadratic form defined on $C_c^\infty(\Bbb R^N\backslash\{0\})$ via $$f\mapsto \int_{\Bbb R^N}\left(|\nabla f(x)|^2+a\frac{|f(x)|^2}{|x|^2}\right)dx.$$
   The operator $\mathcal{L}_a$ is positive. Thus for $a>-\frac{(N-2)^2}{4}$, we define the function space $\dot{H}_a^1$ whose norm is $$\|f\|_{\dot{H}_a^1}^2=\langle\mathcal{L}_{a}f, f\rangle.$$ We denote the energy space $\mathcal{H}_{a}(\mathbb{R}^N)=\dot{H}^{1}_{a}(\mathbb{R}^N)\times L^2(\mathbb{R}^N)$.

    The solution of \eqref{weq} admits the energy
    \begin{align}\label{energy}
        E(\vec{u}(t))=\int_{\Bbb R^N}\left(\frac12|\nabla_{t,x} u(t,x)|^2+a\frac{|u(t,x)|^2}{|x|^2}-\frac{N-2}{2N}|u(t,x)|^{\frac{2N}{N-2}}\right)dx,
    \end{align}
    which is conserved. For $\lambda>0$, $f\in\dot{H}^{1}_{a}(\mathbb{R}^N)$ and $g\in L^2(\mathbb{R}^{N})$, we let
    \begin{align}\label{def: scalling}
        f_{(\lambda)}(x)=\frac{1}{\lambda^{\frac{N-2}{2}}}f\left(\frac{x}{\lambda}\right),\qquad g_{[\lambda]}(x)=\frac{1}{\lambda^{\frac{N}{2}}}g\left(\frac{x}{\lambda}\right).
    \end{align}
    so that
    \begin{align*}
        \|f_{(\lambda)}\|_{\dot{H}_{a}^{1}}=\|f\|_{\dot{H}_{a}^{1}}\quad and\quad \|g_{[\lambda]}\|_{L^2}=\|g\|_{L^2}.
    \end{align*}
    If $u$ is a solution to \eqref{weq}, then $u_{(\lambda)}$ is also a solution. Moreover, $E((u_0)_{(\lambda)}, (u_1)_{[\lambda]})=E(u_0, u_1)$. Therefore, we call the equation \eqref{weq} energy-critical.

    There has been increasing interest in the study of the energy critical wave equation with inverse-square potential \eqref{weq} in recent years. Dispersive and Strichartz estimates for radial solutions are established in \cite{BPS, PST} and global well-posedness were proved under the spherically symmetric assumption in critical spaces. The range of admissible pairs and corresponding well-posedness were improved in \cite{MZZ}. In focusing case, the equation \eqref{weq} admits a ground state as shown in \cite{MZZ}
    \begin{align}\label{eq: the formula of Wa}
        W_a(x)=[N(N-2)\beta^2]^{\frac{N-2}{4}}\left(\frac{|x|^{\beta-1}}{1+|x|^{2\beta}}\right)^{\frac{N-2}{2}},\quad \beta=\sqrt{1+\left(\frac{2}{N-2}\right)^2a}
    \end{align}
    which is the unique  (up to symmetries of the equation) positive radial solution to $$\mathcal{L}_au=|u|^\frac{4}{N-2}u.$$ In \cite{MMZ}, the energy-critical case for large data in dimensions three and four were considered. They showed the global well-posedness and scattering for arbitrary data in the defocusing case and data below the ground states in the focusing case. The global well-posedness and scattering results were generalized in higher dimensions $N\geq5$ in \cite{DL}. However, there are very few results for solutions above the ground state. In this paper, we aim to establish the soliton resolution for radial solutions of \eqref{weq} in all dimensions $N\geq 3$.

    Soliton resolution is one of the long-standing topics in nonlinear dispersive equations. In fact, for nonlinear dispersive equations, it is widely believed that coherent structures (e.g. solitons, kinks, monopoles, vortices, etc) and free radiation describe the long-term asymptotic behavior of generic solutions. This belief is known as the soliton resolution conjecture, which states that asymptotically in time, the evolution of generic solutions decouples as a sum of modulated traveling waves and a free radiation. The soliton resolution conjecture arose in the 70's and 80's from various numerical experiments \cite{K} and integrability theory (\cite{Eck, ES, Sch}). For non-integrable model, the soliton resolution conjecture for energy-critical wave equations has been verified in radial case. Duyckaerts, Kenig and Merle \cite{DKM1} proved the soliton resolution in the three dimensional case near the ground state. The radial assumption was removed in \cite{DKM6}. Without a smallness assumption in the radial case, the soliton resolution along time sequences  for solutions bounded in the energy norm  was proved in \cite{DKM8} in dimension three, in \cite{CR} for all other odd dimensions, in \cite{CKLS} in dimension four and in \cite{JK} in dimension six. In \cite{DJKM}, Duyckaerts, Jia, Kenig and Merle proved that the decomposition holds for any solution of the energy-critical wave equation in space dimensions 3,4,5 for a sequence of times going to its maximal time. Without small size assumption, the soliton resolution for all times was verified in \cite{DKM} in radial case in three dimensions. In a series of works \cite{DKM3, DKM2, DKM7}, Duyckaerts, Kenig and Merle extend the result in \cite{DKM} to all odd dimensions. In their proof, they overcome the fundamental obstruction for the 3D case by reducing the study of a multisoliton solution to a finite dimensional system of ordinary differential equations on the modulation parameters. The key ingredient is to show that this system of equations creates some radiation, contradicting the existence of pure multisolitons. Corresponding results in dimension four were established recently in \cite{DKMM} and extended to dimension six in \cite{CDKM}, although the channel of energy fails in even dimensions. For more details about the soliton resolution of energy critical wave equations, please refer to the survey article \cite{Kenig2}.

    Recently there are also soliton resolution type results for other wave type equations. For wave maps, the soliton resolution along a time sequence was  proved by Cote \cite{Cote} and Jia, Kenig \cite{JK}. Universality of blow up profiles for small blow-up solutions was established in \cite{DJKM2}. The resolution for all time was established in 1-equivariant case by Duyckaerts, Kenig, Martel and Merle \cite{DKMM} and $k$-equivariant case by Jendrej and Lawrie \cite{JL5}. For wave maps, the two bubble dynamics was studied in \cite{JL1}. More quantitative description and uniqueness results were given in \cite{JL2, JL3, JL4}. In addition, the defocusing wave equations with potentials was systematically studied. The soliton resolution were established by Jia, Liu and Xu in \cite{JBX1, JBX2} and with Schlag in \cite{JBSX1}. For focusing wave equation outside a ball in $\Bbb R^3$ with Dirichlet boundary condition and a superquintic power nonlinearity, Duychaerts and Yang \cite{DY} classified all radial stationary solutions, and proved that all radial global solutions are asymptotically the sum of a stationary solution and a free radiation.

    \begin{theorem}\label{main thm}
        Let $N\geq3$, $a>0$. Let $\beta=\sqrt{1+\frac{4a}{(N-2)^2}}$ such that $(N-2)\beta>2$ is odd. Let $u$ be a radial solution of \eqref{weq}, with maximal time of existence $T_+$, such that
        \begin{equation}\label{bh}
        \sup_{0\leq t<T_+}\|\vec{u}(t)\|_{\mathcal{H}_a}<\infty.
        \end{equation}
        Then there exists $J\geq0$, signs $(\iota_j)_j\in \{\pm1\}^J$, and scaling parameters $(\lambda_j)_j\in(0,\infty)^J$ such that
        \begin{equation*}
            \forall j\in \llbracket1, J-1\rrbracket, \ \ \lim_{t\rightarrow T_+}\frac{\lambda_j(t)}{\lambda_{j+1}(t)}=+\infty,
        \end{equation*}
        and
        \begin{itemize}
            \item (Type II blow-up case). If $T_+<\infty$, then $J\geq1$ and there exists $(v_0, v_1)\in\mathcal{H}_a$ such that
            \begin{equation*}
                \lim_{t\rightarrow T_+}\left\|\vec{u}(t)-(v_0, v_1)-\sum_{j=1}^J(\iota_j{W_a}_{(\lambda_j)},0)\right\|_{\mathcal{H}_a}=0,
            \end{equation*}
            where $W_{a(\lambda_j)}(x)=\frac{1}{\lambda_j^\frac{N-2}{2}}W_a\left(\frac{x}{\lambda_j}\right)$ is the $\dot{H}^1$ scaling of the ground state $W_a(x)$. Furthermore, \begin{equation*}
                \lim_{t\rightarrow T_+}\frac{\lambda_1(t)}{T_+-t}=0.
            \end{equation*}
            \item (Global in time case). If $T_+=\infty$, then there exists a solution $v_L$ of the linear wave equation with inverse square potential such that
            \begin{equation*}
                \lim_{t\rightarrow +\infty}\left\|\vec{u}(t)-\vec{v}_L(t)-\sum_{j=1}^J(\iota_jW_{a(\lambda_j)},0)\right\|_{\mathcal{H}_a}=0.
            \end{equation*}
            Furthermore, \begin{equation*}
                \lim_{t\rightarrow \infty}\frac{\lambda_1(t)}{t}=0.
            \end{equation*}
        \end{itemize}
    \end{theorem}

    \begin{remark}

       {\it The restriction of $a$.} In order to establish the channel of energy, we will make use of Hankel transform,  which is the corresponding distorted Fourier transform associated to the operator $\mathcal{L}_a$.  During the derivation, we have to compute integrals involving the factor
       \begin{align*}
           \cos(r\rho_1-\sigma)\cos(r\rho_2-\sigma)&=\frac{1}{2}[\cos(r(\rho_1+\rho_2)-2\sigma)+\cos(r(\rho_1-\rho_2))],
       \end{align*}
       where \begin{equation*}
           2\sigma=\big((N-2)\beta+1\big)\frac{\pi}{2}.
       \end{equation*}
       Only if $2\sigma\in\pi\mathbb{Z}$, we have \begin{align*}
           \cos(r\rho_1-\sigma)\cos(r\rho_2-\sigma)&=\frac{1}{2}\big[(-1)^{k}\cos(r(\rho_1+\rho_2))+\cos(r(\rho_1-\rho_2))\big]
       \end{align*}
       and the computation of the integrals can be reduced to spherical Bessel functions. Compared with the classical wave equation case where $\beta=1$ and  $2\sigma=((N-2)+1)\frac{\pi}{2}$, we see that the coefficients of inverse square potential plays the role of spatial dimensions partially. In addition, we impose the restriction $(N-2)\beta>2$ to ensure $W_a\in L^2(\Bbb R^N)$. It is an interesting open problem to remove the restriction on $a$.
   \end{remark}
   \begin{remark} {\it Improvement of the scattering threshold in the radial case. } In \cite{MMZ}, global well-posedness and scattering holds for solutions with energy and kinetic energy  less than those of $W$.  The essential ingredient in their proof is related to sharp Sobolev embedding
       \begin{equation*}
           \|f\|_{L^\frac{2N}{N-2}(\Bbb R^N)}\leq C_a\|f\|_{\dot{H}^1_a(\Bbb R^N)},
       \end{equation*}
       where $C_a$ denotes the sharp constant, see \cite{KMVZZ}
       $$C_a=\frac{\|W_{a\wedge 0}\|_{\dot{H}^1_{a\wedge0}}}{\|W_{a\wedge 0}\|_{L^{\frac{2N}{N-2}}}}.$$
       However, according to Theorem \ref{main thm}, if the initial data is radial and has energy and kinetic energy less than those of the ground state $W_a$, then the solution is global and scattering holds. Thus in the radial case, the threshold is improved since $C_a<\|W_{a}\|_{\dot{H}^1_{a}}/\|W_{a}\|_{L^{\frac{2N}{N-2}}}$ for $a>0$.  In fact, under the radial assumption the equality in the Sobolev embedding  holds exactly when $f=W_a$.  We expect the improved scattering threshold holds for any $a>-\left(\frac{N-2}{2}\right)^2$ in the radial case.
   \end{remark}
   \medskip

   In their series of works, Duyckaerts, Kenig and Merle have developed a systematic strategy to quantify the ejection of energy, that occurs as we approach the final time $T_+$ for nonlinear wave equations, through a method of energy channels.  Exterior channel of energy estimates for the radial wave equation were first considered in three dimensions in \cite{DKM1}, for five dimensions in \cite{KLS}, and all odd dimensions in \cite{KLLS}. However, the exterior energy estimates were shown to fail in even dimensions in \cite{CKS}. Kenig \cite{Kenig} gave a survey on the method of energy channels. In this paper, we  need to establish the exterior energy estimates for the radial solutions to the free wave equations with inverse-square potentials
   \begin{equation}\label{lwp}
       \left\{
       \begin{aligned}
           \partial_{t}^{2}u-\Delta u+\frac{a}{|x|^2}u&=0, \\
           \vec{u}|_{t=0}=(u_0,u_1)&\in(\dot{H}^{1}_{a}\times L^{2})(\mathbb{R}^{N}).
       \end{aligned}
       \right.
   \end{equation}

   \medskip
   In fact, we will prove
   \begin{theorem}\label{thm:weq}\label{ce}
       For $(N-2)\beta$ is odd and  $(N-2)\beta\geq1$. Let $u_L$ be the solution of \eqref{lwp} with radial initial data  $(f,g)$. Then
       \begin{equation}\label{ineq: channel of energy}
           \max\limits_{\pm}\lim\limits_{t\rightarrow\pm\infty}\int_{r\geq|t|+R}|\nabla_{t,x}u_{L}(t,r)|^2r^{N-1}dr\geq \frac{1}{2}\|\operatorname{\pi}_{P(R)}^{\perp}(f,g)\|_{\dot{H}^1_a\times L^2(r\geq R,r^{N-1}dr)}^2,
       \end{equation}
       where
       \begin{align}\label{def: P(R)}
           P(R)\triangleq\operatorname{span}\bigg\{(r^{\alpha+2i-2},0),(0,r^{\alpha+2j-2})\bigg|i=1,2,\cdots,\Big\lfloor\frac{(N-2)\beta+4}{4}\Big\rfloor;\nonumber\\
           j=1,2,\cdots,\Big\lfloor\frac{(N-2)\beta+2}{4}\Big\rfloor;r\geq R\bigg\}.
       \end{align}
       Here $\operatorname{\pi}_{P(R)}^{\perp}$ denotes the orthogonal projection onto the complement of the plane $P(R)$ in $\dot{H}^1_a\times L^2(r\geq R,r^{N-1}dr)$.
   \end{theorem}

   \begin{remark}
       The conditions that $1\leq i\leq \lfloor\frac{(N-2)\beta+4}{4}\rfloor$ and $1\leq j\leq \lfloor\frac{(N-2)\beta+2}{4}\rfloor$ guarantee that
       \begin{align*}
           (r^{\alpha+2i-2},0),(0,r^{\alpha+2j-2})\in(\dot{H}_{a}^{1}\times L^2)(r\geq R,r^{N-1}dr).
       \end{align*}
   \end{remark}

   In order to establish Theorem \ref{ce}, an important observation is that for radial functions, the operator $\mathcal{L}_a$ can be represented as $\mathcal{L}_a=A^*A$, where $A=\partial_r+\frac{(N-2)(1-\beta)}{2r}$ and $A^*$ is the adjoint operator of $A$. Based on this observation, we define the norm  \begin{equation*}\label{def norm}
       \|f\|_{\dot{H}^{1}_{a}(r>R)}=\|A f\|_{L^2(r>R)}.
   \end{equation*}
 If we use other norms such as $ \|f\|_{\dot{H}^{1}_{a}(r>R)}=\|\sqrt{\mathcal{L}_a}f\|_{L^2(r>R)}$ or $\|f\|_{\dot{H}^1(r>R)}=\|\nabla f\|_{L^2(r>R)}$, then extra uncontrollable terms will appear, which makes the exterior energy estimates unavailable. We expect such observation also works for wave equations with other potentials.

 In addition, we will prove analogous estimates for the linearized wave equation with inverse square  potential at the ground state $W_a$ by using a completely different strategy from \cite{DKM2}. In fact, our method relies on the distorted Hankel transform. In our argument, we show that the linearized potential around $W_a$ produces a negligible error which does not affect the exterior energy estimates dramatically.  Our strategy also works for the linearized wave equation around $W$ in \cite{DKM2}.

    \medskip

    We conclude the introduction with a brief outline of the paper. Section 2 introduces notations and reviews some preliminaries such as Strichartz estimates and profile decomposition. Section 3 concerns the exterior energy estimates for linear wave equations with inverse square potentials. Analogous estimates for linearized wave equation with inverse square potentials around the ground state and multi-solitons are given in Section 4. Then in Section 5, we show the soliton resolution along a time sequence. We consider the case of non-radiative solutions in Section 6. In Section 7, we reduce the proof of soliton resolution to the study of a finite dimensional dynamical system on the scaling parameters and some of the coefficients arising in the expansion of the solutions. Finally, in Section 8 we prove a blow-up/ejection result for this dynamical system and complete the proof.

 {\bf Acknowledgement.} C. ~Miao was supported by the National Key Research and Development Program of China (No. 2020YFA0712900) and NSFC Grant of China (11831004). L.~Zhao was partially supported by the NSFC Grant of China (11771415) and the National Key Research and Development Program of China (2020YFA0713100).

    \section{Preliminaries}
    In this section, we will denote by $\dot{H}_{a}^1(\mathbb{R}^N)$ the homogeneous Sobolev space. Recording some standard facts about the Strichartz estimate, Profile decomposition and importing several relevant results about the ground state $W_a$ for general nonlinear wave equation with inverse-square potential \eqref{weq}.

   \subsection{Structure of linear operator}
   We define
   \begin{equation*}
       V(r)=\frac{(N-2)(1-\beta)}{2r}\triangleq\frac{c}{r},
   \end{equation*}
   where $c=\frac{(N-2)(1-\beta)}{2}$. For the radial case, there admits the decomposition $\mathcal{L}_{a}=A^\ast A$, where
    \begin{equation}\label{def operator A}
        A=\partial_{r}+V(r).
    \end{equation}
Therefore, the homogeneous Sobolev space defined by $\mathcal{L}_a$ has the norm
    \begin{equation}\label{def norm}
        \|f\|_{\dot{H}^{1}_{a}}=\|(\partial_{r}+cr^{-1})f\|_{L^2}
    \end{equation}
and  for any $R>0$, \begin{equation}\label{def norm}
    \|f\|_{\dot{H}^{1}_{a}(r>R)}=\|(\partial_{r}+cr^{-1})f\|_{L^2(r>R)}.
\end{equation}

\subsection{Strichartz estimates}
The Strichartz estimate and its generalized version \cite{GV} play a key role in the proof of well-posedness. Since the derivative of the nonlinearity is no longer Lipschitz continuous in the standard Strichartz space for dimensions $N>6$, Bulut, Czubak, Li, Pavlovic and Zhang \cite{BCLPZ} using certain “exotic Strichartz" spaces (which have same scaling with the standard Strichartz space but lower derivative) to prove the exotic Strichartz estimate. In \cite{MZZ}, the Strichartz-type estimates were proved for the solutions to the linear wave equation with inverse-square potential by C. Miao, J. Zhang and J. Zheng. For dimensions $N\geq3$ and any time interval $I\subset\mathbb{R}$, we introduce the following space
\begin{align*}
        S(I)=L_{t,x}^{\frac{2(N+1)}{N-2}}(I\times\mathbb{R}^{N}),\qquad W(I)=L_{t}^{\frac{2(N+1)}{N-1}}\dot{B}_{\frac{2(N+1)}{N-1},2}^{\frac{1}{2}}(I\times\mathbb{R}^N).
\end{align*}

Based on these spaces, we state the Strichartz estimates for the wave equation with inverse-square potential, which we frequently use throughout the paper.
\begin{lemma}\label{Strichartz estimate}
If $f\in W(I)'$, $(u_0,u_1)\in\mathcal{H}_{a}$, then the solution $u$ of $\partial_{t}^2u+\mathcal{L}_{a}u=f$ with initial data $(u_0,u_1)$ is in $S(I)\cap W(I)$ and
\begin{align}\label{ineq: Strichartz estimate}
\sup\limits_{t\in\mathbb{R}}\|\vec{u}(t)\|_{\mathcal{H}_a}+\|\vec{u}\|_{S(I)}+\|\vec{u}\|_{W(I)}\lesssim\|(u_0,u_1)\|_{\mathcal{H}_a}+\|f\|_{W(I)'}.
\end{align}
\end{lemma}

    \subsection{Profile decomposition}\label{profile}

    Profile decomposition plays an important role in scattering theory and soliton resolutions. The linear profile decomposition for the wave equation was established in \cite{BG} for dimensions $N=3,4$ and extended to high dimensions $N\geq5$ in \cite{CR}.  In \cite{MMZ} and \cite{DL} the linear profile decomposition was established for general solutions to the energy-critical wave equations with inverse square potential. Here we only recall the radial version.  Let $\{u_{0,n}, u_{1,n}\}$ be a bounded sequence of radial functions in $\dot{H}^1_a\times L^2$. There exists a subsequence of $\{u_{0,n}, u_{1,n}\}$ (that we will still denote by $\{u_{0,n}, u_{1,n}\}$) with the following properties. There exists a sequence  $\{U^j_L\}_{j\geq1}$ of radial solutions to the linear wave equation
    \begin{equation}\label{lweq}
        \partial_t^2u-\Delta u+\frac{a}{|x|^2}u=0
    \end{equation}
    with initial data $\{U_0^j, U_1^j\}_j$ in $\dot{H}_{a}^1\times L^2$ for $j\geq1$, sequences $\{\lambda_{j,n}\}_n$,  $\{t_{j,n}\}_n$ with $\lambda_{j,n}>0$, $t_{j,n}\in\Bbb R$, and $(w_{0,n}^Jw_{1,n}^J)\in \dot{H}^1_a\times L^2$ satisfying
    \begin{equation*}
        u_{0,n}=\sum_{j=1}^J\frac{1}{\lambda_{j,n}^\frac{N-2}{2}}U^j_L\left(\frac{-t_{j,n}}{\lambda_{j,n}}, \frac{x}{\lambda_{j,n}}\right)+w_{0,n}^J,
    \end{equation*}
and
\begin{equation*}
    u_{1,n}=\sum_{j=1}^J\frac{1}{\lambda_{j,n}^\frac{N}{2}}(\partial_tU^j_L)\left(\frac{-t_{j,n}}{\lambda_{j,n}}, \frac{x}{\lambda_{j,n}}\right)+w_{1,n}^J,
\end{equation*}
where
\begin{itemize}
\item the pseudo-orthogonality property holds: for $j\neq k$,
\begin{equation*}
\lim_{n \rightarrow \infty}\Big|\log\frac{\lambda_{j,n}}{\lambda_{k,n}}\Big|+\frac{|t_{j,n}-t_{k,n}|^2}{\lambda_{j,n}\lambda_{k,n}}=\infty.
\end{equation*}
\item Let $w_n^J$ be the solution to \eqref{lweq} with data $(w_{0,n}^J w_{1,n}^J)$, then
$$\lim_{J\rightarrow\infty}\limsup\limits_{n\rightarrow\infty}\|w_n^J\|_{S(\Bbb R)}=0.$$
\item We have the decoupling property for each $J\geq1$:
\begin{align*}
\lim\limits_{n \rightarrow\infty}
\bigg(&\|u_{0,n}\|_{\dot{H}^1_a}^2+\|u_{1,n}\|_{L^2}^2-\sum\limits_{j=1}^J\left\|U_L^j\Big(\frac{-t_{j,n}}{\lambda_{j,n}}\Big)\right\|_{\dot{H}^1_a}^2\\
&-\sum\limits_{j=1}^J\left\|\partial_{t}U_L^j\Big(\frac{-t_{j,n}}{\lambda_{j,n}}\Big)\right\|_{L^2}^2-\|w_{0,n}^J\|_{\dot{H}^1_a}^2-\|w_{0,n}^J\|_{L^2}^2\bigg)=0\\
 \lim_{n\rightarrow\infty}\bigg(E_a(u_{0,n},& u_{1,n})-\sum_{j=1}^JE_a
 \left(U_L^j\left(\frac{-t_{j,n}}{\lambda_{j,n}}\right),\partial_tU_L^j\left(\frac{-t_{j,n}}{\lambda_{j,n}}\right)\right)-E_a(w_{0,n}^J, w_{1,n}^J)\bigg)=0
 \end{align*}
 \end{itemize}

\subsection{Ground state $W_a$}
In this subsection, we gather a few purely computational proofs. In $N$ dimensional space and for $a>0$, the ground state soliton is the unique (up to symmetries of the equation) positive solution of \eqref{lwp}. It was given in~\cite{KMVZZ} that the ground state for $a>-\frac{(N-2)^2}{4}$ can be given explicitly,
    \begin{align}\label{eq: the formula of Wa}
        W_a(x)=[N(N-2)\beta^2]^{\frac{N-2}{4}}\left(\frac{|x|^{\beta-1}}{1+|x|^{2\beta}}\right)^{\frac{N-2}{2}},\quad \beta=\sqrt{1+\left(\frac{2}{N-2}\right)^2a}.
    \end{align}

    Denote $\Lambda\triangleq\frac{N-2}{2}+x\cdot\nabla$, $\Lambda_0\triangleq\frac{N}{2}+x\cdot\nabla$. From \eqref{eq: the formula of Wa} we have
    \begin{align}\label{asympototic1}
        |\Lambda_0\Lambda W_a(x)|+|\Lambda W_a|+|W_a|\lesssim\min\left\{|x|^{\frac{(\beta-1)N}{2}-\beta+1},|x|^{-\frac{(\beta+1)N}{2}+\beta+1}\right\},
    \end{align}
    \begin{align}\label{asympototic2}
        \left|(\partial_{r}+\frac{c}{r})\Lambda W_a\right|+\left|(\partial_{r}+\frac{c}{r}) W_a\right|\lesssim\min\left\{|x|^{\frac{(\beta-1)N}{2}-\beta},|x|^{-\frac{(\beta+1)N}{2}+\beta}\right\},
    \end{align}
    and
    \begin{align}\label{asympototic3}
        \left|(-\Delta+\frac{a}{r^2})\Lambda W_a\right|\lesssim\left\{|x|^{\frac{(\beta-1)N}{2}-\beta-1},|x|^{-\frac{(\beta+1)N}{2}+\beta-1}\right\}.
    \end{align}
    Moreover, for any $a,b,c,d>1$ with $b+d>N$, and $|x|=r$, under the assumption that $0<\lambda<\mu$, summing up over all estimates, we conclude
    \begin{align}\label{main W}
        &\quad\int_{\mathbb{R}^N}\min\bigg(\bigg(\frac{|x|}{\lambda}\bigg)^a,\bigg(\frac{\lambda}{|x|}\bigg)^b\bigg)\min\bigg(\bigg(\frac{|x|}{\mu}\bigg)^c,\bigg(\frac{\mu}{|x|}\bigg)^d\bigg)\ud x\nonumber\\
        &=\bigg(\int\limits_{0}^{\lambda}+\int\limits_{\lambda}^{\mu}+\int\limits_{\mu}^{\infty}\bigg)\min\bigg(\bigg(\frac{r}{\lambda}\bigg)^a,\bigg(\frac{\lambda}{r}\bigg)^b\bigg)\min\bigg(\bigg(\frac{r}{\mu}\bigg)^c,\bigg(\frac{\mu}{r}\bigg)^d\bigg)r^{N-1}\ud r\nonumber\\
        &=\int\limits_{0}^{\lambda}\bigg(\frac{r}{\lambda}\bigg)^a\bigg(\frac{r}{\mu}\bigg)^cr^{N-1}\ud r+\int\limits_{\lambda}^{\mu}\bigg(\frac{\lambda}{r}\bigg)^b\bigg(\frac{r}{\mu}\bigg)^cr^{N-1}\ud r+\int\limits_{\mu}^{\infty}\bigg(\frac{\lambda}{r}\bigg)^b\bigg(\frac{\mu}{r}\bigg)^dr^{N-1}\ud r\nonumber\\
        &\lesssim\displaystyle\max\left\{\lambda^{b}\mu^{N-b},\lambda^{N+c}\mu^{-c}\right\}.
    \end{align}

The following lemma introduces some useful inequalities by \eqref{main W}. The proof of this lemma may be found in \cite{DKM7}.
\begin{lemma}\emph{(\cite{DKM7})}
Let $0<\lambda<\mu$. Assume $(N-2)\beta\geq 2$, as a consequence of \eqref{asympototic1}, \eqref{asympototic2} and \eqref{asympototic3}, it holds that
\begin{align}\label{2.10}
\int_{\mathbb{R}^{N}}\Big|(\partial_r&+\frac{c}{r})\left(\Lambda {W_a}_{(\lambda)}\right) \cdot(\partial_r+\frac{c}{r})\left(\Lambda {W_a}_{(\mu)}\right)\Big|\nonumber\\
&+\int_{\mathbb{R}^{N}}\Big|(\partial_r+\frac{c}{r}) {W_a}_{(\lambda)} \cdot (\partial_r+\frac{c}{r}) {W_a}_{(\mu)} \Big|\lesssim\bigg(\frac{\lambda}{\mu}\bigg)^{\frac{(N-2)\beta}{2}},
\end{align}
\begin{align}\label{the bdd of Lambda Wa and Lambda Wa}
\int_{\mathbb{R}^{N}}\left|(\Lambda W_a)_{[\lambda]}(\Lambda W_a)_{[\mu]}\right|+\int_{\mathbb{R}^{N}}\left|(\Lambda W_a)_{[\lambda]}\left(\Lambda_{0} \Lambda W_a\right)_{[\mu]}\right| \lesssim\left(\frac{\lambda}{\mu}\right)^{\frac{(N-2)\beta}{2}-1},
\end{align}

\begin{align*}
&\left\|{W_a}_{(\lambda)} {W_a}_{(\mu)}^{\frac{4}{N-2}}\right\|_{L^{\frac{2 N}{N+2}}}
\lesssim\left\{
\begin{array}{l}
\left(\frac{\lambda}{\mu}\right)^{\frac{(N-2)\beta}{2}},\qquad \text{if}\quad N\geq6\\
\left(\frac{\lambda}{\mu}\right)^{\frac{(N-2)\beta+(N-6)}{2}},\quad \text{if}\quad N\leq5.
\end{array}\right.
\end{align*}

\begin{align}\label{2.12}
&\left\|{W_a}_{(\mu)} {W_a}_{(\lambda)}^{\frac{4}{N-2}}\right\|_{L^{\frac{2 N}{N+2}}}
\lesssim\left\{
\begin{array}{l}
\left(\frac{\lambda}{\mu}\right)^{2\beta},\qquad \text{if}\quad N\geq6\\
\left(\frac{\lambda}{\mu}\right)^{\frac{(N-2)\beta}{2}},\quad \text{if}\quad N\leq5.
\end{array}\right.
\end{align}

\begin{align*}
\int\Big|(\Lambda W_a)_{[\lambda]}\Big((-\Delta+\frac{a}{r^2})\Lambda W_a\Big)_{[\mu]}\Big| \lesssim\left(\frac{\lambda}{\mu}\right)^{\frac{(N-2)\beta}{2}-1},
\end{align*}
\begin{align}\label{2.13}
\int\Big|(\Lambda W_a)_{[\mu]}\Big((-\Delta+\frac{a}{r^2}) \Lambda W_a\Big)_{[\lambda]}\Big| \lesssim\left(\frac{\lambda}{\mu}\right)^{\frac{(N-2)\beta}{2}},
\end{align}

\begin{align}\label{ineq: Wa N N-2}
\int{W_a}_{(\lambda)}^{\frac{N}{N-2}}{W_a}_{(\mu)}^{\frac{N}{N-2}}\ud x\lesssim\left(\frac{\lambda}{\mu}\right)^{\frac{N\beta}{2}}.
\end{align}
\end{lemma}
    \medskip

\begin{proposition}\label{prop: L1L2 bdd of Wa}\emph{(\cite{DKM7})}
Assume $N\geq 5$. Let $0<\lambda<\mu$. Then
\begin{align*}
\left\|\chi_{\{|x| \geq|t|\}} \min \left\{{W_a}_{(\lambda)}^{\frac{4}{N-2}} {W_a}_{(\mu)}, {W_a}_{(\mu)}^{\frac{4}{N-2}} {W_a}_{(\lambda)}\right\}\right\|_{L^{1}\left(\mathbb{R}, L^{2}\right)}\lesssim\left(\frac{\lambda}{\mu}\right)^{\frac{(N+2)\beta}{4}}.
\end{align*}
\end{proposition}
The proof of this proposition is similar to Claim A.3 in the appendix of \cite{DKM7}.

\medskip
\subsection{Pointwise bounds}
In this subsection, we record some inequalities associated with the nonlinear term of \eqref{weq}, which will be used in Section \ref{section 7}.
\begin{proposition}\emph{(\cite{DKM7})}
Assume $N\geq 5$, $J\geq 1$. For all $(y_1, y_2,\ldots,h)\in\mathbb{R}^{J+1}$, we have
\begin{align}\label{nonlinear estimate for y1 to yJ}
&\left|\frac{N-2}{2 N}\bigg|\sum_{j=1}^{J} y_{j}+h\bigg|^{\frac{2 N}{N-2}}-\frac{N-2}{2 N} \sum_{j=1}^{J}\left|y_{j}\right|^{\frac{2 N}{N-2}}-\sum_{j=1}^{J}\left|y_{j}\right|^{\frac{4}{N-2}} y_{j} h-\sum_{1 \leq j, k \leq J}\left|y_{j}\right|^{\frac{4}{N-2}} y_{j} y_{k}\right|\nonumber\\
\leq&\sum_{1 \leq j<k \leq J}\left(\min \left\{\left|y_{j}\right|^{\frac{4}{N-2}} y_{k}^{2},\left|y_{k}\right|^{\frac{4}{N-2}} y_{j}^{2}\right\}+\min \left\{\left|y_{j}\right|^{\frac{N+2}{N-2}}\left|y_{k}\right|,\left|y_{k}\right|^{\frac{N+2}{N-2}}\left|y_{j}\right|\right\}\right)\nonumber\\
\quad&+|h|^{\frac{2 N}{N-2}}+\sum_{j=1}^{J}\left|y_{j}\right|^{\frac{4}{N-2}}|h|^{2}.
\end{align}
\end{proposition}

\medskip
\begin{proposition}\emph{(\cite{DKM7})}
Let $(a,b,c)\in\mathbb{R}^3$ with $a\neq 0$, then
\begin{align}\label{eq: 2 nonlinear estimate}
\Big|F(a+b)-F(a)-F'(a)b\Big|\lesssim \chi_{\{|b|\leq|a|\}}|b|^{2}|a|^{\frac{6-N}{N-2}}+\chi_{\{|b|>|a|\}}|b|^{\frac{N+2}{N-2}},
\end{align}
and
\begin{align}\label{eq: estimate of F2}
&\Big|F(a+b+c)-F(a+b)-F(a+c)+F(a)\Big|
\lesssim\left\{\begin{array}{l}
|a|^{\frac{6-N}{2(N-2)}}|b|^{\frac{N+2}{2(N-2)}}|c|,\qquad if\quad N\geq6\\
|b||c|(|a|+|b|+|c|)^{\frac{6-N}{N-2}},\quad if\quad N\leq5.
\end{array}\right.
\end{align}
where $F(a)=|a|^{\frac{4}{N-2}}a$.

\end{proposition}

\medskip
\section{Exterior energy estimate for the linear equation}\label{section-linear}
    In this section we consider the lower bound on the exterior energy for radial solution of the free wave equation with the inverse-square potential \eqref{lwp}. Our goal is to prove Theorem \ref{ce}.
%
%

\subsection{Spectrum of $-\Delta+\frac{a}{|x|^2}$ and Hankel transform}
\medskip
Before proving Theorem \ref{thm:weq}, we establish the Hankel transform of the operator $\mathcal{L}_a=-\Delta+\frac{a}{|x|^2}$ on $\mathbb{R}^N$, and recall a few standard facts about the solution to \eqref{lwp} with radial initial data.

Let us firstly consider the eigenvalue problem associated with the operator $\mathcal{L}_a$
\begin{align}\label{eigenvalue equation}
\left(-\Delta+\frac{a}{|x|^2}\right)u=\rho^2 u.
\end{align}
Since $u$ is radially symmetric, we let $u(x)=h(r)$, then
\begin{align*}
h''(r)+\frac{N-1}{r}h'(r)-(\frac{a}{r^2}-\rho^2)h(r)=0.
\end{align*}
Let $\lambda=\rho r$ and $h(r)=\lambda^{-\frac{N-2}{2}}\tilde{h}(\lambda)$, we obtain

\begin{align*}\label{bassel equation}
\tilde{h}''(\lambda)+\frac{1}{\lambda}\tilde{h}'(\lambda)+\left(1-\frac{(N-2)^2+4a}{4\lambda^2}\right)\tilde{h}(\lambda)=0.
\end{align*}
Define $\nu=\frac{\sqrt{(N-2)^2+4a}}{2}=\frac{N-2}{2}\beta$, then the Bessel function $J_{\nu}(\lambda)$ solves the Bessel equation \eqref{eigenvalue equation}, and the eigenfunction corresponding to the spectrum $\rho^2$ can be expressed by
\begin{align*}
h_{\rho}(r)=(r\rho)^{-\frac{N-2}{2}}J_{\nu}(r\rho).
\end{align*}

We define the following elliptic operator
\begin{align*}
A_{\nu}\triangleq-\partial_{r}^2-\frac{N-1}{r}\partial_{r}+\frac{a}{r^2}=-\partial_{r}^2-\frac{N-1}{r}\partial_{r}+\frac{\nu^2-(\frac{N-2}{2})^2}{r^2}.
\end{align*}
The Hankel transform of order $\nu$ is defined by:
\begin{equation}\label{Hankel transform}
(\mathcal{H}_{\nu}f)(\rho)\triangleq\tilde{f}(\rho)=\int_{0}^{\infty}(r\rho)^{-\frac{N-2}{2}}J_{\nu}(r\rho)f(r)r^{N-1}\ud r.
\end{equation}

We recall the following properties of the Hankel transform given in \cite{MZZ}.
\begin{lemma}
Let $\mathcal{H}_{\nu}$ and $\mathcal{A}_{\nu}$ be defined as above. Then
\begin{itemize}
\item $\mathcal{H}_{\nu}=\mathcal{H}_{\nu}^{-1}$,
\item $\mathcal{H}_{\nu}$ is self-adjoint, i.e. $\mathcal{H}_{\nu}=\mathcal{H}_{\nu}^{\ast}$,
\item $\mathcal{H}_{\nu}$ is an $L^2$ isometry, i.e. $\|\mathcal{H}_{\nu}\phi\|_{L_{\xi}^{2}}=\|\phi\|_{L_{x}^2}$,
\item $\mathcal{H}_{\nu}(A_{\nu}\phi)(\xi)=|\xi|^2(\mathcal{H}_{\nu}\phi)(\xi)$, for $\phi\in L^2$.
\end{itemize}
\end{lemma}


Taking the Hankel transformation of both sides of \eqref{lwp}, we have
\begin{equation}\label{new equation}
\left\{
\begin{aligned}
&\partial_{t}^{2}\tilde{u}+\rho^2\tilde{u}=0,\\
&\vec{u}|_{t=0}=(\tilde{f},\tilde{g}).\\
\end{aligned}
\right.
\end{equation}
The solution $\tilde{u}$ to \eqref{new equation} is given by
\begin{align}\label{Hankel transform for u}
\tilde{u}(t,\rho)=\frac{1}{2}\left(\tilde{f}(\rho)+\frac{\tilde{g}(\rho)}{i\rho}\right)e^{it\rho}+\frac{1}{2}\left(\tilde{f}(\rho)-\frac{\tilde{g}(\rho)}{i\rho}\right)e^{-it\rho}.
\end{align}
Substituting \eqref{Hankel transform for u} into \eqref{Hankel transform}, for the solution $u(t,r)$ this means that
\begin{align}\label{eq:u}
u(t,r)=\int_{0}^{\infty}(r\rho)^{-\frac{N-2}{2}}J_{\nu}(r\rho)\Big(\cos(t\rho)\tilde{f}(\rho)+\frac{1}{\rho}\sin(t\rho)\tilde{g}(\rho)\Big)\rho^{N-1}d\rho,\nonumber\\
\partial_{t}u(t,r)=\int_{0}^{\infty}(r\rho)^{-\frac{N-2}{2}}J_{\nu}(r\rho)\Big(-\sin(t\rho)\rho \tilde{f}(\rho)+\cos(t\rho)\tilde{g}(\rho)\Big)\rho^{N-1}d\rho.
\end{align}

The standard asymptotic behavior  for $|z|\rightarrow\infty$ of the Bessel functions is given in \cite{AS} that
\begin{align}\label{eq:J}
&J_{\nu}(z)\sim\sqrt{\frac{2}{\pi z}}\Big[(1+\omega_2(z))\cos(z-\sigma)+\omega_{1}(z)\sin(z-\sigma)\Big],\nonumber\\
&J_{\nu}'(z)\sim\sqrt{\frac{2}{\pi z}}\Big[-(1+\tilde{\omega}_2(z))\sin(z-\sigma)+\tilde{\omega}_{1}(z)\cos(z-\sigma)\Big],
\end{align}
where the phase-shift $\sigma=\frac{(N-2)\beta+1}{2}\frac{\pi}{2}$, and for $n\in\mathbb{N}, z\geq1$,
\begin{align*}
&|\omega_1^{(n)}(z)|+|\tilde{\omega}_1^{(n)}(z)|\leq C_{n}z^{-1-n},\\
&|\omega_2^{(n)}(z)|+|\tilde{\omega}_2^{(n)}(z)|\leq C_{n}z^{-2-n}.
\end{align*}

We now sketch the proof of Theorem \ref{ce} and divide it into the following smaller steps.

\medskip
\subsection{Asymptotic formula}
Computing the asymptotic form of the exterior energy as in \cite{KLLS}, for $t\geq0$, from \eqref{eq:u} to \eqref{eq:J}, we have
\begin{align}\label{eq:u,ut}
&\|\partial_{t}u\|_{L^2(r\geq t+R,r^{N-1}\ud r)}^{2}+\|u\|_{\dot{H}_{a}^{1}(r\geq t+R,r^{N-1}\ud r)}^2\nonumber\\
=&\frac{2}{\pi}\lim\limits_{\varepsilon\rightarrow0}\int_{t+R}^{\infty}\iint\limits_{\rho_1,\rho_2>0}\Big(\sin(t\rho_1)\rho_1 \tilde{f}(\rho_1)-\cos(t\rho_1)\tilde{g}(\rho_1)\Big)\Big(\sin(t\rho_2)\rho_2\overline{\tilde{f}(\rho_2)}-\cos(t\rho_2)\overline{\tilde{g}(\rho_2)}\Big)\nonumber\\
&\times\Big[(1+\omega_2(r\rho_1))\cos(r\rho_1-\sigma)+\omega_1(r\rho_1)\sin(r\rho_1-\sigma)\Big]\nonumber\\
&\times\Big[(1+\omega_2(r\rho_2))\cos(r\rho_2-\sigma)+\omega_1(r\rho_2)\sin(r\rho_2-\sigma)\Big]\rho_1^{\frac{N-1}{2}}\rho_2^{\frac{N-1}{2}}e^{-\varepsilon r}\ud \rho_1\ud \rho_2\ud r\nonumber\\
&+\frac{2}{\pi}\lim\limits_{\varepsilon\rightarrow0}\int_{t+R}^{\infty}\iint\limits_{\rho_1,\rho_2>0}\Big(\cos(t\rho_1)\rho_1 \tilde{f}(\rho_1)+\sin(t\rho_1)\tilde{g}(\rho_1)\Big)\Big(\cos(t\rho_2)\rho_2\overline{\tilde{f}(\rho_2)}+\sin(t\rho_2)\overline{\tilde{g}(\rho_2)}\Big)\nonumber\\
&\times\Big[(1+\omega_2(r\rho_1))\cos(r\rho_1-\sigma)+\omega_1(r\rho_1)\sin(r\rho_1-\sigma)\Big]\nonumber\\
&\times\Big[(1+\omega_2(r\rho_2))\cos(r\rho_2-\sigma)+\omega_1(r\rho_2)\sin(r\rho_2-\sigma)\Big]\rho_1^{\frac{N-1}{2}}\rho_2^{\frac{N-1}{2}}e^{-\varepsilon r}\ud \rho_1\ud \rho_2\ud r.
\end{align}
For all $\varepsilon>0$, the error terms which contain $\omega_{i}$, $i=1,2$, example as
\begin{align*}\label{eq:error}
E_{1}(\varepsilon)\triangleq&\int_{t+R}^{\infty}\iint\limits_{\rho_1,\rho_2>0}\sin(t\rho_1)\sin(t\rho_2)\cos(r\rho_1-\sigma)\sin(r\rho_2-\sigma)\\
&\cdot\omega_1(r\rho_1)\cdot \tilde{f}(\rho_1)\overline{\tilde{g}(\rho_2)}(\rho_1\rho_2)^{\frac{N+1}{2}}e^{-\varepsilon r}\ud \rho_1\ud \rho_2\ud r.
\end{align*}
Using the asymptotic expansions of the Bessel functions, we can absorb all error terms in the $o(1)$ term. Therefore as in \cite{KLLS}, \eqref{eq:u,ut} can be reduced to
\begin{align*}
&\quad\|\partial_{t}u\|_{L^2(r\geq t+R,r^{N-1}\ud r)}^{2}+\|u\|_{\dot{H}_{a}^{1}(r\geq t+R,r^{N-1}\ud r)}^2\nonumber\\
=&\frac{2}{\pi}\lim\limits_{\varepsilon\rightarrow0}\int_{t+R}^{\infty}\iint\limits_{\rho_1,\rho_2>0}\Big(-\sin(t\rho_1)\rho_1\tilde{f}(\rho_1)+\cos(t\rho_1)\tilde{g}(\rho_1)\Big)\Big(-\sin(t\rho_2)\rho_2\overline{\tilde{f}(\rho_2)}+\cos(t\rho_2)\overline{\tilde{g}(\rho_2)}\Big)\nonumber\\
&\times\cos(r\rho_1-\sigma)\cos(r\rho_2-\sigma)\rho_1^{\frac{N-1}{2}}\rho_2^{\frac{N-1}{2}}e^{-\varepsilon r}\ud \rho_1\ud \rho_2\ud r\nonumber\\
&+\frac{2}{\pi}\lim\limits_{\varepsilon\rightarrow0}\int_{t+R}^{\infty}\iint\limits_{\rho_1,\rho_2>0}\Big(\cos(t\rho_1)\rho_1\tilde{f}(\rho_1)+\sin(t\rho_1)\nonumber\tilde{g}(\rho_1)\Big)\Big(\cos(t\rho_2)\rho_2\overline{\tilde{f}(\rho_2)}+\sin(t\rho_2)\overline{\tilde{g}(\rho_2)}\Big)\nonumber\\
&\times\sin(r\rho_1-\sigma)\sin(r\rho_2-\sigma)\rho_1^{\frac{N-1}{2}}\rho_2^{\frac{N-1}{2}}e^{-\varepsilon r}\ud \rho_1\ud \rho_2\ud r+o(1),
\end{align*}
where the $o(1)$ is for $t\rightarrow\infty$ and $\sigma=((N-2)\beta+1)\frac{\pi}{4}$. Since
\begin{align*}
\cos(r\rho_1-\sigma)\cos(r\rho_2-\sigma)&=\frac{1}{2}\Big[\cos(r(\rho_1+\rho_2)-2\sigma)+\cos(r(\rho_1-\rho_2))\Big],\\
\sin(r\rho_1-\sigma)\sin(r\rho_2-\sigma)&=\frac{1}{2}\Big[-\cos(r(\rho_1+\rho_2)-2\sigma)+\cos(r(\rho_1-\rho_2))\Big],
\end{align*}
for arbitrary $t\in\mathbb{R}$, it holds that
\begin{equation*}
\lim\limits_{\varepsilon\rightarrow 0+}\int_{t}^{\infty}\cos(a r)e^{\varepsilon r}\ud r=\lim\limits_{\varepsilon\rightarrow 0+}\frac{1}{2}\left(-\frac{e^{t(i a-\varepsilon)}}{i a-\varepsilon}+\frac{e^{-t(i a+\varepsilon)}}{i a+\varepsilon}\right)=\pi\delta_0(a)-\frac{\sin(ta)}{a}.
\end{equation*}
However, $\lim\limits_{\varepsilon\rightarrow 0+}\int_{t}^{\infty}\sin(a r)e^{\varepsilon r}\ud r$ is not well-defined. Therefore we introduce a restrictive condition
\begin{equation*}
2\sigma=\big((N-2)\beta+1\big)\frac{\pi}{2}\in\pi\mathbb{Z}.
\end{equation*}
This means $k\triangleq\frac{(N-2)\beta+1}{2}\in\mathbb{Z}$. Thus,
\begin{align*}
\cos(r\rho_1-\sigma)\cos(r\rho_2-\sigma)&=\frac{1}{2}\Big[(-1)^{k}\cos(r(\rho_1+\rho_2))+\cos(r(\rho_1-\rho_2))\Big],\\
\sin(r\rho_1-\sigma)\sin(r\rho_2-\sigma)&=\frac{1}{2}\Big[-(-1)^k\cos(r(\rho_1+\rho_2))+\cos(r(\rho_1-\rho_2))\Big].
\end{align*}

Since the order $\nu$ of Bessel function is a half-integer, we introduce the {\bf spherical Bessel functions} $j_{n}(z)$,
\begin{equation*}\label{def:jn}
j_{n-\frac{1}{2}}(z)=\left(\frac{\pi}{2z}\right)^{\frac{1}{2}}J_{n}(z).
\end{equation*}
Using $k=\frac{(N-2)\beta+1}{2}$, $\nu=\frac{(N-2)\beta}{2}$, it is easy to check that
$j_{k-1}(z)=(\frac{\pi}{2z})^{\frac{1}{2}}J_{\nu}(z)$.
For $n\geq0$ the spherical Bessel function $j_{n}(z)$ satisfies
\begin{equation*}\label{eq:jn}
j_{n}(z)=(-z)^{n}
\left(\frac{1}{z}\frac{d}{dz}\right)^{n}\frac{\sin z}{z}.
\end{equation*}
It is important to notice that the parity of the function $j_n(z)$ only depends on $n$ odd or even. Since $k\in\mathbb{Z}$, $J_{\nu}(z)$ can be replaced by $j_{k-1}(z)$. Treating $k$ with even and odd cases respectively, we conclude the following lemma as in \cite{KLLS},

\begin{lemma}\label{asf,asg}
Assume $(N-2)\beta\geq1$ and $(N-2)\beta$ is odd. Take $\kappa(x)$ to be a normalized bump function on $\mathbb{R}$, that is smooth, even, nonnegative, and such that $\kappa(|x|)$ is decreasing, supp$\kappa\subset[-1,1]$ and $\int_{\mathbb{R}}\kappa(x)\ud x=1$. Denote $\kappa_{\varepsilon}(x)=\frac{1}{\varepsilon}\kappa(\frac{x}{\varepsilon})$. For every $R>0$, every radial solution of \eqref{lwp} satisfies
\begin{align}\label{eq:lower bound}
\max\limits_{\pm}&\lim\limits_{t\rightarrow\pm\infty}\left(\|\partial_{t}u(t)\|_{L^2(r\geq|t|+R,r^{N-1}\ud r)}^2+\|u(t)\|_{\dot{H}^1_a(r\geq|t|+R,r^{N-1}\ud r)}^2\right)\nonumber\\
\geq& AS_{N}^{a}(f)+AS_{N}^{a}(g),
\end{align}
with
\begin{align}\label{eq:ASf}
AS_{N}^{a}(f)=&\frac{1}{2}\int_{0}^{\infty}\Big|(\partial_{r}+\frac{c}{r})f(r)\Big|^2r^{N-1}\ud r\nonumber\\
-&\frac{1}{2\pi^2}\lim\limits_{\varepsilon\rightarrow0}\iint\limits_{r_1,r_2>0}K_{R,\varepsilon}(r_1,r_2)u_0(r_1)\overline{u_0(r_2)}(r_1r_2)^{\frac{N-1}{2}}\ud r_1\ud r_2,
\end{align}
where
\begin{equation*}
K_{R,\varepsilon}(r_1,r_2)=\frac{1}{2r_1^2r_2^2}\int_{-R}^{R}(\kappa_{\varepsilon}\ast\mathcal{F}\psi_{k-1})(\frac{\xi}{r_1})\overline{(\kappa_{\varepsilon}\ast\mathcal{F}\psi_{k-1})}(\frac{\xi}{r_2})\ud \xi,
\end{equation*}
and
\begin{align}\label{eq:ASg}
AS_{N}^{a}(g)=&\frac{1}{2}\int_{0}^{\infty}|g(r)|^2r^{N-1}\ud r\nonumber\\
-&\frac{1}{2\pi^2}\lim\limits_{\varepsilon\rightarrow0}\iint\limits_{r_1,r_2>0}\tilde{K}_{R,\varepsilon}(r_1,r_2)u_1(r_1)\overline{u_1(r_2)}(r_1r_2)^{\frac{N-1}{2}}\ud r_1\ud r_2,
\end{align}
where
\begin{equation*}
\tilde{K}_{R,\varepsilon}(r_1,r_2)=\frac{1}{2r_1r_2}\int_{-R}^{R}(\kappa_{\varepsilon}\ast\mathcal{F}\varphi_{k-1})(\frac{\xi}{r_1})\overline{(\kappa_{\varepsilon}\ast\mathcal{F}\varphi_{k-1})}(\frac{\xi}{r_2})\ud \xi.
\end{equation*}
Here $k=\frac{(N-2)\beta+1}{2}$, $\varphi_{k-1}(z)=z\cdot j_{k-1}(z)$, $\psi_{k-1}(z)=z\varphi_{k-1}(z)=z^2j_{k-1}(z)$ with $j_{k-1}(z)$ being the spherical Bessel functions defined in \eqref{def:jn}. Moreover, the equality \eqref{eq:lower bound} holds when initial data is $(f,0)$ or $(0,g)$, and it holds for both time directions.
\end{lemma}

\subsection{Algebraic expression for the projection}
Before establishing the asymptotic behavior, we first introduce the Cauchy matrix as in \cite{S}.
Let $a_{1}$, $\cdots$, $a_{k}$ be linearly independent vectors in some inner product space $(V,\langle,\rangle)$ and let
$$
W=\operatorname{span}\left\{a_{1}, \cdots a_{k}\right\}.
$$
Taking any vector $u \in V$, the orthogonal projection onto $W^{\perp}$ is written as
\begin{equation*}
\operatorname{\pi}_{W^{\perp}} u=u-\left(\lambda_{1} a_{1}+\cdots \lambda_{k} a_{k}\right),
\end{equation*}
with the coefficients satisfying
\begin{equation*}
\left\langle\operatorname{\pi}_{W^{\perp}} u, a_{j}\right\rangle=\left\langle u, a_{j}\right\rangle-\sum\limits_{i} \lambda_{i}\left\langle a_{i}, a_{j}\right\rangle=0.
\end{equation*}
Denote
\begin{equation*}
U=\left[\left\langle u, a_{i}\right\rangle\right]_{1 \times k}, \quad \Lambda=\left[\lambda_{i}\right]_{1 \times k}, \quad A=\left[\left\langle a_{i}, a_{j}\right\rangle\right]_{k \times k}.
\end{equation*}
Using the fact that $A$ is symmetric, invertible and positive definite, setting $A=\left[a_{i j}\right], B=A^{-1}=\left[b_{i j}\right]$, then
\begin{align}\label{proj u norm}
\left\|\operatorname{\pi}_{W^{\perp}}u\right\|^{2}&=\langle u, u\rangle-\sum_{i,j=1}^{k}\lambda_{i}\lambda_{j}\left\langle a_{i}, a_{j}\right\rangle=\langle u, u\rangle-U A^{-1} U^{t}\nonumber\\
&=\langle u, u\rangle-\sum_{i,j=1}^{k}b_{i,j}\left\langle u,a_{i}\right\rangle\left\langle u,a_{j}\right\rangle.
\end{align}

In order to compute the matrix $B$, we need to introduce the {\bf Cauchy Matrix}. If $A$ is an $m\times m$ matrix of the form
\begin{equation*}
A=\left[\frac{1}{x_{i}-y_{j}}\right], \quad x_{i}-y_{j} \neq 0 ; \quad 1 \leq i, j \leq m,
\end{equation*}
then its determinant is
\begin{equation*}\label{det A}
\operatorname{det}(A)=\frac{\prod_{i<j}\left(x_{i}-x_{j}\right)\left(y_{j}-y_{i}\right)}{\prod_{i, j=1}^{m}\left(x_{i}-y_{j}\right)},
\end{equation*}
hence we conclude that the Cauchy matrix is invertible. Using Cramer's rule, we obtain an explicit formula for its inverse
\begin{equation}\label{def B}
B=\left[b_{i j}\right]=A^{-1}, \quad b_{i j}=\left(x_{j}-y_{i}\right) A_{j}\left(y_{i}\right) B_{i}\left(x_{j}\right),
\end{equation}
where $A_{i}(x)$ and $B_{j}(y)$ are the Lagrange polynomials for $x_{i}, y_{j}$ respectively, i.e.,
\begin{align}\label{def Ai}
A_{i}(x)&=\frac{\prod_{\ell \neq i}\left(x-x_{\ell}\right)}{\prod_{\ell \neq i}\left(x_{i}-x_{\ell}\right)}=\prod_{1 \leq \ell \leq m, \ell\neq i} \frac{x-x_{\ell}}{x_{i}-x_{\ell}},
\end{align}
\begin{align}\label{def Bj}
B_{j}(y) &=\prod_{1 \leq \ell\leq m,\ell\neq j} \frac{y-y_{\ell}}{y_{j}-y_{\ell}}.
\end{align}

Now we derive the explicit formula for the projection of initial data $(f,g)$ in space $\tilde{W}\times W$.

{\bf For initial data $(0,g)$}, we choose the space $V:=L^{2}(r\geq R,r^{N-1}dr)$, and the generalized kernel space $W$ of $\mathcal{L}_a$ in $V$ is
\begin{equation}\label{eq:span g}
W=\operatorname{span}\left\{r^{\alpha-2+2i}\bigg|i=1,2,\cdot\cdot\cdot,\tilde{k}_1=\Big\lfloor\frac{(N-2)\beta+2}{4}\Big\rfloor;r\geq R\right\}.
\end{equation}
Denote the Cauchy matrix $A=[\langle a_i,a_j\rangle]_{\tilde{k}_1\times\tilde{k}_1}:=(a_{ij})_{\tilde{k}_1\times\tilde{k}_1}$, where
\begin{align*}
a_{i j}(R)&\triangleq\left\langle r^{\alpha-2+2i},r^{\alpha-2+2j} \right\rangle_{V}=\int_{R}^{\infty}r^{2\alpha-4+2i+2j}r^{N-1}\ud r\nonumber\\
&=\frac{R^{-(N-2)\beta+2i+2j-2}}{(N-2)\beta-2i-2j+2}.
\end{align*}
In particular, when $R=1$, the Cauchy matrix is
\begin{equation*}
A(1)=\left[\frac{1}{(N-2)\beta-2i-2j+2}\right]_{\tilde{k}_1\times\tilde{k}_1}:=\left[\frac{1}{x_i-y_j}\right]_{\tilde{k}_1\times\tilde{k}_1},
\end{equation*}
here $x_i=(N-2)\beta-2i+2$, $y_j=2j$. By the inverse matrix formula \eqref{def B}, \eqref{def Ai} and \eqref{def Bj} we conclude
\begin{align}\label{bij for g}
b_{i j}(1)=&(x_j-y_i)A_{j}(y_i)B_{i}(x_j)\nonumber\\
=&\big((N-2)\beta-2i-2j+2\big)\prod_{1\leq \ell\leq\tilde{k}_1,\ell\neq j}\frac{2i-(N-2)\beta+2\ell-2}{2\ell-2j}\nonumber\\
&\times\prod_{1\leq \ell\leq\tilde{k}_1,\ell\neq i}\frac{(N-2)\beta-2i-2\ell+2}{2\ell-2i}\nonumber\\
=&\frac{\prod_{1\leq \ell\leq\tilde{k}_1}\big(2i+2\ell-2-(N-2)\beta\big)\prod_{1\leq \ell\leq\tilde{k}_1}\big(2j+2\ell-2-(N-2)\beta\big)}{\big((N-2)\beta-2i-2j+2\big)\prod_{1\leq \ell\leq\tilde{k}_1,\ell\neq j}(2\ell-2j)\prod_{1\leq \ell\leq\tilde{k}_1,\ell\neq i}(2\ell-2i)}.
\end{align}
Therefore we obtain the inverse Cauchy matrix
\begin{equation*}
B(R)=A(R)^{-1}=\langle b_{i j}(R)\rangle=b_{i j}(1)R^{(N-2)\beta-2i-2j+2}.
\end{equation*}

From \eqref{proj u norm} and \eqref{bij for g}, the projection formula is given by
\begin{align}\label{proj g}
&\quad\|\operatorname{\pi}_{W^{\perp}}g\|_{L^2(r\geq R,r^{N-1}\ud r)}\nonumber\\
&=\int_{R}^{\infty}|g(r)|^2r^{N-1}\ud r-\sum\limits_{i,j=1}^{\tilde{k}_1}\frac{R^{(N-2)\beta-2i-2j+2}}{(N-2)\beta-2i-2j+2}\tilde{c}_i\tilde{c}_j\int_{R}^{\infty}g(r)r^{\alpha+2i-2}r^{N-1}\ud r\nonumber\\
&\qquad\qquad\qquad\qquad\qquad\qquad\qquad\qquad\times\int_{R}^{\infty}\overline{g(r)}r^{\alpha+2j-2}r^{N-1}\ud r\nonumber\\
&=\int_{R}^{\infty}|g(r)|^2r^{N-1}\ud r-\sum\limits_{i,j=1}^{\tilde{k}_1}\frac{R^{-2c-2i-2j+N}}
{-2c-2i-2j+N}\tilde{c}_i\tilde{c}_j\int_{R}^{\infty}g(r)r^{c+2i-1}\ud r\int_{R}^{\infty}\overline{g(r)}r^{c+2j-1}\ud r,
\end{align}
with
\begin{equation}\label{eq:proj g}
\tilde{c}_{j}=\frac{\prod_{1\leq \ell\leq\tilde{k}_1}\big((N-2)\beta-2\ell-2j+2\big)}{\prod_{1\leq \ell\leq\tilde{k}_1,\ell\neq j}(2\ell-2j)}, \qquad \ell\leq j\leq \tilde{k}_1=\Big\lfloor\frac{(N-2)\beta+2}{4}\Big\rfloor.
\end{equation}

{\bf For initial data $(f,0)$.} Let $\tilde{V}=\dot{H}^{1}_{a}(r\geq R,r^{N-1}\ud r)$, and the corresponding generalized kernel space be
\begin{equation}\label{eq:span f}
\tilde{W}=\operatorname{span}\left\{r^{\alpha-2+2i}\bigg|i=1,2,\cdot\cdot\cdot,\tilde{k}_2=\Big\lfloor\frac{(N-2)\beta+4}{4}\Big\rfloor;r\geq R\right\}.
\end{equation}
The same computation as before yields that the elements in the matrix $A(R)$ are
\begin{align*}
\tilde{a}_{i j}(R)=\left(\alpha-2+2i\right)\left(\alpha-2+2j\right)\frac{R^{-(N-2)\beta+2i+2j-4}}{(N-2)\beta-2i-2j+4}.
\end{align*}
Direct computation shows the inverse of the Cauchy matrix $[\alpha_{i j}]:=[\frac{1}{(N-2)\beta-2i-2j+4}]_{\tilde{k}_2\times\tilde{k}_2}$. By \eqref{def B}, \eqref{def Ai} and \eqref{def Bj}, the elements $\beta_{i j}$ of $B_j$ have the form
\begin{align*}
\beta_{i j}=&\frac{1}{(N-2)\beta-2i-2j+4}\\
&\times\frac{\prod_{1\leq \ell\leq\tilde{k}_2}\big((N-2)\beta-2i-2\ell+4\big)\prod_{1\leq \ell\leq\tilde{k}_2}\big(N-2)\beta-2j-2\ell+4\big)}{\prod_{1\leq \ell\leq\tilde{k}_2,\ell\neq j}(2\ell-2j)\prod_{1\leq \ell\leq\tilde{k}_2,\ell\neq i}(2\ell-2i)}.
\end{align*}
Thus the inverse matrix $\tilde{B}(R)=[b_{i j}(R)]_{\tilde{k}_{2}\times\tilde{k}_2}$ can be written as
\begin{align}\label{def: tilde bij}
\tilde{b}_{i j}(R)=&\frac{1}{(N-2)\beta-2i-2j+4}\frac{R^{(N-2)\beta-2i-2j+4}}{(\alpha-2+2i)(\alpha-2+2j)}\nonumber\\
&\times\frac{\prod_{1\leq \ell\leq\tilde{k}_2}\big((N-2)\beta-2i-2\ell+4\big)\prod_{1\leq \ell\leq\tilde{k}_2}\big(N-2)\beta-2j-2\ell+4\big)}{\prod_{1\leq \ell\leq\tilde{k}_2,\ell\neq j}(2\ell-2j)\prod_{1\leq \ell\leq\tilde{k}_2,\ell\neq i}(2\ell-2i)}.
\end{align}
Applying \eqref{proj u norm} with \eqref{def: tilde bij}, we obtain the projection formula
\begin{align}\label{projection of f}
&\quad\|\operatorname{\pi}_{\tilde{W}^{\perp}}f\|^{2}_{\dot{H}^1_a(r\geq R,r^{N-1}\ud r)}\nonumber\\
=&\int_{R}^{\infty}\left|(\partial_{r}+\frac{c}{r})f(r)\right|^2r^{N-1}\ud r -\sum_{i,j=1}^{\tilde{k}_2}\tilde{d}_i\tilde{d}_j\frac{R^{(N-2)\beta-2i-2j+4}}{(N-2)\beta-2i-2j+4}\nonumber\\
&\times\int_{R}^{\infty}(\partial_{r}+\frac{c}{r})f(r)r^{\alpha+2i-3}r^{N-1}\ud r\int_{R}^{\infty}(\partial_{r}+\frac{c}{r})\overline{f(r)}r^{\alpha+2j-3}r^{N-1}\ud r\nonumber\\
=&\int_{R}^{\infty}\left|(\partial_{r}+\frac{c}{r})f(r)\right|^2r^{N-1}\ud r -\sum_{i,j=1}^{\tilde{k}_2}\tilde{d}_i\tilde{d}_j\frac{R^{-2c-2i-2j+N+2}}{-2c-2i-2j+N+2}\nonumber\\
&\times\int_{R}^{\infty}(\partial_{r}+\frac{c}{r})f(r)r^{c+2i-2}\ud r\int_{R}^{\infty}(\partial_{r}+\frac{c}{r})\overline{f(r)}r^{c+2j-2}\ud r\nonumber\\
=&\int_{R}^{\infty}\left|(\partial_{r}+\frac{c}{r})f(r)\right|^2r^{N-1}\ud r -\sum_{i,j=1}^{\tilde{k}_2}\tilde{d}_i\tilde{d}_j\frac{R^{-2c-2i-2j+N+2}}{-2c-2i-2j+N+2}\nonumber\\
&\times\int_{R}^{\infty}\partial_{r}(f(r)r^c)r^{2i-2}\ud r\int_{R}^{\infty}\partial_{r}(\overline{f(r)r^c})r^{2j-2}\ud r,
\end{align}
where
\begin{equation}\label{eq:proj f}
\tilde{d}_{j}=\frac{\prod_{1\leq \ell\leq\tilde{k}_{2}}((N-2)\beta-2\ell-2j+4)}{\prod_{1\leq \ell\leq\tilde{k}_{2},\ell\neq j}(2\ell-2j)},\qquad 1\leq j\leq\tilde{k}_{2}=\Big\lfloor\frac{(N-2)\beta+4}{4}\Big\rfloor.
\end{equation}

\medskip
Let us recall some useful results about the coefficients $\tilde{c}_j$ and $\tilde{d}_j$.
\begin{lemma}\emph{(\cite{KLLS})}\label{coefficients results}
Given the coefficients $\tilde{c}_j$, $1\leq j\leq \tilde{k}_1=\lfloor\frac{(N-2)\beta+2}{4}\rfloor$ and $\tilde{d}_j$, $1\leq j\leq \tilde{k}_2=\lfloor\frac{(N-2)\beta+4}{4}\rfloor$ defined as in \eqref{eq:proj g} and \eqref{eq:proj f}, we have the following identities
\begin{align*}
\sum\limits_{j=1}^{\tilde{k}_1}\frac{\tilde{c}_j}{(N-2)\beta-2m-2j+2}=1, \quad for\ any\ m\in\mathbb{Z}, 1\leq m\leq\tilde{k}_1,
\end{align*}
\begin{align}\label{prod cj}
\sum\limits_{j=1}^{\tilde{k}_1}\frac{\tilde{c}_j}{2j}+1=\prod\limits_{\ell=1}^{\tilde{k}_1}\frac{(N-2)\beta-2\ell+2}{2\ell}.
\end{align}

Similarly, for coefficients $\tilde{d}_j$ we have
\begin{align}\label{sum dj}
\sum\limits_{j=1}^{\tilde{k}_2}\frac{\tilde{d}_j}{(N-2)\beta-2m-2j+4}=1, \quad for\ any\ m\in\mathbb{Z}, 1\leq m\leq\tilde{k}_2,
\end{align}
\begin{align*}
\sum\limits_{j=1}^{\tilde{k}_2}\frac{\tilde{d}_j}{2j}+1=\prod\limits_{\ell=1}^{\tilde{k}_2}\frac{(N-2)\beta-2\ell+4}{2\ell}.
\end{align*}
\end{lemma}

The following lemma in \cite{KLLS} collects various limits which we will use repeatedly in the calculation of $AS_{N}^{a}g$ and $AS_{N}^{a}f$. These limits all involve regularization by the mollifier $\kappa_{\varepsilon}(x)$ which we introduce in Lemma~\ref{asf,asg}. In the lemma we denote by $\delta_{ab}$ the Kronecker delta
\begin{equation*}
\delta_{ab}=\begin{cases}1, & \text { when } a=b, \\ 0, & \text { when } a \neq b.\end{cases}
\end{equation*}
In contrast, $\delta_{y}$ is the standard Dirac measure centered at $y \in \mathbb{R} .$ In other words, $\left(f * \delta_{y}\right)(x)=f(x-y)$.

\begin{lemma}\label{le:limits}
Let $i$, $j\geq 1$, $r_{1}$, $r_{2}>0$. For $a$, $b\in\{0,1\}$, we have the following limits for any test function $\phi\left(r_{1}, r_{2}\right)\in C_{0}^{\infty}\left(\mathbb{R}^{+}\times\mathbb{R}^{+}\right)$

\begin{align}\label{kappa xi kappa xi}
&\lim_{\varepsilon\rightarrow 0}\iint\limits_{r_{1},r_{2}>0}\left[\int_{-R}^{R}\left[\kappa_{\varepsilon}\ast\xi^{i} \chi_{(-1,1)}\right]\Big(\frac{\xi}{r_{1}}\Big)\left[\kappa_{\varepsilon}\ast\xi^{j}\chi_{(-1,1)}\right]\Big(\frac{\xi}{r_{2}}\Big)\ud \xi\right]\phi\left(r_{1}, r_{2}\right) \ud r_{1} \ud r_{2}\nonumber\\
&\quad=\left(1-(-1)^{i+j+1}\right)\iint\limits_{r_{1}, r_{2}>0} \frac{1}{i+j+1} \frac{\min \left(r_{1}, r_{2}, R\right)^{i+j+1}}{r_{1}^{i} r_{2}^{j}} \phi\left(r_{1}, r_{2}\right) \ud r_{1} \ud r_{2},
\end{align}

\begin{align}\label{delta,delta}
&\lim_{\varepsilon\rightarrow 0}\iint\limits_{r_{1},r_{2}>0}\left[\int_{-R}^{R}\left[\kappa_{\varepsilon}\ast \delta_{(-1)^{a}}\right]\Big(\frac{\xi}{r_{1}}\Big)\left[\kappa_{\varepsilon}\ast\delta_{(-1)^{b}}\right]\Big(\frac{\xi}{r_{2}}\Big)\ud \xi\right]\phi\left(r_{1},r_{2}\right)\ud r_{1}\ud r_{2}\nonumber\\
&\quad=\delta_{ab}\int_{0}^{R}\xi^{2}\phi(\xi,\xi)\ud \xi,
\end{align}

\begin{align}\label{delta',delta'}
&\lim_{\varepsilon\rightarrow 0}\iint\limits_{r_{1},r_{2}>0}\left[\int_{-R}^{R}\left[\kappa_{\varepsilon}\ast \delta_{(-1)^{a}}^{\prime}\right]\Big(\frac{\xi}{r_{1}}\Big)\left[\kappa_{\varepsilon}\ast \delta_{(-1)^{b}}^{\prime}\right]\Big(\frac{\xi}{r_{2}}\Big)\ud \xi\right]\phi\left(r_{1},r_{2}\right)\ud r_{1}\ud r_{2}\nonumber\\
&\quad=\bigg.\delta_{ab}\int_{0}^{R}\left[\partial_{x}\partial_{y}\left(x^{2}y^{2}\phi(x,y)\right)\right]\bigg|_{x=y=\xi} d \xi,
\end{align}

\begin{align}\label{xi,delta}
&\lim_{\varepsilon\rightarrow 0}\iint\limits_{r_{1},r_{2}>0}\left[\int_{-R}^{R}\left[\kappa_{\varepsilon}\ast\xi^{i} \chi_{(-1,1)}\right]\Big(\frac{\xi}{r_{1}}\Big)\left[\kappa_{\varepsilon}\ast\delta_{(-1)^{a}}\right]\Big(\frac{\xi}{r_{2}}\Big)\ud \xi\right] \phi\left(r_{1},r_{2}\right)\ud r_{1}\ud r_{2}\nonumber\\
&\quad=\int_{0}^{\infty}\int_{0}^{\min\left(r_{1},R\right)}r_{2}\left(\frac{(-1)^{a}r_{2}}{r_{1}}\right)^{i}\phi\left(r_{1},r_{2}\right)\ud r_{2}\ud r_{1},
\end{align}

\begin{align}\label{xi,delta'}
&\lim_{\varepsilon\rightarrow 0}\iint\limits_{r_{1},r_{2}>0}\left[\int_{-R}^{R}\left[\kappa_{\varepsilon}\ast\xi^{i} \chi(-1,1)\right]\Big(\frac{\xi}{r_{1}}\Big)\left[\kappa_{\varepsilon}\ast\delta_{(-1)^{a}}^{\prime}\right]\Big(\frac{\xi}{r_{2}}\Big)\ud \xi\right]\phi\left(r_{1},r_{2}\right)\ud r_{1}\ud r_{2}\nonumber\\
&\quad=(-1)^{a}\int_{0}^{\infty}\int_{0}^{\min\left(r_{1},R\right)}\left(\frac{(-1)^{a}r_{2}}{r_{1}}\right)^{i}\partial_{r_{2}}\left(r_{2}^{2} \phi\left(r_{1},r_{2}\right)\right)\ud r_{1}\ud r_{2},
\end{align}

\begin{align}\label{delta,delta'}
&\lim_{\varepsilon\rightarrow 0}\iint\limits_{r_{1},r_{2}>0}\left[\int_{-R}^{R}\left[\kappa_{\varepsilon}\ast \delta_{(-1)^{a}}\right]\Big(\frac{\xi}{r_{1}}\Big)\left[\kappa_{\varepsilon}\ast\delta_{(-1)^{b}}^{\prime}\right]\Big(\frac{\xi}{r_{2}}\Big)\ud \xi\right]\phi\left(r_{1},r_{2}\right)\ud r_{1}\ud r_{2}\nonumber\\
&\quad=\bigg.(-1)^{a}\delta_{ab}\int_{0}^{R}\left[\partial_{y}\left(x y^{2}\phi(x,y)\right)\right]\Big|_{x=y=\xi}\ud \xi.
\end{align}
\end{lemma}

\subsection{Proof of Theorem \ref{thm:weq} for $(0,g)$ data}\label{SS:2.4}
By the formula of $AS_{N}^{a}g$ in \eqref{eq:ASg}, the Fourier expansion formula of $\varphi_{k-1}(x)$ needs to be used. Since the parity of $\varphi_j(x)$ is crucial, we have to discuss separately the two cases according to the order of the function.
\begin{lemma}\label{lem:fourier g}
Let $k_1\geq 1$. With $k-1=2k_1$, and $k-1$ is even, we have
\begin{align}\label{even varphi}
\mathcal{F}\varphi_{k-1}=&\mathcal{F}\varphi_{2k_1}\nonumber\\
=&\pi i(-1)^{k_1}\bigg[\sum\limits_{j=1}^{k_1}\frac{\prod_{\ell=1}^{k_1}(4k_1-2\ell-2j+3)}{\prod_{l\leq \ell\leq k,\ell\neq j}(2\ell-2j)}\xi^{2k_1-2j+1}\chi_{(-1,1)}(\xi)+(\delta_1+\delta_{-1})\bigg],
\end{align}
we denote
\begin{align*}
c_j=\frac{\prod_{\ell=1}^{k_1}(4k_1-2j-2\ell+3)}{\prod_{1\leq \ell\leq k_1,\ell\neq j}(2\ell-2j)}.
\end{align*}
With $k-1=2k_1-1$ odd
\begin{align}\label{odd varphi}
\mathcal{F}\varphi_{k-1}=&\mathcal{F}\varphi_{2k_1-1}\nonumber\\
=&\pi(-1)^{k_1-1}\bigg[\sum\limits_{j=1}^{k_1}\frac{\prod_{\ell=1}^{k_1}(4k_1-2j-2\ell+1)}{\prod_{1\leq \ell\leq k_1,\ell\neq j}(2\ell-2j)}\xi^{2k_1-2j}\chi_{(-1,1)}(\xi)+(-\delta_{1}-\delta_{-1})\bigg],
\end{align}
we also denote
\begin{align*}
c_{j}=\frac{\prod_{\ell=1}^{k_1}(4k_1-2j-2\ell+1)}{\prod_{1\leq \ell\leq k_1,\ell\neq j}(2\ell-2j)}.
\end{align*}
\end{lemma}

\medskip
Assume $r_1,r_2>R$ as before, since we consider initial data $(0,g)$, without loss of generality, assume that $g(r)=0$ for $r\leq R$. Hence by Lemma~\ref{le:limits}, the terms which integral region contain $(0,R)$ is equal zero and the terms created by the $\delta$-function will make no contribution to our integral.

If $k-1\in\mathbb{Z}$ is even, then there exists $k_1\in\mathbb{Z}$ such that $k-1=2k_1$. By definition of $c=\frac{(N-2)(1-\beta)}{2}$, then $2k_1=-c+\frac{N-3}{2}$. Using \eqref{eq:ASg}, \eqref{even varphi} and \eqref{kappa xi kappa xi}, we obtain
\begin{align*}
\quad&2AS_{N}^{a}g\nonumber\\
=&\int_{R}^{\infty}|g(r)|^2r^{N-1}\ud r-\frac{1}{2\pi^2}\lim\limits_{\varepsilon\rightarrow0}\iint\int_{-R}^{R}\big(\kappa_{\varepsilon}\ast\mathcal{F}\varphi_{k-1}\big)\big(\frac{\xi}{r_1}\big)
\overline{\big(\kappa_{\varepsilon}\ast\mathcal{F}\varphi_{k-1}\big)\big(\frac{\xi}{r_2}\big)}\ud \xi\nonumber\\
&\times g(r_1)\overline{g(r_2)}(r_1r_2)^{\frac{N-3}{2}}\ud r_1\ud r_2\nonumber\\
=&\int_{R}^{\infty}|g(r)|^2r^{N-1}\ud r-\frac{1}{2}\sum\limits_{i,j=1}^{k_1}\lim\limits_{\varepsilon\rightarrow0}\iint\int_{-R}^{R}\left[\kappa_{\varepsilon}\ast\xi^{2k_1-2i+1}\chi_{(-1,1)}\right]\big(\frac{\xi}{r_1}\big)\nonumber\\
&\times \left[\kappa_{\varepsilon}\ast\xi^{2k_1-2j+1}\chi_{(-1,1)}\right]\big(\frac{\xi}{r_2}\big)\ud\xi~ g(r_1)\overline{g(r_2)}(r_1r_2)^{\frac{N-3}{2}}\ud r_1\ud r_2
\end{align*}
\begin{align*}
=&\int_{R}^{\infty}|g(r)|^2r^{N-1}\ud r-\sum\limits_{i,j=1}^{k_1}\frac{c_i c_j R^{4k_1-2i-2j+3}}{4k_1-2i-2j+3}\int_{R}^{\infty}g(r)r^{-2k_1+2i+\frac{N-5}{2}}\ud r\nonumber\\
&\times\int_{R}^{\infty}\overline{g(r)}r^{-2k_1+2j+\frac{N-5}{2}}\ud r\nonumber\\
=&\int_{R}^{\infty}|g(r)|^2r^{N-1}\ud r-\sum\limits_{i,j=1}^{k_1}\frac{c_i c_j R^{-2c-2i-2j+N}}{-2c-2i-2j+N}\int_{R}^{\infty}g(r)r^{c+2i-1}\ud r\int_{R}^{\infty}\overline{g(r)}r^{c+2j-1}\ud r.
\end{align*}
Due to the definition \eqref{eq:span g} of the space $W$, the dimension is $\tilde{k}_2=\lfloor\frac{(N-2)\beta+2}{4}\rfloor=\lfloor k_1+\frac{3}{4}\rfloor=k_1$, and $4k_1+3-2i-2j=(N-2)\beta+2-2i-2j$. Compared with \eqref{proj g}, we obtain
\begin{equation}\label{even ASg}
AS_{N}^{a}(g)=\frac{1}{2}\big\|\operatorname{\pi}_{W^{\perp}}g\big\|_{L^2(r\geq R,r^{N-1}\ud r)}^2,   \qquad(N-2)\beta=4k_1+1.
\end{equation}

\medskip
Similarly, if $k-1$ is odd, then there exists $k_1\in\mathbb{Z}$ such that $k-1=2k_1-1$, then $k=2k_1=-c+\frac{N-1}{2}$. We thus see from \eqref{odd varphi} that
\begin{align*}
\quad&2AS_{N}^{a}g\nonumber\\
=&\int_{R}^{\infty}|g(r)|^2r^{N-1}\ud r-\frac{1}{2\pi^2}\lim\limits_{\varepsilon\rightarrow0}\iint\int_{-R}^{R}\big(\kappa_{\varepsilon}\ast\mathcal{F}\varphi_{2k_1-1}\big)\big(\frac{\xi}{r_1}\big)
\overline{\big(\kappa_{\varepsilon}\ast\mathcal{F}\varphi_{2k_1-1}\big)\big(\frac{\xi}{r_2}\big)}\ud \xi\nonumber\\
&\times g(r_1)\overline{g(r_2)}(r_1r_2)^{\frac{N-3}{2}}\ud r_1\ud r_2\nonumber\\
=&\int_{R}^{\infty}|g(r)|^2r^{N-1}\ud r-\frac{1}{2}\sum\limits_{i,j=1}^{k_1}c_i c_j\lim\limits_{\varepsilon\rightarrow0}\iint\int_{-R}^{R}\left[\kappa_{\varepsilon}\ast\xi^{2k_1-2i}\chi_{(-1,1)}\right]\big(\frac{\xi}{r_1}\big)\nonumber\\
&\times\left[\kappa_{\varepsilon}\ast\xi^{2k_1-2j}\chi_{(-1,1)}\right]\big(\frac{\xi}{r_2}\big)\ud\xi g(r_1)\overline{g(r_2)}(r_1r_2)^{\frac{N-3}{2}}\ud r_1\ud r_2\nonumber\\
=&\int_{R}^{\infty}|g(r)|^2r^{N-1}\ud r-\sum\limits_{i,j=1}^{k_1}\frac{c_i c_j R^{4k_1-2i-2j+1}}{4k_1-2i-2j+1}\int_{R}^{\infty}g(r)r^{-2k_1+2i+\frac{N-3}{2}}\ud r\nonumber\\
&\times\int_{R}^{\infty}\overline{g(r)}r^{-2k_1+2j+\frac{N-3}{2}}\ud r\nonumber\\
=&\int\limits_{R}^{\infty}|g(r)|^2r^{N-1}\ud r-\sum\limits_{i,j=1}^{k_1}\frac{c_i c_j R^{-2c-2i-2j+N}}{-2c-2i-2j+N}\int_{R}^{\infty}g(r)r^{c+2i-1}\ud r\int_{R}^{\infty}\overline{g(r)}r^{c+2j-1}\ud r.
\end{align*}

Compared with \eqref{proj g}, we have proved
\begin{align}\label{odd ASg}
AS_{N}^{a}(g)=\frac{1}{2}\big\|\operatorname{\pi}_{W^{\perp}}g\big\|_{L^2(r\geq R,r^{N-1}\ud r)}^2,   \qquad(N-2)\beta=4k_1-1.
\end{align}

For each $k\in\mathbb{Z}$, together \eqref{even ASg} with \eqref{odd ASg}, we conclude the estimate \eqref{ineq: channel of energy} of Theorem~\ref{ce} for $(0,g)$.

\subsection{Proof of Theorem~\ref{thm:weq} for $(f,0)$ data}\label{SS:2.5}
Indeed, if we obtain the Fourier transform of $\psi_{k-1}$, we can calculate $AS_{N}^{a}f(r)$, and use the projection formula \eqref{projection of f} to show the estimate \eqref{ineq: channel of energy} for $(f,0)$. As before, we also divide into two situations.
\begin{lemma}\label{Fourier of psi}
Let $k_1\geq 2$. With $k-1=2k_1-2$ and $k-1$ is even, we have
\begin{align}\label{even Fourier psi}
\mathcal{F}\psi_{k-1}(\xi)=\mathcal{F}\psi_{2k_1-2}(\xi)=&\pi(-1)^{k_1}\bigg[\sum\limits_{j=1}^{k_1-1}d_j(2k_1-2j-1)\xi^{2k_1-2j-2}\chi_{(-1,1)}(\xi)\nonumber\\
&+\sum\limits_{j=1}^{k_1-1}d_j(-\delta_{1}-\delta_{-1})+(-\delta'_{1}+\delta'_{-1})\bigg],
\end{align}
where
\begin{align*}
d_j=\frac{\prod_{\ell=1}^{k_1-1}(4k_1-2j-2\ell-1)}{\prod_{1\leq \ell\leq k_1-1,\ell\neq j}(2\ell-2j)},\qquad 1\leq j\leq k_1-1.
\end{align*}
For $k-1=2k_1-1$ we have
\begin{align}\label{odd Fourier psi}
\mathcal{F}\psi_{k-1}(\xi)=\mathcal{F}\psi_{2k_1-1}(\xi)=&\pi(-1)^{k_1-1}i\bigg[\sum\limits_{j=1}^{k_1}d_j(2k_1-2j)\xi^{2k_1-2j-1}\chi_{(-1,1)}(\xi)\nonumber\\
&+\sum\limits_{j=1}^{k_1}d_j(-\delta_{1}+\delta_{-1})+(-\delta'_{1}-\delta'_{-1})\bigg],
\end{align}
where
\begin{align*}
d_j=\frac{\prod_{\ell=1}^{k_1}(4k_1-2j-2\ell+1)}{\prod_{1\leq \ell\leq k_1,\ell\neq j}(2\ell-2j)},\qquad 1\leq j\leq k_1.
\end{align*}
\end{lemma}
\subsubsection{\bf{Case $k-1=2k_1-2$}}
In this case, the dimension of the space $\tilde{W}$ is $\tilde{k}_1=\lfloor\frac{-2c+N+2}{4}\rfloor=\lfloor k_1+\frac{1}{4}\rfloor=k_1$. We also introduce the following notation to simplify the formula when we have the symmetric structure
$$\big\langle S(i,j)\big\rangle=S(i,j)+\overline{S(j,i)}.$$

Substituting the Fourier expansion of $\psi_{k-1}$ \eqref{even Fourier psi} into \eqref{eq:ASf}, combining Lemma~\ref{le:limits} \eqref{kappa xi kappa xi} to \eqref{delta,delta'} with the notation $F(r_1,r_2):=f(r_1)r_1^c\overline{f(r_2)r_2^c}(r_1r_2)^{-c+\frac{N-5}{2}}$, it holds
\begin{align*}
\quad&2AS_{N}^{a}f(r)\\
=&\int_0^{\infty}\Big|(\partial_{r}+\frac{c}{r})f(r)\Big|^{2}r^{N-1}\ud r-\lim\limits_{\varepsilon\rightarrow0}\iint\limits_{r_1,r_2>0}\bigg[\frac{1}{2\pi^2r_1^2r_2^2}\int_{-R}^{R}\big[\kappa_{\varepsilon}\ast\mathcal{F}\psi_{2k_1-2}\big]\big(\frac{\xi}{r_1}\big)\\
&\times\overline{\big[\kappa_{\varepsilon}\ast\mathcal{F}\psi_{2k_1-2}\big]\big(\frac{\xi}{r_2}\big)}d\xi\bigg]f(r_1)\overline{f(r_2)}(r_1r_2)^{\frac{N-1}{2}}\ud r_1\ud r_2
\end{align*}
\begin{align*}
=&\int_{0}^{\infty}\Big|(\partial_{r}+\frac{c}{r})f(r)\Big|^{2}r^{N-1}\ud r-\frac{1}{2}\sum\limits_{i,j=1}^{k_1-1}d_{i}d_{j}(2k_1-2i-1)(2k_1-2j-1)\\
&\times\lim\limits_{\varepsilon\rightarrow0}\iint\bigg[\int_{-R}^{R}\left[\kappa_{\varepsilon}\ast\xi^{2k_1-2i-2}\chi_{(-1,1)}\right]\big(\frac{\xi}{r_1}\big)\left[\kappa_{\varepsilon}\ast\xi^{2k_1-2j-2}\chi_{(-1,1)}\right]\big(\frac{\xi}{r_2}\big)\ud \xi\bigg]F(r_1,r_2)\ud r_1\ud r_2\\
&-\frac{1}{2}\sum\limits_{i,j=1}^{k_1-1}d_id_j\lim\limits_{\varepsilon\rightarrow0}\iint\bigg[\int_{-R}^{R}\big[\kappa_{\varepsilon}\ast(\delta_{1}+\delta_{-1})\big]\big(\frac{\xi}{r_1}\big)\big[\kappa_{\varepsilon}\ast(\delta_{1}+\delta_{-1})\big]\big(\frac{\xi}{r_2}\big)\bigg]F(r_1,r_2)\ud r_1\ud r_2\\
&-\frac{1}{2}\lim\limits_{\varepsilon\rightarrow0}\iint\bigg[\int_{-R}^{R}\big[\kappa_{\varepsilon}\ast(-\delta'_{1}+\delta'_{-1})\big]\big(\frac{\xi}{r_1}\big)\big[\kappa_{\varepsilon}\ast(-\delta'_{1}+\delta'_{-1})\big]\big(\frac{\xi}{r_2}\big)\ud\xi\bigg]F(r_1,r_2)\ud r_1\ud r_2\\
&+\frac{1}{2}\bigg\langle\sum_{i,j=1}^{k_1-1}(2k_1-2i-1)d_id_j\lim\limits_{\varepsilon\rightarrow0}\iint\int_{-R}^{R}\left[\kappa_{\varepsilon}\ast\xi^{2k_1-2i-2}\chi_{(-1,1)}\right]\big(\frac{\xi}{r_1}\big)\\
&\times\left[\kappa_{\varepsilon}\ast(\delta_{1}+\delta_{-1})\right]\big(\frac{\xi}{r_2}\big)\ud \xi F(r_1,r_2)\ud r_1\ud r_2\bigg\rangle\\
&-\frac{1}{2}\bigg\langle\sum\limits_{i=1}^{k_1-1}(2k_1-2i-1)d_i\lim\limits_{\varepsilon\rightarrow0}\iint\int_{-R}^{R}\left[\kappa_{\varepsilon}\ast\xi^{2k_1-2i-2}\chi_{(-1,1)}\right]\big(\frac{\xi}{r_1}\big)\big[\kappa_{\varepsilon}\ast(\delta'_{1}-\delta'_{-1})\big]\big(\frac{\xi}{r_2}\big)\ud \xi\\
&\times F(r_1,r_2)\ud r_1\ud r_2\bigg\rangle\\
&-\frac{1}{2}\bigg\langle\sum\limits_{i=1}^{k_1-1}d_i\lim\limits_{\varepsilon\rightarrow0}\iint\bigg[\int_{-R}^{R}\big[\kappa_{\varepsilon}\ast(\delta_{1}+\delta_{-1})\big]\big(\frac{\xi}{r_1}\big)\big[\kappa_{\varepsilon}\ast(\delta'_{1}-\delta'_{-1})\big]\big(\frac{\xi}{r_2}\big)\ud \xi\bigg]F(r_1,r_2)\ud r_1\ud r_2\bigg\rangle\\
=&\int_{0}^{\infty}\Big|(\partial_{r}+\frac{c}{r})f(r)\Big|^2r^{N-1}\ud r-\sum\limits_{i,j=1}^{k_1-1}d_id_j\iint\limits_{r_1,r_2>0}\frac{(2k_1-2i-1)(2k_1-2j-1)}{4k_1-2i-2j-3}\\
&\times\frac{\min(r_1,r_2,R)^{4k_1-2i-2j-3}}{r_1^{2k_1-2i-2}r_2^{2k_1-2j-2}}f(r_1)r_1^{c}\overline{f(r_2)r_2^c}(r_1r_2)^{-c+\frac{N-5}{2}}\ud r_1\ud r_2\\
&-\sum\limits_{i,j=1}^{k_1-1}d_id_j\int_{0}^{R}|f(r)r^c|^2r^{-2c+N-3}\ud r-\int_{0}^{R}\left|\partial_{r}\big(f(r)r^c\cdot r^{-c+\frac{N-1}{2}}\big)\right|^2\ud r\\
&+\bigg\langle\sum\limits_{i,j=1}^{k_1-1}d_id_j(2k_1-2i-1)\int_{0}^{\infty}\int_{0}^{\min(r_1,R)}r_2\big(\frac{r_2}{r_1}\big)^{2k_1-2i-2}f(r_1)r_1^{c}\overline{f(r_2)r_2^{c}}(r_1r_2)^{-c+\frac{N-5}{2}}\ud r_2\ud r_1\bigg\rangle\\
&-\bigg\langle-\sum\limits_{i=1}^{k_1-1}d_i(2k_1-2i-1)\int_{0}^{\infty}\int_{0}^{\min{(r_1,R)}}\big(\frac{r_2}{r_1}\big)^{2k_1-2i-2}f(r_1)r_1^{c}r_1^{-c+\frac{N-5}{2}}\\
&\times\partial_{r_2}\left(\overline{f(r_2)r_2^{c}}r_2^{-c+\frac{N-1}{2}}\right)\ud r_2\ud r_1\bigg\rangle\\
&-\bigg\langle\sum\limits_{i=1}^{k_1-1}d_i\int_{0}^{R}f(r)r^{c}r^{-c+\frac{N-3}{2}}\partial_{r}\left(\overline{f(r)r^c}r^{-c+\frac{N-1}{2}}\right)\ud r\bigg\rangle\\
=&\int_{0}^{\infty}\Big|(\partial_{r}+\frac{c}{r})f(r)\Big|^2r^{N-1}\ud r-(S_1+S_2+S_3+S_4+S_5+S_6).
\end{align*}

Due to the symmetric position of $r_1,r_2$ and the structure of the operator $\mathcal{L}_{a}=A^{\ast} A$, we have
\begin{align*}
0=&\left\langle(-\Delta+\frac{a}{r^2})u,u\right\rangle=\left\langle\Big((-\partial_{r}+\frac{N-1}{r})+\frac{c}{r}\Big)\big(\partial_{r}+\frac{c}{r}\big)u,u\right\rangle\\
=&\left\langle(\partial_{r}+\frac{c}{r})u,(\partial_{r}+\frac{c}{r})u\right\rangle=\bigg\langle\frac{\partial_{r}\big(u(r)r^c\big)}{r^c},\frac{\partial_{r}\big(u(r)r^c\big)}{r^c}\bigg\rangle.
\end{align*}
We can assume $f(r)r^c=f(R)R^c$, for $r\leq R$. Notice that $2k_1=-c+\frac{N+1}{2}$, we compute each $S_i$, respectively. Integration by parts gives
\begin{align*}
S_1=&\sum\limits_{i,j=1}^{k_1-1}\frac{d_id_j(2k_1-2i-1)(2k_1-2j-1)}{(-2c-2i-2j+N+2)}\\
&\times\iint\limits_{r_1,r_2>0}\frac{\min(r_1,r_2,R)^{-2c-2i-2j+N-2}}{r_1^{-c-2i+\frac{N-3}{2}}r_2^{-c-2j+\frac{N-3}{2}}}(f(r_1)r_1^{c})(\overline{f(r_2)r_2^{c}})(r_1r_2)^{-c+\frac{N-5}{2}}\ud r_1\ud r_2\\
\quad=&\bigg\langle\sum\limits_{i,j=1}^{k_1-1}\frac{d_id_j(2k_1-2i-1)(2k_1-2j-1)}{(-2c-2i-2j+N+2)}\\
&\times\iint\limits_{0<r_1<r_2<R}r_1^{-2c-2j+N-3}r_2^{2j-1}(f(r_1)r_1^{c})(\overline{f(r_2)r_2^c})\ud r_1\ud r_2\bigg\rangle
\end{align*}
\begin{align*}
&+\bigg\langle\sum\limits_{i,j=1}^{k_1-1}\frac{d_id_j(2k_1-2i-1)(2k_1-2j-1)}{(-2c-2i-2j+N+2)}\\
&\times\iint\limits_{0<r_1<R<r_2}r_1^{-2c-2j+N-3}r_2^{2j-1}(f(r_1)r_1^{c})(\overline{f(r_2)r_2^c})\ud r_1\ud r_2\bigg\rangle\\
&+\sum\limits_{i,j=1}^{k_1-1}\frac{d_id_j(2k_1-2i-1)(2k_1-2j-1)}{(-2c-2i-2j+N+2)}R^{-2c-2i-2j+N-2}\\
&\times\int_{R}^{\infty}(f(r)r^c)r^{2i-1}\ud r\int_{R}^{\infty}(\overline{f(r)r^c})r^{2j-1}\ud r\\
=&\bigg\langle\sum\limits_{i,j=1}^{k_1-1}d_id_j\frac{(2k_1-2i-1)(2k_1-2j-1)}{(-2c-2i-2j+N-2)(-2c-2j+N-2)}\\
&\times\bigg[\frac{R^{-2c+N-2}}{-2c+N-2}|f(R)R^{c}|^2+R^{-2c-2j+N-2}(f(R)R^c)\int_{R}^{\infty}(\overline{f(r)r^c})r^{2j-1}\ud r\bigg]\bigg\rangle\\
&+\sum\limits_{i,j=1}^{k_1-1}\frac{d_id_j(2k_1-2i-1)(2k_1-2j-1)}{(-2c-2i-2j+N+2)}R^{-2c-2i-2j+N-2}\\
&\times\int_{R}^{\infty}(f(r)r^c)r^{2i-1}\ud r\int_{R}^{\infty}(\overline{f(r)r^c})r^{2j-1}\ud r.
\end{align*}

From the formula,
\begin{equation*}
\int_{R}^{\infty}(f(r)r^c)r^{2j-1}\ud r=-\frac{1}{2j}\bigg[(f(R)R^c)R^{2j}+\int_{R}^{\infty}\partial_{r}(f(r)r^c)r^{2j}\ud r\bigg],
\end{equation*}
we rewrite $S_1$-$S_6$ as
\begin{align}\label{S1}
S_1=&\bigg\langle-\sum\limits_{i,j=1}^{k_1-1}\frac{d_id_j(2k_1-2i-1)(2k_1-2j-1)}{2j(-2c-2i-2j+N-2)}\nonumber\\
&\times\bigg[\frac{R^{-2c+N-2}}{-2c+N-2}\big|f(R)R^c\big|^2+\frac{R^{-2c-2j+N-2}}{-2c-2j+N-2}f(R)R^c\int_{R}^{\infty}\partial_{r}(\overline{f(r)r^c})r^{2j}\ud r\bigg]\bigg\rangle\nonumber\\
&+\bigg\langle\sum\limits_{i,j=1}^{k_1-1}\frac{d_id_j(2k_1-2i-1)(2k_1-2j-1)}{4i j(-2c-2i-2j+N-2)}R^{-2c-2j+N-2}\nonumber\\
&\times f(R)R^{c}\int_{R}^{\infty}\partial_{r}(\overline{f(r)r^c})r^{2j}\ud r\bigg\rangle\nonumber\\
&+\sum\limits_{i,j=1}^{k_1-1}\frac{d_id_j(2k_1-2i-1)(2k_1-2j-1)}{4i j}\frac{R^{-2c-2i-2j+N-2}}{-2c-2i-2j+N-2}\nonumber\\
&\times\int_{R}^{\infty}\partial_{r}(f(r)r^c)r^{2i}\ud r\int_{R}^{\infty}\partial_{r}(\overline{f(r)r^c})r^{2j}\ud r\nonumber\\
&+\sum\limits_{i,j=1}^{k_1-1}\frac{d_id_j(2k_1-2i-1)(2k_1-2j-1)}{4i j}\frac{R^{-2c+N-2}}{-2c-2i-2j+N-2}\big|f(R)R^c\big|^2.
\end{align}

\begin{align}\label{S2}
S_2=&\sum\limits_{i,j=1}^{k_1-1}d_id_j\int_{0}^{R}\big|f(r)r^c\big|^2r^{-2c+N-3}\ud r=\sum\limits_{i,j=1}^{k_1-1}d_id_j\frac{R^{-2c+N-2}}{-2c+N-2}\big|f(R)R^c\big|^2.
\end{align}

\begin{align}\label{S3}
S_3=\int_{0}^{R}\Big|\partial_{r}\big(f(r)r^c r^{-c+\frac{N-1}{2}}\big)\Big|^2\ud r=\Big(-c+\frac{N-1}{2}\Big)^2\frac{R^{-2c+N-2}}{-2c+N-2}\big|f(R)R^c\big|^2.
\end{align}

\begin{align}\label{S4}
S_4=&-\bigg\langle\sum\limits_{i,j=1}^{k_1-1}d_id_j(2k_1-2i-1)\nonumber\\
&\times\int_{0}^{\infty}\int_{0}^{\min(r_1,R)}r_2^{-2c-2i+N-3}r_1^{2i-1}\big(f(r_1)r_1^{c}\big)\big(\overline{f(r_2)r_2^{c}}\big)\ud r_2\ud r_1\bigg\rangle\nonumber\\
=&-\bigg\langle\sum\limits_{i,j=1}^{k_1-1}\frac{d_id_j(2k_1-2i-1)}{-2c-2i+N-2}\Big[\frac{R^{-2c+N-2}}{-2c+N-2}\big|f(R)R^c\big|^2\nonumber\\
&+R^{-2c-2i+N-2}\overline{f(R)R^c}\int_{R}^{\infty}\big(f(r)r^c\big)r^{2i-1}\ud r\Big]\bigg\rangle\nonumber\\
=&\bigg\langle\sum\limits_{i,j=1}^{k_1-1}\frac{d_id_j(2k_1-2i-1)}{2i}\Big[\frac{R^{-2c+N-2}}{-2c+N-2}\big|f(R)R^c\big|^2\nonumber\\
&+\frac{R^{-2c-2i+N-2}}{-2c-2i+N-2}\overline{f(R)R^c}\int_{R}^{\infty}\partial_{r}\big(f(r)r^c\big)r^{2i}\ud r\Big]\bigg\rangle.
\end{align}
\begin{align*}
S_5=&\bigg\langle-\sum\limits_{i=1}^{k_1-1}d_i(2k_1-2i-1)\nonumber\\
&\times\int_{R}^{\infty}\int_{R}^{\min{(r_1,R)}}r_2^{-c-2i+\frac{N-3}{2}}r_1^{2i-1}\big(f(r_1)r_1^c\big)\partial_{r_2}\left(\overline{f(r_2)r_2^c}r_2^{-c+\frac{N-1}{2}}\right)\ud r_2\ud r_1\bigg\rangle\nonumber\\
=&\bigg\langle\sum\limits_{i=1}^{k_1-1}d_i(2k_1-2i-1)(-c-2i+\frac{N-3}{2})\nonumber\\
&\times\int_{0}^{\infty}\int_{0}^{\min{(r_1,R)}}\big(f(r_1)r_1^c\big)r_1^{2i-1}\overline{f(r_2)r_2^c}r_2^{-2c-2i+N-3}\ud r_2\ud r_1\bigg\rangle\nonumber\\
&-\bigg\langle\sum\limits_{i=1}^{k_1-1}d_i(2k_1-2i-1)\overline{f(R)R^c}\int_{0}^{\infty}\min(r_1,R)^{-2c-2i+N-2}\big(f(r_1)r_1^c\big)r_1^{2i-1}\ud r_1\bigg\rangle\nonumber\\
=&-\bigg\langle\sum\limits_{i=1}^{k_1-1}d_i(2k_1-2i-1)\Big[\big|f(R)R^c\big|^2\frac{R^{-2c+N-2}}{-2c+N-2}\nonumber\\
&+R^{-2c-2i+N-2}\overline{f(R)R^c}\int_{R}^{\infty}\big(f(r)r^c\big)r^{2i-1}\ud r\Big]\bigg\rangle\\
&+\bigg\langle\sum\limits_{i=1}^{k_1-1}d_i(2k_1-2i-1)(-c-2i+\frac{N-3}{2})\\
&\times\Big[\frac{|f(R)R^c|^2}{-2c-2i+N-2}\frac{R^{-2c+N-2}}{-2c+N-2}+\frac{R^{-2c-2i+N-2}}{-2c-2i+N-2}\overline{f(R)R^{c}}\int_{R}^{\infty}(f(r)r^c)r^{2i-1}\ud r\Big]\bigg\rangle\\
=&-\Big(-c+\frac{N-1}{2}\Big)\bigg\langle\sum\limits_{i=1}^{k_1-1}\frac{d_i(2k_1-2i-1)}{-2c-2i+N-2}\Big[\big|f(R)R^C\big|^2\frac{R^{-2c+N-2}}{-2c+N-2}\nonumber\\
&+\overline{f(R)R^c}R^{-2c-2i+N-2}\int_{R}^{\infty}\big(f(r)r^c)r^{2i-1}\ud r\Big]\bigg\rangle\nonumber
\end{align*}
\begin{align}\label{S5}
=&\Big(-c+\frac{N-1}{2}\Big)\bigg\langle\sum\limits_{i=1}^{k_1-1}\frac{d_1(2k_1-2i-1)}{2i}\Big[\frac{R^{-2c+N-2}}{-2c+N-2}\big|f(R)R^c\big|^2\nonumber\\
&+\frac{R^{-2c-2i+N-2}}{-2c-2i+N-2}\overline{f(R)R^c}\int_{R}^{\infty}\partial_{r}\big(f(r)r^c\big)r^{2i}\ud r\Big]\bigg\rangle.
\end{align}
\begin{align}\label{S6}
S_6=&\bigg\langle\sum\limits_{i=1}^{k_1-1}d_i\int_{0}^{R}\big(f(r)r^c\big)r^{-c+\frac{N-3}{2}}\partial_{r}\big(\overline{f(r)r^c}r^{-c+\frac{N-1}{2}}\big)\ud r\bigg\rangle\nonumber\\
=&\Big(-c+\frac{N-1}{2}\Big)\bigg\langle\sum\limits_{i=1}^{k_1-1}d_i\big|f(R)R^c\big|^2\frac{R^{-2c+N-2}}{-2c+N-2}\bigg\rangle.
\end{align}

Direct computation shows that the terms from $S_1$ to $S_6$ can be divided into three categories based on the integral. Sorting out the formulas we got and collecting the same terms, we see that:

\medskip
{\bf The first type} The terms with $\int_{R}^{\infty}\partial_{r}(f(r)r^c)r^{2i}\ud r\int_{R}^{\infty}\partial_{r}(\overline{f(r)r^c})r^{2j}\ud r$.
Recall from the assumption $k_1=\frac{(N-2)\beta+3}{4}$ and $\lfloor\frac{(N-2)\beta+4}{4}\rfloor=\lfloor k_1+\frac{1}{4}\rfloor=k_1$, we have
\begin{align*}
d_j=\frac{\prod_{1\leq \ell\leq k_1-1}(4k_1-2\ell-2j-1)}{\prod_{1\leq \ell\leq k_1-1,\ell\neq j}(2\ell-2j)}, \quad 1\leq j\leq k_1-1,
\end{align*}
and
\begin{align*}
\tilde{d_j}=\frac{\prod_{1\leq \ell\leq k_1}(4k_1-2\ell-2j+1)}{\prod_{1\leq \ell\leq k_1,\ell\neq j}(2\ell-2j)}, \quad 1\leq j\leq k_1.
\end{align*}
These facts gives
\begin{equation}\label{relationship of di}
d_j=\tilde{d}_{j+1}\frac{-2j}{2k_1-2j-1}, \quad 1\leq j\leq k_1-1.
\end{equation}

Thus we obtain terms of the first type from $S_1$ \eqref{S1},
\begin{align*}
{\rm\uppercase\expandafter{\romannumeral1}}=&\sum\limits_{i,j=1}^{k_1-1}\frac{d_id_j(2k_1-2i-1)(2k_1-2j-1)}{4i j}\frac{R^{-2c-2i-2j+N-2}}{-2c-2i-2j+N-2}\\
&\times\int_{R}^{\infty}\partial_{r}\big(f(r)r^c\big)r^{2i}\ud r\int_{R}^{\infty}\partial_{r}(\overline{f(r)r^c})r^{2j}\ud r\\
=&\sum\limits_{i',j'=2}^{k_1}\tilde{d}_i'\tilde{d}_j'\frac{R^{-2c-2i'-2j'+N-2}}{-2c-2i'-2j'+N-2}\int_{R}^{\infty}\partial_{r}\big(f(r)r^c\big)r^{2i'-2}\ud r\int_{R}^{\infty}\partial_{r}\big(\overline{f(r)r^c}\big)r^{2j'-2}\ud r.
\end{align*}

{\bf The second type}  We collect the terms with $\int_{R}^{\infty}\partial_{r}(f(r)r^c)r^{2i}\ud r$ from $S_1$, $S_4$, $S_5$, \eqref{S1}, \eqref{S4}, \eqref{S5}.
Since there are two types of terms:  $\int_{R}^{\infty}\partial_{r}(f(r)r^c)r^{2i}\ud r$ and $\int_{R}^{\infty}\partial_{r}(\overline{f(r)r^c})r^{2j}\ud r$. Since
they are symmetric, we only treat the term $\int_{R}^{\infty}\partial_{r}(f(r)r^c)r^{2i}\ud r$. We perform the same merge processing to the other case.
\begin{align*}
{\rm\uppercase\expandafter{\romannumeral2}}=&-\sum\limits_{i,j=1}^{k_1-1}\frac{d_id_j(2k_1-2i-1)(2k_1-2j-1)}{2j(-2c-2i-2j+N-2)}\qquad\qquad\qquad\qquad\qquad\\
&\times\frac{R^{-2c-2j+N-2}}{-2c-2j+N-2}f(R)R^c\int_{R}^{\infty}\partial_{r}(\overline{f(r)r^c})r^{2j}\ud r\qquad\qquad\qquad\qquad\qquad\\
&+\sum\limits_{i,j=1}^{k_1-1}\frac{d_id_j(2k_1-2i-1)(2k_1-2j-1)}{4i j(-2c-2i-2j+N-2)}R^{-2c-2j+N-2}f(R)R^c\int_{R}^{\infty}\partial_{r}(\overline{f(r)r^c})r^{2j}\ud r\\
&+\sum\limits_{i,j=1}^{k_1-1}\frac{d_id_j(2k_1-2j-1)}{2j}\frac{R^{-2c-2j+N-2}}{-2c-2j+N-2}f(R)R^c\int_{R}^{\infty}\partial_{r}(\overline{f(r)r^c})r^{2j}\ud r\\
&+\big(-c+\frac{N-1}{2}\big)\sum\limits_{j=1}^{k_1-1}\frac{d_j(2k_1-2j-1)}{2j}\frac{R^{-2c-2j+N-2}}{-2c-2j+N-2}f(R)R^c\int_{R}^{\infty}\partial_{r}(\overline{f(r)r^c})r^{2j}\ud r\\
=&\sum\limits_{j=1}^{k_1-1}f(R)R^c R^{-2c-2j+N-2}\int_{R}^{\infty}\partial_{r}(\overline{f(r)r^c})r^{2j}\ud r\\
&\times\sum\limits_{i=1}^{k_1-1}\bigg[-\frac{d_id_j(2k_1-2i-1)(2k_2-2j-1)}{2j(-2c-2i-2j+N-2)}+\frac{d_id_j(2k_1-2i-1)(2k_1-2j-1)}{4i j(-2c-2i-2j+N-2)}\\
&+\frac{d_id_j(2k_1-2j-1)}{2j(-2c-2j+N-2)}\bigg]
\end{align*}
\begin{align*}
&+\big(-c+\frac{N-1}{2}\big)\sum\limits_{j=1}^{k_1-1}\frac{d_j(2k_1-2j-1)}{2j}\frac{R^{-2c-2j+N-2}}{-2c-2j+N-2}f(R)R^c\int_{R}^{\infty}\partial_{r}(\overline{f(r)r^c})r^{2j}\ud r\\
=&\big(-c+\frac{N-1}{2}\big)\sum\limits_{j=1}^{k_1-1}f(R)R^c R^{-2c-2j+N-2}\int_{R}^{\infty}\partial_{r}(\overline{f(r)r^c})r^{2j}\ud r\\
&\times\bigg[\sum\limits_{i=1}^{k_1-1}\frac{d_id_j(2k_1-2j-1)}{4i j(-2c-2j+N-2)}+\frac{d_j(2k_1-2j-1)}{2j(-2c-2j+N-2)}\bigg].
\end{align*}

Define $i'=i+1$, $j'=j+1$, by the relationship \eqref{relationship of di}, we easily check that
\begin{align}\label{the type 2}
{\rm\uppercase\expandafter{\romannumeral2}}=&-\big(-c+\frac{N-1}{2}\big)\sum\limits_{j'=2}^{k_1}\tilde{d}_j'\bigg(1-\sum\limits_{i'=2}^{k_1}\frac{\tilde{d}_i'}{2k_1-2i'+1}\bigg)\nonumber\\
&\times \frac{R^{-2c+N-2}}{-2c+N-2}f(R)R^c\int_{R}^{\infty}\partial_{r}(\overline{f(r)r^c})r^{2j'-2}\ud r.
\end{align}
By \eqref{sum dj} and the fact that $2k_1=-c+\frac{N+1}{2}$, $(N-2)\beta+4=4k_1+1$, we choose $m=k_1$ and get
\begin{align}\label{relationship d1}
1-\sum\limits_{i'=2}^{k_1}\frac{\tilde{d}_i'}{2k_1-2i'+1}=\frac{\tilde{d}'_1}{2k_1-1}=\frac{\tilde{d}'_1}{-c+\frac{N-1}{2}}.
\end{align}

Finally, by substituting \eqref{relationship d1} into \eqref{the type 2}, the second type can be written as
\begin{align*}
{\rm\uppercase\expandafter{\romannumeral2}}=&-\sum\limits_{j'=2}^{k_1}\tilde{d}_j'\tilde{d}'_1 \frac{R^{-2c-2j'+N}}{-2c-2j'+N}f(R)R^c\int_{R}^{\infty}\partial_{r}(\overline{f(r)r^c})r^{2j'-2}\ud r.
\end{align*}

{\bf The third type} We collect the terms with $|f(R)R^c|^2$, which come from $S_1$, $S_2$, $\cdots$, $S_6$, \eqref{S1} to \eqref{S6}.
\begin{align*}
{\rm\uppercase\expandafter{\romannumeral3}}=&\sum\limits_{i,j=1}^{k_1-1}\frac{d_id_j(2k_1-2i-1)(2k_1-2j-1)}{4i j}\frac{R^{-2c+N-2}}{-2c-2i-2j+N-2}\big|f(R)R^c\big|^2\\
&-\bigg\langle\sum\limits_{i,j=1}^{k_1-1}\frac{d_id_j(2k_1-2i-1)(2k_1-2j-1)}{2i(-2c-2i-2j+N-2)}\frac{R^{-2c+N-2}}{-2c+N-2}\big|f(R)R^c\big|^2\bigg\rangle
\end{align*}
\begin{align*}
&+\sum\limits_{i,j=1}^{k_1-1}d_id_j\frac{R^{-2c+N-2}}{-2c+N-2}\big|f(R)R^c\big|^2+\big(-c+\frac{N-1}{2}\big)^2\frac{R^{-2c+N-2}}{-2c+N-2}\big|f(R)R^c\big|^2\\
&+\bigg\langle\sum\limits_{i,j=1}^{k_1-1}\frac{d_id_j(2k_1-2i-1)}{2i}\frac{R^{-2c+N-2}}{-2c+N-2}\big|f(R)R^c\big|^2\bigg\rangle\\
&+\bigg\langle\big(-c+\frac{N-1}{2}\big)\sum\limits_{i=1}^{k_1-1}\frac{d_i(2k_1-2i-1)}{2i}\frac{R^{-2c+N-2}}{-2c+N-2}\big|f(R)R^c\big|^2\bigg\rangle\\
&+\bigg\langle\big(-c+\frac{N-1}{2}\big)\sum\limits_{i=1}^{k_1-1}d_i\frac{R^{-2c+N-2}}{-2c+N-2}\big|f(R)R^c\big|^2\bigg\rangle.
\end{align*}

\noindent Taking $i'=i+1$, $j'=j+1$, we conclude by \eqref{relationship of di} and \eqref{sum dj}
\begin{align*}
{\rm\uppercase\expandafter{\romannumeral3}}=&\sum\limits_{i',j'=2}^{k_1}\tilde{d}_{i'}\tilde{d}_{j'}\bigg[\frac{1}{-2c-2i'-2j'+N+2}\Big(1+\frac{-2j'+2}{(-2c+N-2)}+\frac{-2i'+2}{(-2c+N-2)}\Big)\\
&-\frac{-2i'+2}{(2k_1-2i'+1)(-2c+N-2)}-\frac{-2j'+2}{(2k_1-2j'+1)(-2c+N-2)}\\
&+\frac{(-2i'+2)(-2j'+2)}{(2k_1-2i'+1)(2k_1-2j'+1)(-2c+N-2)}\bigg]R^{-2c+N-2}\big|f(R)R^c\big|^2\\
&-(-c+\frac{N-1}{2})\sum\limits_{i'=2}^{k_1}\tilde{d}_{i'}\Big(1+\frac{2i'-2}{2k_1-2i'+1}\Big)\frac{R^{-2c+N-2}}{-2c+N-2}\big|f(R)R^c\big|^2\\
&-(-c+\frac{N-1}{2})\sum\limits_{j'=2}^{k_1}\tilde{d}_{j'}\Big(1+\frac{2j'-2}{2k_1-2j'+1}\Big)\frac{R^{-2c+N-2}}{-2c+N-2}\big|f(R)R^c\big|^2\\
&+(-c+\frac{N-1}{2})^2\frac{R^{-2c+N-2}}{-2c+N-2}\big|f(R)R^c\big|^2\\
=&(-c+\frac{N-1}{2})^2\bigg(\sum\limits_{i',j'=2}^{k_1}\frac{\tilde{d}_{i'}\tilde{d}_{j'}}{(2k_1-2i'+1)(2k_1-2j'+1)}-\sum\limits_{i'=2}^{k_1}\frac{\tilde{d}_{i'}}{2k_1-2i'+1}\\
&-\sum\limits_{j'=2}^{k_1}\frac{\tilde{d}_{j'}}{2k_1-2j'+1}+1\bigg)\frac{R^{-2c+N-2}}{-2c+N-2}\big|f(R)R^c\big|^2\\
=&(2k_1-1)^2\bigg[1-\sum\limits_{i'=2}^{k_1}\frac{\tilde{d}_{i'}}{2k_1-2i'+1}\bigg]\bigg[1-\sum\limits_{j'=2}^{k_1}\frac{\tilde{d}_{j'}}{2k_1-2j'+1}\bigg]\frac{R^{-2c+N-2}}{-2c+N-2}\big|f(R)R^c\big|^2\\
=&\tilde{d}_{1'}\tilde{d}_{1'}\frac{R^{-2c+N-2}}{-2c+N-2}\big|f(R)R^c\big|^2.
\end{align*}

Summing up all the estimates and  \eqref{projection of f}, we obtain that for $k-1$ is even
\begin{align*}
AS_{N}^{a}f
=&\int_{R}^{\infty}\Big|(\partial_{r}+\frac{c}{r})f(r)\Big|^2r^{N-1}\ud r -\sum\limits_{i,j=1}^{k_1}\tilde{d}_i\tilde{d}_j\frac{R^{-2c-2i-2j+N-2}}{-2c-2i-2j+N-2}\\
&\times\int_{R}^{\infty}\partial_{r}(f(r)r^c)r^{2i-2}\ud r\int_{R}^{\infty}\partial_{r}(\overline{f(r)r^c})r^{2j-2}\ud r\\
=&\frac{1}{2}\big\|\operatorname{\pi}_{\tilde{W}^{\perp}}f\big\|^{2}_{\dot{H}^1_a(r\geq R,r^{N-1}\ud r)}.
\end{align*}

\subsubsection{\bf Case $k-1=2k_1-1$} Now we deal with the case that $k-1$ is odd,  where the dimension of $\tilde{W}$ is $\tilde{k}_1=\lfloor\frac{(N-2)\beta}{4}\rfloor=\lfloor k_1+\frac{3}{4}\rfloor=k_1$ and $2k_1=k=-c+\frac{N-1}{2}$. Relying on the Fourier expansion of $\psi_{k-1}$ \eqref{odd Fourier psi} and the notation $F(r_1,r_2):=f(r_1)r_1^c\overline{f(r_2)r_2^c}(r_1r_2)^{-c+\frac{N-5}{2}}$, we infer that
\begin{align*}
\quad&2AS_{N}^{a}f(r)\\
=&\int_0^{\infty}\Big|(\partial_{r}+\frac{c}{r})f(r)\Big|^{2}r^{N-1}\ud r-\lim\limits_{\varepsilon\rightarrow0}\iint\limits_{r_1,r_2>0}\Big[\frac{1}{2\pi^2r_1^2r_2^2}\int_{-R}^{R}\big[\kappa_{\varepsilon}\ast\mathcal{F}\psi_{2k_1-1}\big]\big(\frac{\xi}{r_1}\big)\\
&\times\overline{\big[\kappa_{\varepsilon}\ast\mathcal{F}\psi_{2k_1-1}\big]\big(\frac{\xi}{r_2}\big)}\ud \xi\bigg] f(r_1)\overline{f(r_2)}(r_1r_2)^{\frac{N-1}{2}}\ud r_1\ud r_2
\end{align*}
\begin{align*}
=&\int_{0}^{\infty}\Big|(\partial_{r}+\frac{c}{r})f(r)\Big|^{2}r^{N-1}\ud r-\frac{1}{2}\sum\limits_{i,j=1}^{k_1}d_{i}d_{j}(2k_1-2i)(2k_1-2j)\\
&\times\lim\limits_{\varepsilon\rightarrow0}\iint\bigg[\int_{-R}^{R}\big[\kappa_{\varepsilon}\ast\xi^{2k_1-2i-1}\chi_{(-1,1)}\big]\big(\frac{\xi}{r_1}\big)\big[\kappa_{\varepsilon}\ast\xi^{2k_1-2j-1}\chi_{(-1,1)}\big]\big(\frac{\xi}{r_2}\big)\ud \xi\bigg]\\
&\times F(r_1,r_2)\ud r_1\ud r_2\\
&-\frac{1}{2}\sum\limits_{i,j=1}^{k_1}d_id_j\lim\limits_{\varepsilon\rightarrow0}\iint\bigg[\int_{-R}^{R}\big[\kappa_{\varepsilon}\ast(-\delta_{1}+\delta_{-1})\big]\big(\frac{\xi}{r_1}\big)\big[\kappa_{\varepsilon}\ast(-\delta_{1}+\delta_{-1})\big]\big(\frac{\xi}{r_2}\big)\bigg]\\
&\times F(r_1,r_2)\ud r_1\ud r_2\\
&-\frac{1}{2}\lim\limits_{\varepsilon\rightarrow0}\iint\Big[\int_{-R}^{R}\big[\kappa_{\varepsilon}\ast(-\delta'_{1}-\delta'_{-1})\big]\big(\frac{\xi}{r_1}\big)\big[\kappa_{\varepsilon}\ast(-\delta'_{1}-\delta'_{-1})\big]\big(\frac{\xi}{r_2}\big)\ud\xi\Big]F(r_1,r_2)\ud r_1\ud r_2\\
&-\frac{1}{2}\bigg\langle\sum\limits_{i,j=1}^{k_1}(2k_1-2i)d_id_j\lim\limits_{\varepsilon\rightarrow0}\iint\int_{-R}^{R}\big[\kappa_{\varepsilon}\ast\xi^{2k_1-2i-1}\chi_{(-1,1)}\big]\big(\frac{\xi}{r_1}\big)\\
&\times\big[\kappa_{\varepsilon}\ast(-\delta_{1}+\delta_{-1})\big]\big(\frac{\xi}{r_2}\big)\ud\xi F(r_1,r_2)\ud r_1\ud r_2\bigg\rangle\\
&-\frac{1}{2}\bigg\langle\sum\limits_{i=1}^{k_1}(2k_1-2i)d_i\lim\limits_{\varepsilon\rightarrow0}\iint\int_{-R}^{R}\big[\kappa_{\varepsilon}\ast\xi^{2k_1-2i-1}\chi_{(-1,1)}\big]\big(\frac{\xi}{r_1}\big)\big[\kappa_{\varepsilon}\ast(-\delta'_{1}-\delta'_{-1})\big]\big(\frac{\xi}{r_2})\ud \xi\\
&\times F(r_1,r_2)\ud r_1\ud r_2\bigg\rangle\\
&-\frac{1}{2}\bigg\langle\sum\limits_{i=1}^{k_1}d_i\lim\limits_{\varepsilon\rightarrow0}\iint\Big[\int_{-R}^{R}[\kappa_{\varepsilon}\ast(\delta_{1}+\delta_{-1})](\frac{\xi}{r_1})[\kappa_{\varepsilon}\ast(-\delta'_{1}-\delta'_{-1})](\frac{\xi}{r_2})\ud \xi\Big]\\
&\times F(r_1,r_2)\ud r_1\ud r_2\bigg\rangle\\
=&\int_{0}^{\infty}\Big|(\partial_{r}+\frac{c}{r})f(r)\Big|^2r^{N-1}\ud r-\sum\limits_{i,j=1}^{k_1}d_id_j\iint\limits_{r_1,r_2>0}\frac{(2k_1-2i)(2k_1-2j)}{4k_1-2i-2j-1}\\
&\times\frac{\min(r_1,r_2,R)^{4k_1-2i-2j-1}}{r_1^{2k_1-2i-1}r_2^{2k_1-2j-1}}f(r_1)r_1^{c}\overline{f(r_2)r_2^c}(r_1r_2)^{-c+\frac{N-5}{2}}\ud r_1\ud r_2\\
&-\sum\limits_{i,j=1}^{k_1}d_id_j\int_{0}^{R}|f(r)r^c|^2r^{-2c+N-3}\ud r-\int_{0}^{R}\big|\partial_{r}(f(r)r^c\cdot r^{-c+\frac{N-1}{2}})\big|^2\ud r\\
&+\bigg\langle\sum\limits_{i,j=1}^{k_1}d_id_j(2k_1-2i)\int_{0}^{\infty}\int_{0}^{\min(r_1,R)}r_2\big(\frac{r_2}{r_1}\big)^{2k_1-2i-1}f(r_1)r_1^{c}\overline{f(r_2)r_2^{c}}(r_1r_2)^{-c+\frac{N-5}{2}}\ud r_2\ud r_1\bigg\rangle
&+\bigg\langle\sum\limits_{i=1}^{k_1}d_i(2k_1-2i)\int_{0}^{\infty}\int_{0}^{\min(r_1,R)}\big(\frac{r_2}{r_1}\big)^{2k_1-2i-1}f(r_1)r_1^{c}r_1^{-c+\frac{N-5}{2}}\partial_{r_2}\big(\overline{f(r_2)r_2^{c}}r_2^{-c+\frac{N-1}{2}}\big)\ud r_2\ud r_1\bigg\rangle\\
&-\bigg\langle\sum\limits_{i=1}^{k_1}d_i\int_{0}^{R}f(r)r^{c}r^{-c+\frac{N-3}{2}}\partial_{r}\big(\overline{f(r)r^c}r^{-c+\frac{N-1}{2}}\big)\ud r\bigg\rangle\\
=&\int_{0}^{\infty}\Big|(\partial_{r}+\frac{c}{r})f(r)\Big|^2r^{N-1}\ud r-(\tilde{S}_1+\tilde{S}_2+\tilde{S}_3+\tilde{S}_4+\tilde{S}_5+\tilde{S}_6).
\end{align*}

As before, we recall the relationship $2k_1=-c+\frac{N-1}{2}$ and assumption $f(r)r^c=f(R)R^c$ for $r\leq R$, we compute each $\tilde{S}_i$:
\begin{align*}
\tilde{S}_1=&\bigg\langle\sum\limits_{i,j=1}^{k_1}\frac{d_id_j(2k_1-2i)(2k_1-2j)}{(-2c-2i-2j+N-2)}\iint\limits_{0<r_1<r_2<R}r_1^{-2c-2j+N-3}r_2^{2j-1}(f(r_1)r_1^{c})(\overline{f(r_2)r_2^c})\ud r_1\ud r_2\bigg\rangle\\
&+\bigg\langle\sum\limits_{i,j=1}^{k_1}\frac{d_id_j(2k_1-2i)(2k_1-2j)}{(-2c-2i-2j+N-2)}\iint\limits_{0<r_1<R<r_2}r_1^{-2c-2j+N-3}r_2^{2j-1}(f(r_1)r_1^{c})(\overline{f(r_2)r_2^c})\ud r_1\ud r_2\bigg\rangle\\
&+\sum\limits_{i,j=1}^{k_1}\frac{d_id_j(2k_1-2i)(2k_1-2j)}{(-2c-2i-2j+N-2)}R^{-2c-2i-2j+N-2}\int_{R}^{\infty}(f(r)r^c)r^{2i-1}\ud r\int_{R}^{\infty}(\overline{f(r)r^c})r^{2j-1}\ud r\\
=&\bigg\langle\sum\limits_{i,j=1}^{k_1}d_id_j\frac{(2k_1-2i)(2k_1-2j)}{(-2c-2i-2j+N-2)(-2c-2j+N-2)}\\
&\times\Big[\frac{R^{-2c+N-2}}{-2c+N-2}|f(R)R^{c}|^2+R^{-2c-2j+N-2}(f(R)R^c)\int_{R}^{\infty}(\overline{f(r)r^c})r^{2j-1}\ud r\Big]\bigg\rangle\\
&+\sum\limits_{i,j=1}^{k_1-1}\frac{d_id_j(2k_1-2i)(2k_1-2j)}{(-2c-2i-2j+N-2)}R^{-2c-2i-2j+N-2}\int_{R}^{\infty}(f(r)r^c)r^{2i-1}\ud r\int_{R}^{\infty}(\overline{f(r)r^c})r^{2j-1}\ud r.
\end{align*}

\noindent Integration by parts implies
\begin{equation*}
\int_{R}^{\infty}(\overline{f(r)r^c)}r^{2j-1}\ud r=-\frac{1}{2j}\Big[(\overline{f(R)R^c})R^{2j}+\int_{R}^{\infty}\partial_{r}(\overline{f(r)r^c})r^{2j}\ud r\Big],
\end{equation*}
so we have
\begin{align}\label{SS1}
\tilde{S}_1=&\bigg\langle-\sum\limits_{i,j=1}^{k_1}\frac{d_id_j(2k_1-2i)(2k_1-2j)}{2j(-2c-2i-2j+N-2)}\nonumber\\
&\times\Big[\frac{R^{-2c+N-2}}{-2c+N-2}|f(R)R^c|^2+\frac{R^{-2c-2j+N-2}}{-2c-2j+N-2}f(R)R^c\int_{R}^{\infty}\partial_{r}(\overline{f(r)r^c})r^{2j}\ud r\Big]\bigg\rangle\nonumber\\
&+\bigg\langle\sum\limits_{i,j=1}^{k_1}\frac{d_id_j(2k_1-2i)(2k_1-2j)}{4i j(-2c-2i-2j+N-2)}R^{-2c-2j+N-2}f(R)R^{c}\int_{R}^{\infty}\partial_{r}(\overline{f(r)r^c})r^{2j}\ud r\bigg\rangle\nonumber\\
&+\sum\limits_{i,j=1}^{k_1}\frac{d_id_j(2k_1-2i)(2k_1-2j)}{4i j}\frac{R^{-2c-2i-2j+N-2}}{-2c-2i-2j+N-2}\nonumber\\
&\times\int_{R}^{\infty}\partial_{r}(f(r)r^c)r^{2i}\ud r\int_{R}^{\infty}\partial_{r}(\overline{f(r)r^c})r^{2j}\ud r\nonumber\\
&+\sum\limits_{i,j=1}^{k_1}\frac{d_id_j(2k_1-2i)(2k_1-2j)}{4i j}\frac{R^{-2c+N-2}}{-2c-2i-2j+N-2}|f(R)R^c|^2.
\end{align}

\begin{align}\label{SS2}
\tilde{S}_2=&\sum\limits_{i,j=1}^{k_1}d_id_j\int_{0}^{R}|f(r)r^c|^2r^{-2c+N-3}\ud r=\sum\limits_{i,j=1}^{k_1}d_id_j\frac{R^{-2c+N-2}}{-2c+N-2}|f(R)R^c|^2.
\end{align}

\begin{align}\label{SS3}
\tilde{S}_3=\int_{0}^{R}\left|\partial_{r}(f(r)r^c r^{-c+\frac{N-1}{2}})\right|^2\ud r=\big(-c+\frac{N-1}{2}\big)^2\frac{R^{-2c+N-2}}{-2c+N-2}|f(R)R^c|^2.
\end{align}

\begin{align}\label{SS4}
\tilde{S}_4=&-\bigg\langle\sum\limits_{i,j=1}^{k_1}d_id_j(2k_1-2i)\int_{0}^{\infty}\int_{0}^{\min(r_1,R)}r_2^{-2c-2i+N-3}r_1^{2i-1}(f(r_1)r_1^{c})(\overline{f(r_2)r_2^{c}})\ud r_2\ud r_1\bigg\rangle\nonumber\\
=&-\bigg\langle\sum\limits_{i,j=1}^{k_1}\frac{d_id_j(2k_1-2i)}{-2c-2i+N-2}\Big[\frac{R^{-2c+N-2}}{-2c+N-2}|f(R)R^c|^2\nonumber\\
&\qquad+R^{-2c-2i+N-2}\overline{f(R)R^c}\int_{R}^{\infty}(f(r)r^c)r^{2i-1}\ud r\Big]\bigg\rangle\nonumber\\
=&\bigg\langle\sum\limits_{i,j=1}^{k_1}\frac{d_id_j(2k_1-2i)}{2i}\Big[\frac{R^{-2c+N-2}}{-2c+N-2}|f(R)R^c|^2\nonumber\\
&\quad+\frac{R^{-2c-2i+N-2}}{-2c-2i+N-2}\overline{f(R)R^c}\int_{R}^{\infty}\partial_{r}(f(r)r^c)r^{2i}\ud r\Big]\bigg\rangle.
\end{align}
\begin{align}\label{SS5}
\tilde{S}_5=&\bigg\langle\sum\limits_{i=1}^{k_1}d_i(2k_1-2i)(-c-2i+\frac{N-3}{2})\nonumber\\
&\times\int_{0}^{\infty}\int_{0}^{\min{(r_1,R)}}(f(r_1)r_1^c)r_1^{2i-1}\overline{f(r_2)r_2^c}r_2^{-2c-2i+N-3}\ud r_2\ud r_1\bigg\rangle\nonumber\\
&-\bigg\langle\sum\limits_{i=1}^{k_1}d_i(2k_1-2i)\overline{f(R)R^c}\int_{0}^{\infty}(f(r_1)r_1^c)r_1^{2i-1}\min(r_1,R)^{-2c-2i+N-2}\ud r_1\bigg\rangle\nonumber\\
=&-\big(-c+\frac{N-1}{2}\big)\bigg\langle\sum\limits_{i=1}^{k_1}\frac{d_i(2k_1-2i)}{-2c-2i+N-2}\Big[|f(R)R^C|^2\frac{R^{-2c+N-2}}{-2c+N-2}\nonumber\\
&+\overline{f(R)R^c}R^{-2c-2i+N-2}\int_{R}^{\infty}(f(r)r^c)r^{2i-1}\ud r\Big]\bigg\rangle\nonumber\\
=&\big(-c+\frac{N-1}{2}\big)\bigg\langle\sum\limits_{i=1}^{k_1}\frac{d_1(2k_1-2i)}{2i}\Big[\frac{R^{-2c+N-2}}{-2c+N-2}|f(R)R^c|^2\nonumber\\
&+\frac{R^{-2c-2i+N-2}}{-2c-2i+N-2}\overline{f(R)R^c}\int_{R}^{\infty}\partial_{r}(f(r)r^c)r^{2i}\ud r\Big]\bigg\rangle.
\end{align}
\begin{align}\label{SS6}
\tilde{S}_6=&\big(-c+\frac{N-1}{2}\big)\bigg\langle\sum\limits_{i=1}^{k_1}d_i|f(R)R^c|^2\frac{R^{-2c+N-2}}{-2c+N-2}\bigg\rangle.
\end{align}
Similarly, we sort out and collect the same terms in the above formulas.

{\bf Type ${\rm\uppercase\expandafter{\romannumeral1}}$}\quad The terms with $\int_{R}^{\infty}\partial_{r}(f(r)r^c)r^{2i}\ud r\int_{R}^{\infty}\partial_{r}(\overline{f(r)r^c})r^{2j}\ud r$.
The assumption $k_1=\frac{(N-2)\beta+1}{4}$ and $\lfloor\frac{(N-2)\beta+4}{4}\rfloor=\lfloor k_1+\frac{3}{4}\rfloor=k_1$ yields
\begin{align*}
d_j=\frac{\prod_{1\leq \ell\leq k_1}(4k_1-2\ell-2j+1)}{\prod_{1\leq \ell\leq k_1-1,\ell\neq j}(2\ell-2j)}, \quad 1\leq j\leq k_1,\\
\tilde{d_j}=\frac{\prod_{1\leq \ell\leq k_1}(4k_1-2\ell-2j+3)}{\prod_{1\leq \ell\leq k_1,\ell\neq j}(2\ell-2j)}, \quad 1\leq j\leq k_1,
\end{align*}
hence
\begin{equation}\label{relationship of ddi}
d_j=\tilde{d}_{j+1}\frac{-2j}{2k_1-2j}, \quad 1\leq j\leq k_1-1.
\end{equation}

Since the term vanishes when the $i,j$ are equal to $k_1$, the summation is from $1$ to $k_1-1$. We preform the change of indices $i'=i+1$ and $j'=j+1$, according to \eqref{relationship of ddi} and  the first type of term \eqref{SS1} from $\tilde{S}_1$, we conclude
\begin{align}\label{odd type first}
{\rm\uppercase\expandafter{\romannumeral1}}=&\sum\limits_{i,j=1}^{k_1}\frac{d_id_j(2k_1-2i)(2k_1-2j)}{4i j}\frac{R^{-2c-2i-2j+N-2}}{-2c-2i-2j+N-2}\nonumber\\
&\times\int_{R}^{\infty}\partial_{r}(f(r)r^c)r^{2i}\ud r\int_{R}^{\infty}\partial_{r}(\overline{f(r)r^c})r^{2j}\ud r\nonumber\\
=&\sum\limits_{i',j'=2}^{k_1}\tilde{d}_i'\tilde{d}_j'\frac{R^{-2c-2i'-2j'+N-2}}{-2c-2i'-2j'+N-2}\int_{R}^{\infty}\partial_{r}(f(r)r^c)r^{2i'-2}\ud r\int_{R}^{\infty}\partial_{r}(\overline{f(r)r^c})r^{2j'-2}\ud r.
\end{align}

\medskip
{\bf Type ${\rm\uppercase\expandafter{\romannumeral2}}$}\quad As before we only collect the terms with $\int_{R}^{\infty}\partial_{r}(f(r)r^c)r^{2i}\ud r$ from $\tilde{S}_1$, $\tilde{S}_4$, $\tilde{S}_5$, \eqref{SS1}, \eqref{SS4}, \eqref{SS5}, the other terms are symmetric.
\begin{align*}
{\rm \uppercase\expandafter{\romannumeral2}}=&-\sum\limits_{i,j=1}^{k_1}\frac{d_id_j(2k_1-2i)(2k_1-2j)}{2j(-2c-2i-2j+N-2)}\frac{R^{-2c-2j+N-2}}{-2c-2j+N-2}f(R)R^c\int_{R}^{\infty}\partial_{r}(\overline{f(r)r^c})r^{2j}\ud r\nonumber\\
&+\sum\limits_{i,j=1}^{k_1}\frac{d_id_j(2k_1-2i)(2k_1-2j)}{4i j(-2c-2i-2j+N-2)}R^{-2c-2j+N-2}f(R)R^c\int_{R}^{\infty}\partial_{r}(\overline{f(r)r^c})r^{2j}\ud r\nonumber\\
&+\sum\limits_{i,j=1}^{k_1}\frac{d_id_j(2k_1-2j)}{2j}\frac{R^{-2c-2j+N-2}}{-2c-2j+N-2}f(R)R^c\int_{R}^{\infty}\partial_{r}(\overline{f(r)r^c})r^{2j}\ud r\nonumber\\
&+\big(-c+\frac{N-1}{2}\big)\sum\limits_{j=1}^{k_1}\frac{d_j(2k_1-2j)}{2j}\frac{R^{-2c-2j+N-2}}{-2c-2j+N-2}f(R)R^c\int_{R}^{\infty}\partial_{r}(\overline{f(r)r^c})r^{2j}\ud r\nonumber\\
=&2k_1\sum\limits_{j=1}^{k_1}d_j\Big(\sum\limits_{i=1}^{k_1}\frac{d_i}{2i}+1\Big)\frac{2k_1-2j}{2j}\frac{R^{-2c-2j+N-2}}{-2c-2j+N-2}f(R)R^c\int_{R}^{\infty}\partial_{r}(\overline{f(r)r^c})r^{2j}\ud r.
\end{align*}
Calculate the summation term by \eqref{prod cj}
\begin{align}\label{prod important}
1+\sum\limits_{i=1}^{k_1}\frac{d_i}{2i}&=\prod\limits_{\ell=1}^{k_1}\frac{4k_1-2\ell+1}{2\ell}=\frac{1}{2k_1}\frac{\prod_{\ell=1}^{k_1}(4k_1-2\ell+1)}{\prod_{\ell=2}^{k_1}(2\ell-2)}=\frac{\tilde{d}_1}{2k_1}.
\end{align}

Since $2k_1=-c+\frac{N-1}{2}$, the second type can be written equivalently as
\begin{align}\label{odd type second}
{\rm\uppercase\expandafter{\romannumeral2}}=&\sum\limits_{j=1}^{k_1}d_j\tilde{d}_1\frac{2k_1-2j}{2j}\frac{R^{-2c-2j+N-2}}{-2c-2j+N-2}f(R)R^c\int_{R}^{\infty}\partial_{r}(\overline{f(r)r^c})r^{2j}\ud r\nonumber\\
=&-\sum\limits_{j'=2}^{k_1}\tilde{d}_j'\tilde{d}'_1 \frac{R^{-2c-2j'+N}}{-2c-2j'+N}f(R)R^c\int_{R}^{\infty}\partial_{r}(\overline{f(r)r^c})r^{2j'-2}\ud r.
\end{align}

{\bf Type ${\rm\uppercase\expandafter{\romannumeral3}}$}\quad We collect the terms with $|f(R)R^c|^2$, which come from $\tilde{S}_1$, $\tilde{S}_2$, $\cdots$, $\tilde{S}_6$, \eqref{SS1} to \eqref{SS6}.
\begin{align*}
{\rm\uppercase\expandafter{\romannumeral3}}=&\sum\limits_{i,j=1}^{k_1}\frac{d_id_j(2k_1-2i)(2k_1-2j)}{4i j}\frac{R^{-2c+N-2}}{-2c-2i-2j+N-2}|f(R)R^c|^2\\
&-\bigg\langle\sum\limits_{i,j=1}^{k_1}\frac{d_id_j(2k_1-2i)(2k_1-2j)}{2i(-2c-2i-2j+N-2)}\frac{R^{-2c+N-2}}{-2c+N-2}|f(R)R^c|^2\bigg\rangle\\
&+\sum\limits_{i,j=1}^{k_1}d_id_j\frac{R^{-2c+N-2}}{-2c+N-2}|f(R)R^c|^2+\big(-c+\frac{N-1}{2}\big)^2\frac{R^{-2c+N-2}}{-2c+N-2}|f(R)R^c|^2\\
&+\bigg\langle\sum\limits_{i,j=1}^{k_1}\frac{d_id_j(2k_1-2i)}{2i}\frac{R^{-2c+N-2}}{-2c+N-2}|f(R)R^c|^2\bigg\rangle\\
&+\bigg\langle\big(-c+\frac{N-1}{2}\big)\sum\limits_{i=1}^{k_1}\frac{d_i(2k_1-2i)}{2i}\frac{R^{-2c+N-2}}{-2c+N-2}|f(R)R^c|^2\bigg\rangle\\
&+\bigg\langle\big(-c+\frac{N-1}{2}\big)\sum\limits_{i=1}^{k_1}d_i\frac{R^{-2c+N-2}}{-2c+N-2}|f(R)R^c|^2\bigg\rangle\\
=&\big(-c+\frac{N-1}{2}\big)^2\Big[\sum\limits_{i=1}^{k_1}\frac{d_i}{2i}+1\Big]\Big[\sum
\limits_{j=1}^{k_1}\frac{d_j}{2j}+1\Big]\frac{R^{-2c+N-2}}{-2c+N-2}|f(R)R^c|^2.
\end{align*}
Using the identity \eqref{prod important}, and taking $i'=i+1$, $j'=j+1$, we rewrite
\begin{align}\label{odd typy three}
{\rm\uppercase\expandafter{\romannumeral3}}=&\tilde{d'_1}^2\frac{R^{-2c+N-2}}{-2c+N-2}|f(R)R^c|^2.
\end{align}

In conclusion, using the identities \eqref{odd type first}, \eqref{odd type second} and \eqref{odd typy three} we infer that when $k-1$ is odd
\begin{align*}
AS_{N}^{a}f
=&\int_{R}^{\infty}\Big|(\partial_{r}+\frac{c}{r})f(r)\Big|^2r^{N-1}\ud r -\sum\limits_{i,j=1}^{k_1}\tilde{d}_i\tilde{d}_j\frac{R^{-2c-2i-2j+N-2}}{-2c-2i-2j+N-2}\\
&\times\int_{R}^{\infty}\partial_{r}(f(r)r^c)r^{2i-2}\ud r\int_{R}^{\infty}\partial_{r}(\overline{f(r)r^c})r^{2j-2}\ud r\\
=&\frac{1}{2}\big\|\operatorname{\pi}_{\tilde{W}^{\perp}}f\big\|^{2}_{\dot{H}^1_a(r\geq R,r^{N-1}\ud r)}.
\end{align*}

\medskip

Combining all the identities for data $(0,g)$ in Section~\ref{SS:2.4} and for $(f,0)$ in Section \ref{SS:2.5}, the proof of Theorem~\ref{ce} is completed.

\medskip
\section{Channel of energy for the linearized wave equation with inverse-square potential}\label{4}
\subsection{Exterior energy lower bound for the linearized wave equation around the ground state.}
We consider the linearized equation around the ground state
\begin{align}\label{eq:liearized wave}
\partial_{t}^{2}u+\mathcal{L}_{W}^{a}u=0,
\end{align}
where $\mathcal{L}_{W}^{a}$ is the linearized operator
\begin{align*}
\mathcal{L}_{W}^{a}=-\Delta-\frac{N+2}{N-2}W_{a}^{\frac{4}{N-2}}+\frac{a}{|x|^2}.
\end{align*}
Define:
\begin{align*}
    \mathcal{Z}=\big\{Z\in\dot{H}_{a}^1:\mathcal{L}_{W}^a Z=0\big\}.
\end{align*}
Since \eqref{weq} is invariant by scaling translation, let
\begin{align*}
\Lambda W_{a}\triangleq x\cdot\nabla W_{a}+\frac{N-2}{2}W_{a},
\end{align*}
we have the following nondegeneracy property of $W_a$
\begin{align*}
    \mathcal{Z}=\operatorname{span}\{\Lambda W_{a}\}.
\end{align*}
Since we assume $k=\frac{(N-2)\beta+1}{2}\in\mathbb{Z}$, $(N-2)\beta$ is an odd integer. Note that for $(N-2)\beta\geq 1$, $\Lambda W_{a}\in \dot{H}_{a}^1$ and that $\Lambda W_a\in L^2$ if and only if $(N-2)\beta\geq 3$.

Define
\begin{align}\label{def: mathbf mathcal Z}
    \mathbf{\mathcal{Z}}=(\mathcal{Z}\times\mathcal{Z})\cap\mathcal{H}_a,
\end{align}
which is a finite dimensional subspace of $\mathcal{H}_a$. Note that if $u$ is a solution of \eqref{eq:liearized wave} with initial data $(u_0,u_1)\in\mathbf{\mathcal{Z}}$, then $u(t,x)=u_0(x)+t u_1(x)$.

Let us denote by $\operatorname{\pi}_{\mathbf{\mathcal{Z}}}(R)$ the orthogonal projection on $\mathbf{\mathcal{Z}}$ \eqref{def: mathbf mathcal Z} in the space $(\dot{H}_{a}^{1}\times L^2)(r\geq R,r^{N-1}\ud r)$. More precisely, for every $R>0$,
\begin{align}\label{def: projection of initial f g}
\operatorname{\pi}_{\mathbf{\mathcal{Z}}(R)}(f,g)=\frac{\int_{R}^{\infty}\left(\partial_{r}+\frac{c}{r}\right)f(r)\left(\partial_{r}+\frac{c}{r}\right)(\Lambda W_{a})r^{N-1}\ud r}{\int_{R}^{\infty}\left(\left(\partial_{r}+\frac{c}{r}\right)\Lambda W_{a}\right)^{2}r^{N-1}\ud r}\left(\Lambda W_{a},0\right).
\end{align}
The projection onto the orthogonal complement is
\begin{align*}
\operatorname{\pi}^{\bot}_{\mathbf{\mathcal{Z}}(R)}(f,g)=(f,g)-\operatorname{\pi}_{\mathbf{\mathcal{Z}}(R)}(f,g).
\end{align*}

We perform a spectral analysis associated with the operator $\mathcal{L}_{W}^{a}$ to find the solution of \eqref{eq:liearized wave} and establish the distorted Hankel transform corresponding to $\mathcal{L}_{W}^{a}$ on $\mathbb{R}^{N}$. Restricted to the subspace of radial functions, $\mathcal{L}_{W}^a$ can be represented as
\begin{align*}
\mathcal{L}_{W}^{a}=-\left(\partial_{r}^{2}+\frac{N-1}{r}\partial_{r}\right)+\frac{a}{r^{2}}+V(r),
\end{align*}
where $V(r)=-\frac{N+2}{N-2}W_{a}^{\frac{4}{N-2}}(r)$ is a smooth real valued even function of $r$ which decays with order $r^{-2(1+\beta)}$ as $r\rightarrow\infty$. Since the operator $\mathcal{L}_{W}^{a}$ admits absolutely continuous spectrum given by $[0,\infty)$, in order to define the distorted Hankel transformation, we consider $\mathcal{L}_{W}^{a}f=Ef$. Assume $E=\rho^{2}$, $\rho>0$ and let $\lambda=\rho r$, $\tilde{f}(r,\rho)=\lambda^{-\frac{N-2}{2}}\tilde{g}(\lambda)$. Then one can show that
\begin{align}\label{eq:spectrum equation}
\tilde{g}''(\lambda)+\frac{1}{\lambda}\tilde{g}'(\lambda)+\left(1-\frac{\nu^2}{\lambda^2}-\frac{N(N-2)\beta^2}{\lambda^2}\Big(\frac{\lambda^{\beta}\rho^{\beta}}{\lambda^{2\beta}+\rho^{2\beta}}\Big)^2\right)\tilde{g}(\lambda)=0,
\end{align}
where $\nu=\frac{(N-2)\beta}{2}$ is the order of the Bessel function $J_{\nu}$. The solution to $\eqref{eq:spectrum equation}$ with prescribed asymptotics at infinity can be constructed by the following lemma.

\begin{lemma}\label{lem:solution of spectrum}
For all $\rho\geq0$, there exists a real solution $\tilde{f}(r,\rho)$ to the equation
\begin{align*}
\mathcal{L}_{W}^{a}\tilde{f}(\cdot,\rho)=\rho^2\tilde{f}(\cdot,\rho),
\end{align*}
such that $\tilde{f}$ is a smooth function of $(r,\rho)\in\mathbb{R}\times\mathbb{R}_{+}$ of the form
\begin{align}\label{asymptotic of solution}
\tilde{f}(r,\rho)=(r\rho)^{-\frac{N-2}{2}}J_{\nu}(r\rho)+O\left((r\rho)^{-\frac{N+1}{2}}\right),\qquad as~r\rightarrow\infty.
\end{align}
\end{lemma}

\begin{proof}
Define the potential by
\begin{align*}
\tilde{V}(r)=\frac{N(N-2)\beta^2}{\lambda^2}\Big(\frac{\lambda^{\beta}\rho^{\beta}}{\lambda^{2\beta}+\rho^{2\beta}}\Big)^2.
\end{align*}
Then the equation \eqref{eq:spectrum equation} can be rewritten as
\begin{align*}
\tilde{g}''(\lambda)+\frac{1}{\lambda}\tilde{g}'(\lambda)+\Big(1-\frac{\nu^2}{\lambda^2}\Big)\tilde{g}(\lambda)=\tilde{V}(r)\tilde{g}(\lambda),
\end{align*}
which can be written as equivalent integral form
\begin{align*}
\tilde{g}(r,\rho)=J_{\nu}(\lambda)+&\int_{\lambda}^{\infty}{J_{\nu}}(\lambda-s)\tilde{V}(s)\tilde{g}(s)\ud s.
\end{align*}

Since for $\lambda\leq\rho$,
\begin{align*}
\left[\frac{\lambda^{\beta-1}\rho^{\beta}}{\rho^{2\beta}+\lambda^{2\beta}}\right]^2=\lambda^{-2\beta-2}\bigg[\frac{\rho^{\beta}}{1+\left(\frac{\rho}{\lambda}\right)^{2\beta}}\bigg]^2\approx\lambda^{-2}\left(\frac{\lambda}{\rho}\right)^{2\beta}\leq\lambda^{-2},
\end{align*}
and similarly, for $\lambda\geq\rho$,
\begin{align*}
\left[\frac{\lambda^{\beta-1}\rho^{\beta}}{\rho^{2\beta}+\lambda^{2\beta}}\right]^2\lesssim\lambda^{-4\beta}\lambda^{2\beta-2}\rho^{2\beta}=\lambda^{-2}\left(\frac{\rho}{\lambda}\right)^{2\beta}\leq\lambda^{-2},
\end{align*}
we obtain
\begin{align*}
|\tilde{V}(\lambda)|\leq C\lambda^{-2}.
\end{align*}

Meanwhile, the asymptotic series expansion of $J_{\nu}(z)$ in \eqref{eq:J} yields
\begin{align*}
|J_{\nu}(\lambda-s)|\leq C|\lambda-s|^{-\frac{1}{2}},
\end{align*}
and $\tilde{g}$ solves a Volterra integral equation hence $\sup\limits_{\lambda\in\mathbb{R}}|\tilde{g}(\lambda)|\leq C(\tilde{V})$. Therefore,
\begin{align*}
|\tilde{g}(\lambda)|\leq&J_{\nu}(\lambda)+C(V)\int_{\lambda}^{\infty}(\lambda-s)^{-\frac{1}{2}}s^{-2}\ud s\leq J_{\nu}(\lambda)+C(V)\lambda^{-\frac{3}{2}}.
\end{align*}

\noindent Thus, we get
\begin{align*}
\tilde{f}(r,\rho)=(r\rho)^{-\frac{N-2}{2}}J_{\nu}(r\rho)+O\left((r\rho)^{-\frac{N+1}{2}}\right).
\end{align*}
This concludes the proof of Lemma \ref{lem:solution of spectrum}.
\end{proof}

\medskip
Now, we define the distorted Hankel transform $H: L^{2}(\mathbb{R}^{N})\rightarrow L^{2}(\mathbb{R}_{+})$ corresponding to the operator $\mathcal{L}_{W}^{a}$.
\begin{align}\label{def:distorted hankel transform}
(H\varphi){(\rho)}=\frac{1}{(2\pi)^{\frac{N}{2}}}\int_{0}^{\infty}\tilde{f}(r,\rho)\varphi(r)\ud r,\qquad\qquad \varphi\in L_{rad}^{2}(\mathbb{R}^{N}).
\end{align}
Then we prove an exterior energy bound for the linearized equation (\ref{eq:liearized wave}) around the ground state.

\medskip
\begin{theorem}\label{thm: linearized wave}
Assume $(N-2)\beta\geq1$ and $(N-2)\beta$ is an odd integer. Then there exists a constant $C>0$ and $R>0$, such that for the radial solution to \eqref{eq:liearized wave} with initial data $(f,g)\in(\dot{H_{a}^{1}}\times L^{2})(\mathbb{R}^N)$, it holds that
\begin{align*}
\sum\limits_{\pm}\lim\limits_{t\rightarrow\pm\infty}\int_{|x|\geq t+R}\big|\nabla_{t,x}u(t,x)\big|^{2}\ud x\geq C\big\|\operatorname{\pi}_{\mathbf{\mathcal{Z}}}^{\perp}(f,g)\big\|_{(\dot{H_{a}^{1}}\times L^2)(r\geq R,r^{N-1}\ud r)}.
\end{align*}
where the space $\mathbf{\mathcal{Z}}$ is
\begin{equation*}
\mathbf{\mathcal{Z}}=\left\{
\begin{aligned}
&\operatorname{span}\{(\Lambda W_{a},0)\}\qquad &\text{if}~(N-2)\beta=1\\
&\operatorname{span}\{(\Lambda W_{a},0),(0,\Lambda W_{a})\}\qquad &\text{if}~(N-2)\beta\geq 3
\end{aligned}
\right.
\end{equation*}
\end{theorem}
\begin{proof}
Taking the distorted Hankel transform of the both sides of \eqref{eq:liearized wave}
\begin{equation}\label{eq:linearized spectrum}
\left\{
\begin{aligned}
&\partial_{t}^{2}H u-\rho^2H u=0,\\
&H u(0)=Hf,\quad\partial_{t}(H u)(0)=Hg.
\end{aligned}
\right.
\end{equation}
Then the solution $H u(t,\rho)$ for \eqref{eq:linearized spectrum} with initial data $(Hf,Hg)$ is given by
\begin{align*}
H u(t,\rho)=\cos(t\rho)Hf(\rho)+\frac{\sin(t\rho)}{\rho}Hg{(\rho)}.
\end{align*}
Note that if $f$, $g$ are radial functions, then $Hf$ and $Hg$ are also radial. Therefore, from \eqref{def:distorted hankel transform}, we get
\begin{align*}
u(t,\rho)=&\frac{1}{(2\pi)^{\frac{N}{2}}}\int_{0}^{\infty}\tilde{f}(r,\rho)H u(\rho)\ud\rho\\
=&\frac{1}{(2\pi)^{\frac{N}{2}}}\int_{0}^{\infty}\tilde{f}(r,\rho)\Big(\cos(t\rho)Hf(\rho)+\frac{\sin(t\rho)}{\rho}Hg{(\rho)}\Big)\ud\rho,
\end{align*}
and
\begin{align*}
\partial_{t}u(r,\rho)=&\frac{1}{(2\pi)^{\frac{N}{2}}}\int_{0}^{\infty}\tilde{f}(r,\rho)\Big(-\rho\sin(t\rho)Hf(\rho)+\cos(t\rho)Hg{(\rho)}\Big)\ud\rho,
\end{align*}
where the standard asymptotic for $\tilde{f}(r,\rho)$ is given by \eqref{asymptotic of solution}.\par

Now we start to compute the asymptotic form of the exterior energy. On the one hand, by \cite{KLLS}
\begin{align}\label{eq:lower bdd of u}
&(2\pi)^{N}\Big(\|\partial_{t}u(t,r)\|_{L^2(r>R+t, r^{N-1})}+\|(\partial_{r}+\frac{c}{r})u(t,r)\|_{L^2(r>R+t, r^{N-1})}\Big)\nonumber\\
\geq&\frac{1}{2}\int_{R}^{\infty}\Big(\Big|(\partial_{r}+\frac{c}{r})f(r)\Big|^2+\left|g(r)\right|^2\Big)r^{N-1}\ud r-|f(R)|^2R^{N-2}.
\end{align}

On the other hand, recalling the definition of $\operatorname{\pi}_{\mathcal{Z}}^{\bot}$ \eqref{def: projection of initial f g}, we obtain
\begin{align*}
&\left\|\operatorname{\pi}_{\mathcal{Z}}^{\bot}(f,g)\right\|_{\dot{H}_{a}^{1}\times L^2(r\geq R,r^{N-1}\ud r)}\\
=&\bigg\|\frac{\int_{R}^{\infty}\left(\partial_{r}+\frac{c}{r}\right)f(r)\left(\partial_{r}+\frac{c}{r}\right)(\Lambda W_{a})r^{N-1}\ud r}{\int_{R}^{\infty}\left(\left(\partial_{r}+\frac{c}{r}\right)\Lambda W_{a}\right)^{2}r^{N-1}\ud r}(\Lambda W_{a},0)\bigg\|_{\dot{H}_{a}^{1}\times L^2(r\geq R,r^{N-1}\ud r)}\\
=&\frac{\left|\int_{R}^{\infty}\left(\partial_{r}+\frac{c}{r}\right)f(r)\left(\partial_{r}+\frac{c}{r}\right)(\Lambda W_{a})r^{N-1}\ud r\right|^{2}}{\|\Lambda W_{a}\|_{\dot{H}_{a}^{1}(r\geq R,r^{N-1}\ud r)}},
\end{align*}
where
\begin{align*}
\Lambda W_{a}&=r\cdot\nabla W_{a}+\frac{N-2}{2}W_{a}\\
&=\frac{N-2}{2}\beta(N(N-2)\beta^2)^{\frac{N-2}{4}}\Big(\frac{r^{\beta-1}}{1+r^{2\beta}}\Big)^{\frac{N-2}{2}}\frac{1-r^{2\beta}}{1+r^{2\beta}},
\end{align*}
and
\begin{align*}
\partial_{r}(\Lambda W_{a}r^{c})=&\frac{N-2}{2}(N(N-2)\beta^2)^{\frac{N-2}{4}}\beta\Big(\frac{r^{\beta-1}}{1+r^{2\beta}}\Big)^{\frac{N}{2}+1}\\
&\Big[\big((N-2)(\beta-1)+2c\big)r^{c-2\beta+1}+\big(-(N-2)2\beta+2(c-4\beta)\big)r^{c+1}\\
&+\big((N-2)(1+\beta)-2(c+4\beta)\big)r^{c+2\beta+1}-2c r^{c+4\beta+1}\Big].
\end{align*}
Hence, for every $R>0$, direct computation shows
\begin{align}\label{eq:projection on lambda W}
&\|\Lambda W_{a}\|_{\dot{H}_{a}^{1}(r\geq R,r^{N-1}\ud r)}^2=\int_{R}^{\infty}\Big(\big(\partial_{r}+\frac{c}{r}\big)(\Lambda W_{a})\Big)^2r^{N-1}\ud r\nonumber\\
=&(\frac{N-2}{2})^2(N(N-2)\beta^2)^{\frac{N-2}{2}}\beta^2\int_{R}^{\infty}\Big|\Big(\frac{r^{\beta-1}}{1+r^{2\beta}}\Big)^{\frac{N+2}{2}}\nonumber\\
&\Big[\big((N-2)(\beta-1)+2c\big)r^{c-2\beta+1}+\big(-2(N-1)\beta+2(c-4\beta)\big)r^{c+1}\nonumber\\
&+\big((N-2)(1+\beta)-2(c+4\beta)\big)r^{c+2\beta+1}-2c r^{c+4\beta+1}\Big]\Big|^2r^{N-1-2c}\ud r\nonumber\\
\leq& CR^{-(N-2)\beta}.
\end{align}
And for sufficiently large $R>0$, we estimate by Taylor expansion,
\begin{align}\label{eq:projection on f}
&\int_{R}^{\infty}\Big(\partial_{r}+\frac{c}{r}\Big)f(r)\Big(\partial_{r}+\frac{c}{r}\Big)(\Lambda W_{a})r^{N-1}\ud r\nonumber\\
=&\frac{N-2}{2}(N(N-2)\beta^2)^{\frac{N-2}{4}}\beta\nonumber\\
&\times\Big[\big((N-2)(\beta-1)+2c\big)\int_{R}^{\infty}\partial_{r}(f(r)r^{c})(1+r^{2\beta})^{-\frac{N+2}{2}}r^{\frac{(\beta+1)N}{2}-\beta-1-c}\ud r\nonumber\\
&-\big(-2(N-2)\beta+2(c-4\beta)\big)\int_{R}^{\infty}\partial_{r}(f(r)r^{c})(1+r^{2\beta})^{-\frac{N+2}{2}}r^{\frac{(\beta+1)N}{2}-\beta-3-c}\ud r\nonumber\\
&+\big((N-2)(1+\beta)-2(c+4\beta)\big)\int_{R}^{\infty}\partial_{r}(f(r)r^{c})(1+r^{2\beta})^{-\frac{N+2}{2}}r^{\frac{(\beta+1)N}{2}+3\beta-3-c}\ud r\nonumber\\
&-2c\int_{R}^{\infty}\partial_{r}(f(r)r^{c})(1+r^{2\beta})^{-\frac{N+2}{2}}r^{\frac{(\beta+1)N}{2}+5\beta-1-c}\ud r\Big]\nonumber\\
\leq& C_1f(R)R^{c}+C_2R^{\frac{-(N-2)\beta}{2}}\Big(\int_{R}^{\infty}|\partial_{r}(f(r)r^c)|^2r^{N-1-2c}\ud r\Big)^{\frac{1}{2}}+o(1).
\end{align}
Thus, \eqref{eq:projection on lambda W} and \eqref{eq:projection on f} yield
\begin{align*}
&\|\operatorname{\pi}_{\mathcal{Z}}^{\bot}(f,g)\|_{\dot{H}_{a}^{1}\times L^{2}(r\geq R, r^{N-1})}^2\\
\leq&\|(f,g)\|_{\dot{H}_{a}^{1}\times L^{2}(r\geq R, r^{N-1})}^2-CR^{(N-2)\beta}\Big(C_1|f(R)R^{c}|^2\\
&+C_{2}R^{-(N-2)\beta}\int_{R}^{\infty}|(\partial_{r}+\frac{c}{r})f(r)|r^{N-1}\ud r\Big)\\
\leq&\|(f,g)\|_{\dot{H}_{a}^{1}\times L^{2}(r\geq R, r^{N-1})}^2-C_1|f(R)|^2R^{N-2}+C_{2}\int_{R}^{\infty}|(\partial_{r}+\frac{c}{r})f(r)|r^{N-1}\ud r\\
\leq& C\Big(\|(f,g)\|_{\dot{H}_{a}^{1}\times L^{2}(r\geq R, r^{N-1})}^2-|f(R)|^2R^{N-2}\Big).
\end{align*}
This estimate, combined with \eqref{eq:lower bdd of u}, implies
\begin{align*}
\|\operatorname{\pi}_{\mathcal{Z}}^{\bot}(f,g)\|_{\dot{H}_{a}^{1}\times L^{2}(r\geq R, r^{N-1})}^2\leq& C\max\limits_{\pm}\lim\limits_{t\rightarrow\pm\infty}\bigg(\|\partial_{t}u(t,r)\|_{L^2(r>R+t, r^{N-1})}\\
&+\big\|(\partial_{r}+\frac{c}{r})u(t,r)\big\|_{L^2(r>R+t, r^{N-1})}\bigg).
\end{align*}
Thus the proof is completed.
\end{proof}

\subsection{The exterior energy estimate close to a multi-soliton}

As a corollary of Theorem~\ref{thm:weq}, we will prove an exterior energy lower bound for the linearized operator close to a radial multi-soliton solution. First of all, we prove some estimates on space-time norms.

\begin{lemma}\label{ineq: theta N}
Assume $(N-2)\beta\geq 1$ and $0<\lambda<\mu$. Then it holds that
\begin{align}\label{estimate W, Lambda W}
&\Big\|\mathbf{1}_{\{|x|\geq|t|\}}{W_a}_{(\lambda)}^{\frac{4}{N-2}}{(\Lambda W_a)}_{(\mu)}\Big\|_{L^{1}\left(\mathbb{R}, L^{2}\right)}+\Big\|t\mathbf{1}_{\{|x|\geq|t|\}}{W_a}_{(\lambda)}^{\frac{4}{N-2}}{(\Lambda W_a)}_{[\mu]}\Big\|_{L^{1}\left(\mathbb{R}, L^{2}\right)}\nonumber\\
\lesssim&\left\{
\begin{aligned}
\left(\frac{\lambda}{\mu}\right)^{\frac{(N-2)\beta}{2}}, \qquad\qquad &\text{if}~(N-6)\beta\leq 0,\\
\left(\frac{\lambda}{\mu}\right)^{2\beta}, \qquad\qquad\qquad &\text{if}~(N-6)\beta>0.
\end{aligned}
\right.
\end{align}
\end{lemma}
\begin{proof}
Due to \eqref{asympototic1}, we have
\begin{align}\label{eq: Wa asymptotic}
|\Lambda W_{a}(x)|+|W_{a}(x)|\lesssim\min\left\{|x|^{\frac{(N-2)(\beta-1)}{2}},|x|^{-\frac{(N-2)(1+\beta)}{2}}\right\}.
\end{align}
By symmetry, it suffices to prove the $t\geq0$ case. Without loss of generality, we assume $\mu=1$. Now we estimate the norms by dividing the domain of integration with respect to $r$ into three parts: $(0,\lambda)$, $(\lambda,1)$ and $(1,\infty)$.

In order to deal with the case $\Big\|\mathbf{1}_{\{|x|\geq|t|\}}{W_a}_{(\lambda)}^{\frac{4}{N-2}}{(\Lambda W_a)}_{(\mu)}\Big\|_{L^{1}\left(\mathbb{R}, L^{2}\right)}$, by \eqref{eq: Wa asymptotic}, we define
\begin{align*}
f(r,t)=&\lambda^{-2}W_{a}^{\frac{4}{N-2}}\Big(\frac{x}{\lambda}\Big)W_{a}(x)\\
=&\lambda^{-2}\min\left\{\left(\frac{r}{\lambda}\right)^{2(\beta-1)},\left(\frac{\lambda}{r}\right)^{2(\beta+1)}\right\}\min\left\{r^{\frac{(N-2)(\beta-1)}{2}},r^{-\frac{(N-2)(\beta+1)}{2}}\right\}.
\end{align*}
Then it suffices to estimate
\begin{align*}
\int_{0}^{\infty}&\Big(\int_{t}^{\infty}f(r,t)r^{N-1}\ud r\Big)^{\frac{1}{2}}\ud t\lesssim\int_{0}^{\lambda}\Big(\int_{t}^{\lambda}f(r,t)r^{N-1}\ud r\Big)^{\frac{1}{2}}\ud t\\
&+\int_{0}^{1}\Big(\int_{\max\{t,\lambda\}}^{1}f(r,t)r^{N-1}\ud r\Big)^{\frac{1}{2}}\ud t+\int_{0}^{\infty}\Big(\int_{\max\{t,1\}}^{\infty}f(r,t)r^{N-1}\ud r\Big)^{\frac{1}{2}}\ud t\\
\lesssim&\lambda^{-2\beta}\int_{0}^{\lambda}\Big(\int_{t}^{\lambda}r^{\beta(N+2)-3}\ud r\Big)^{\frac{1}{2}}\ud t+\lambda^{2\beta}\int_{0}^{1}\Big(\int_{\max\{t,\lambda\}}^{1}r^{(N-6)\beta-3}\ud r\Big)^{\frac{1}{2}}\ud t\\
&+\lambda^{2\beta}\int_{0}^{\infty}\Big(\int_{\max\{t,1\}}^{\infty}r^{-\beta(N+2)-3}\ud r\Big)^{\frac{1}{2}}\ud t\\
=&{\rm\uppercase\expandafter{\romannumeral1}}+{\rm\uppercase\expandafter{\romannumeral2}}+{\rm\uppercase\expandafter{\romannumeral3}}.
\end{align*}
We estimate these three terms, respectively. First,
\begin{align}\label{eq: bdd of (1)}
{\rm\uppercase\expandafter{\romannumeral1}}\lesssim\lambda^{-2\beta}\lambda^{\frac{(N+2)\beta}{2}}\lesssim\lambda^{\frac{(N-2)\beta}{2}}.
\end{align}

Next we estimate ${\rm\uppercase\expandafter{\romannumeral2}}$. If $(N-6)\beta<3$, direct computation shows
\begin{align}\label{eq: bdd of (2.1)}
{\rm\uppercase\expandafter{\romannumeral2}}=&\lambda^{2\beta}\int_{0}^{1}\Big(\int_{\max\{t,\lambda\}}^{1}r^{(N-6)\beta-3}\ud r\Big)^{\frac{1}{2}}\ud t\nonumber\\
=&\lambda^{2\beta}\int_{0}^{\lambda}\Big(\int_{\lambda}^{1}r^{(N-6)\beta-3}\ud r\Big)^{\frac{1}{2}}\ud t+\lambda^{2\beta}\int_{\lambda}^{1}\Big(\int_{t}^{1}r^{(N-6)\beta-3}\ud r\Big)^{\frac{1}{2}}\ud t\nonumber\\
\lesssim&\lambda^{2\beta}\int_{0}^{\lambda}\lambda^{\frac{(N-6)\beta-2}{2}}\ud t+\lambda^{2\beta}\int_{\lambda}^{1}t^{\frac{(N-6)\beta-2}{2}}\ud t\nonumber\\
\lesssim&
\begin{cases}\lambda^{\frac{(N-2)\beta}{2}}, & \text {if}\quad(N-6)\beta\leq 0,\\ \lambda^{2\beta}, & \text {if}\quad(N-6)\beta> 0.\end{cases}
\end{align}
Furthermore, if $(N-6)\beta\geq3$, then
\begin{align}\label{eq: bdd of (2.2)}
{\rm\uppercase\expandafter{\romannumeral2}}=&\lambda^{2\beta}\int_{0}^{\lambda}\Big(\int_{\lambda}^{1}r^{(N-6)\beta-3}\ud r\Big)^{\frac{1}{2}}\ud t+\lambda^{2\beta}\int_{\lambda}^{1}\Big(\int_{t}^{1}r^{(N-6)\beta-3}\ud r\Big)^{\frac{1}{2}}\ud t\nonumber\\
\lesssim&\lambda^{2\beta}.
\end{align}

Finally, we give the bound on ${\rm\uppercase\expandafter{\romannumeral3}}$.
\begin{align}\label{eq: bdd of (3)}
{\rm\uppercase\expandafter{\romannumeral3}}=&\lambda^{2\beta}\int_{0}^{\infty}\Big(\int_{\max\{t,1\}}^{\infty}r^{-\beta(N+2)-3}\ud r\Big)^{\frac{1}{2}}\ud t\nonumber\\
=&\lambda^{2\beta}\int_{0}^{1}\Big(\int_{1}^{\infty}r^{-\beta(N+2)-3}\ud r\Big)^{\frac{1}{2}}\ud t+\lambda^{2\beta}\int_{1}^{\infty}\Big(\int_{1}^{\infty}r^{-\beta(N+2)-3}\ud r\Big)^{\frac{1}{2}}\ud t\nonumber\\
\lesssim&\lambda^{2\beta}.
\end{align}

For $(N-6)\beta\leq 0$, we have $\frac{\beta(N-2)}{2}\leq 2\beta$, then $\lambda^{2\beta}\leq\lambda^{\frac{\beta(N-2)}{2}}$. Combining \eqref{eq: bdd of (1)}, \eqref{eq: bdd of (2.1)}, \eqref{eq: bdd of (2.2)} and \eqref{eq: bdd of (3)}, we get
\begin{align}\label{eq: bdd of first}
\Big\|\mathbf{1}_{\{|x|\geq|t|\}}{W_a}_{(\lambda)}^{\frac{4}{N-2}}{(\Lambda W_a)}_{(\mu)}\Big\|_{L^{1}\left(\mathbb{R}, L^{2}\right)}\lesssim\begin{cases}\big(\frac{\lambda}{\mu}\big)^{\frac{(N-2)\beta}{2}}, & \text {if}\quad(N-6)\beta\leq 0,\\\big(\frac{\lambda}{\mu}\big)^{2\beta}, & \text {if}\quad(N-6)\beta> 0.\end{cases}
\end{align}
By analoguous arguments we obtain
\begin{align}\label{eq: bdd of second}
\Big\|t\mathbf{1}_{\{|x|\geq|t|\}}{W_a}_{(\lambda)}^{\frac{4}{N-2}}{(\Lambda W_a)}_{[\mu]}\Big\|_{L^{1}\left(\mathbb{R}, L^{2}\right)}\lesssim\begin{cases}\big(\frac{\lambda}{\mu}\big)^{\frac{(N-2)\beta+2}{2}}, & \text {if}\quad(N-6)\beta\leq -2,\\\big(\frac{\lambda}{\mu}\big)^{2\beta}, & \text {if}\quad(N-6)\beta>-2.\end{cases}
\end{align}
Due to \eqref{eq: bdd of first} and \eqref{eq: bdd of second}, for $-2\leq (N-6)\beta\leq 0$, we infer that $\lambda^{2\beta}\leq\lambda^{\frac{(N-2)\beta}{2}}$, and \eqref{estimate W, Lambda W} is obtained.
\end{proof}

\medskip
We will also use the following asymptotic property for the solution of the linear wave equation with inverse-square potential. Firstly, we introduce a function space on $\mathbb{R}\times S^{N-1}$:
\begin{align*}
\dot{H}_{\eta}^{1}(\mathbb{R}\times S^{N-1})=\left\{g\in C^{0}(\mathbb{R},L^2(S^{N-1})):\int_{\mathbb{R}\times S^{N-1}}|\partial_{\eta}g(\eta,\omega)|^2\ud \eta\ud \omega<\infty\right\}.
\end{align*}
\begin{definition}
Define $\dot{\mathcal{H}}_{\eta}^{1}$ as the quotient space of $\dot{H}_{\eta}^{1}$ by the equivalence relation:
\begin{align*}
g\sim\tilde{g}\Leftrightarrow\exists a(\omega)\in L^{2}(S^{N-1}):\forall\eta\in\mathbb{R}, g(\eta,\omega)-\tilde{g}(\eta,\omega)=a(\omega)\quad for\quad a.e.\quad\omega\in S^{N-1}.
\end{align*}
and denote by $\bar{g}\in\dot{\mathcal{H}_{\eta}^{1}}$ the equivalence class of $g\in\dot{H}_{\eta}^{1}$ with norm on $\dot{\mathcal{H}}_{\eta}^{1}$:
\begin{align*}
\|\bar{g}\|_{\dot{\mathcal{H}_{\eta}^{1}}(\mathbb{R}\times S^{N-1})}=\|\partial_{\eta}g\|_{L^2(\mathbb{R}\times S^{N-1})}.
\end{align*}
\end{definition}

\medskip
\begin{theorem}\label{thm: asymptotic liear wave}
Assume $(N-2)\beta\geq 1$ and let $u_L$ be a solution of the linear wave equation with inverse-square potential~\eqref{lwp}. Then we have
\begin{align*}
\lim\limits_{t\rightarrow+\infty}\Big\|\frac{1}{r}\nabla_{\omega}u_{L}(t)\Big\|_{L^2(\mathbb{R}\times S^{N-1})}+\Big\|\frac{1}{r}u_L(t)\Big\|_{L^2(\mathbb{R}\times S^{N-1})}=0,
\end{align*}
and there exists a unique $G_{+}\in L^{2}\left(\mathbb{R}\times S^{N-1}\right)$ such that
\begin{align*}
\lim\limits_{t\rightarrow+\infty}\int_{0}^{+\infty}\int_{S^{N-1}}\left|r^{\frac{N-1}{2}}\partial_{r}u_{L}(t,r\omega)-G_{+}(r-t,\omega)\right|^{2}\ud \omega\ud r&=0,\\
\lim\limits_{t\rightarrow+\infty}\int_{0}^{+\infty}\int_{S^{N-1}}\left|r^{\frac{N-1}{2}}\partial_{r}u_{L}(t,r\omega)+G_{+}(r-t,\omega)\right|^{2}\ud \omega\ud r&=0.
\end{align*}
Furthermore,
\begin{align*}
E_{L}(f,g)=\int_{\mathbb{R}\times S^{N-1}}\left|G_{+}(\eta,\omega)\right|^{2}\ud \eta\ud\omega=\left\|G_{+}\right\|_{L^2}^{2},
\end{align*}
and the map
\begin{align*}
(f,g)\mapsto&\sqrt{2}G_{+}\\
(\dot{H}_{a}^{1}\times L^{2})(\mathbb{R}^{N})\rightarrow&L^{2}(\mathbb{R}\times S^{N-1})
\end{align*}
is a bijective isometry.
\end{theorem}

\medskip
Before proving this theorem, we need some preparations. We first prove the result for smooth, compactly supported functions.
\begin{lemma}\label{lem 4.6}
Let $(N-2)\beta\geq1$, $(v_0,v_1)\in(C_{0}^{\infty}(\mathbb{R}^N))^2$ and let $v$ be the solution to the wave equation \eqref{lwp}. Then $\|v(t)\|_{L^2}$ is bounded and there exists $R>0$ such that
\begin{align}\label{lem3.6 ineq1}
\lim\limits_{t\rightarrow\infty}\int_{\Bbb R^N}\frac{1}{|x|^2}|v(t,x)|^2\ud x=0,
\end{align}
and
\begin{align}\label{lem3.6 ineq2}
\lim\limits_{R\rightarrow\infty}\limsup\limits_{t\rightarrow\infty}\int_{|x|\leq |t|-R}\Big|(\partial_{r}+\frac{c}{r})v(t,x)\Big|^2+|\partial_{t}v(t,x)|^2\ud x=0.
\end{align}
\end{lemma}
\begin{proof}
The proof of \eqref{lem3.6 ineq1} is just a repetition of Proposition $A.5$ in \cite{CKS}. So it remains to prove \eqref{lem3.6 ineq2}
\begin{align*}
&\frac{\ud }{\ud t}\left(\int_{\Bbb R^N} x\cdot\Big(\partial_{r}+\frac{c}{r}\Big)v\partial_{t}v+\frac{N-1}{2}\int_{\Bbb R^N} v\partial_{t}v\right)\\
=&\int_{\Bbb R^N} x\cdot\Big(\partial_{r}+\frac{c}{r}\Big)v_{t}v_t+x\cdot\Big(\partial_{r}+\frac{c}{r}\Big)v\cdot\left(\Delta+\frac{a}{r^2}\right)v+\frac{N-1}{2}\int_{\Bbb R^N}\left(\partial_{t}v\right)^2+v\cdot\Big(\Delta-\frac{a}{r^2}\Big)v\\
=&{\rm\uppercase\expandafter{\romannumeral1}}+{\rm\uppercase\expandafter{\romannumeral2}}+{\rm\uppercase\expandafter{\romannumeral3}}+{\rm\uppercase\expandafter{\romannumeral4}},
\end{align*}
where
\begin{align*}
{\rm\uppercase\expandafter{\romannumeral1}}=&\int_0^{\infty} r\Big(\partial_{r}+\frac{c}{r}\Big)v_{t}\cdot v_t r^{N-1}\ud r=\int_0^\infty v_{t}\partial_{r}(v_{t}r^{c})r^{N-c}\ud r\\
=&-\frac{N-2c}{2}\int\limits_0^\infty (\partial_{t}v)^{2}r^{N-1}\ud r,
\end{align*}
\begin{align*}
{\rm\uppercase\expandafter{\romannumeral2}}=&\int_0^{\infty} r\left(\partial_{r}+\frac{c}{r}\right)v\cdot\left(\Delta-\frac{a}{r^2}\right)v r^{N-1}\ud r=\int_0^{\infty}\partial_{r}(v r^{c})\cdot\Big(\partial_{r}+\frac{N-1}{r}-\frac{c}{r}\Big)\left(\partial_{r}+\frac{c}{r}\right)v r^{N-c}\ud r\\
=&\int_0^{\infty}\partial_{r}(v r^{c})\cdot\partial_{r}\left(\partial_{r}(v r^{c})r^{-2c}\right)r^{N}\ud r+\int_0^{\infty}\frac{N-1}{r}\cdot(\partial_{r}(v r^{c}))^2 r^{N-2c}\ud r\\
=&\int_0^{\infty}\partial_{r}(v r^{c})\Big(\partial_{r}+\frac{N-1}{r}\Big)\big(\partial_{r}(v
r^{c})\big)r^{N-2c}\ud r-2c\int_0^{\infty}\left(\partial_{r}(v r^{c})\right)^2 r^{N-1-2c}\ud r\\
=&\frac{1}{2}\int_0^{\infty}\partial_{r}\Big((\partial_{r}(v r^{c})r^{N-1})^2\Big)r^{-N-2c+2}\ud r-2c\int_0^{\infty}\int_0^{\infty}\left(\partial_{r}(v r^{c})\right)^2 r^{N-1-2c}\ud r\\
=&\frac{N-2-2c}{2}\int_0^{\infty}\big(\partial_{r}(v r^c)\big)^2 r^{N-1-2c}\ud r,
\end{align*}
and
\begin{align*}
{\rm\uppercase\expandafter{\romannumeral4}}=&\int_0^\infty v\cdot\left(\Delta-\frac{a}{r^2}\right)v r^{N-1}\ud r=\int_0^\infty v\partial_{r}\left(\partial_{r}(v r^c)r^{-2c}\right)r^{N-1+c}\ud r+\int_0^\infty v\frac{N-1}{r}\partial_{r}(v r^c)r^{N-1-c}\ud r\\
=&\int_0^\infty v\Big(\partial_{r}+\frac{N-1}{r}\Big)\partial_{r}(v r^c)r^{N-1-c}\ud r-2c\int_0^\infty v\partial_{r}(v r^c)r^{N-2-c}\ud r\\
=&\int_0^\infty v r^c\partial_{r}\Big(\partial_{r}(v r^c)r^{N-1}\Big)r^{-2c}\ud r-2c\int_0^\infty v\partial_{r}(v r^c)r^{N-2-c}\ud r\\
=&-\int_0^\infty \big(\partial_{r}(v r^c)\big)^2r^{N-1-2c}\ud r.
\end{align*}
Since note $c=\frac{(N-2)(1-\beta)}{2}<0$, we have
\begin{align*}
&\frac{\ud }{\ud t}\left(\int_{\Bbb R^N} x\cdot\left(\partial_{r}+\frac{c}{r}\right)v\partial_{t}v+\frac{N-1}{2}\int_{\Bbb R^N} v\partial_{t}v\right)\\
=&-\frac{N+2c}{2}\int_0^\infty(\partial_{t}v)^2r^{N-1}\ud r+\frac{N-2-2c}{2}\int_0^\infty\big(\partial_{r}(v r^c)\big)^2r^{N-1-2c}\ud r\\
&+\frac{N-1}{2}\int_0^\infty(\partial_{t}v)^2r^{N-1}\ud r-\frac{N-1}{2}\int_0^\infty\big(\partial_{r}(v r^c)\big)^2r^{N-1-2c}\ud r\\
=&-E_{a}(v_0,v_1)-2c E_{a}(v_0,v_1)\geq-E_{a}(v_0,v_1).
\end{align*}
Integrate from $0$ to $t$, yields
\begin{align*}
&\int_{\Bbb R^N} x\left(\partial_{r}+\frac{c}{r}\right)v\cdot \partial_{t}v\ud x+\frac{N-1}{2}\int_{\Bbb R^N} v\partial_{t}v\ud x\\
\geq&-t E_{a}(v_0,v_1)+\int_{\Bbb R^N} x\left(\partial_{r}+\frac{c}{r}\right)v_0\cdot v_1\ud x+\frac{N-1}{2}\int_{\Bbb R^N} v_0v_1\ud x.
\end{align*}
Since $(N-2)\beta\geq1$, $\dot{H}_{a}^{-1}$ is a Hilbert space, and $C_{0}^{\infty}(\mathbb{R}^{N})\subset\dot{H}_{a}^{-1}(\mathbb{R}^{N})$. By conservation of the $L^{2}\times\dot{H}_{a}^{-1}$ norm of $(v,\partial_{t}v)$,
\begin{align}\label{v bound}
\left|\int_{\Bbb R^N} x\cdot\left(\partial_{r}+\frac{c}{r}\right)v\cdot\partial_{t}v \ud x\right|\geq t E_{a}(v_0,v_1)-C(v_0,v_1).
\end{align}
Let $M>0$ be such that $|x|\leq M$ on the support of $(v_0,v_1)$. Then $|x|\leq M+|t|$ on the support of $(v(t),\partial_{t}v(t))$. By the finite speed of propagation and Cauchy-Schwarz inequality
\begin{align*}
&\left|\int_{\Bbb R^N} x\cdot\left(\partial_{r}+\frac{c}{r}\right)v\cdot\partial_{t}v\ud x\right|\\
\leq&(t-R)\int_{|x|\leq t-R}\Big|\left(\partial_{r}+\frac{c}{r}\right)v\cdot\partial_{t}v\Big|\ud x
+(t+M)\int_{|x|\geq t-R}\Big|\left(\partial_{r}+\frac{c}{r}\right)v\cdot\partial_{t}v\Big|\ud x\\
\leq&(t+M)E_{a}(v_0,v_1)-\frac{R}{2}\int_{|x|\leq t-R}|\nabla_{t,x}v|^2\ud x,
\end{align*}
then combining with \eqref{v bound}, yields
\begin{align*}
\frac{R}{2}\int_{|x|\leq t-R}|\nabla_{t,x}v|^2\ud x\leq C+ME_{a}(v_0,v_1),
\end{align*}
which implies
\begin{align*}
\lim\limits_{R\rightarrow\infty}\limsup\limits_{t\rightarrow\infty}\int_{|x|\leq t-R}|\nabla_{t,x}v|^2\ud x=0.
\end{align*}
\end{proof}

\medskip
As a consequence of Lemma \ref{lem 4.6}, as in \cite{CKS} we get
\begin{lemma}
Let $(N-2)\beta\geq 1$, $(v_0,v_1)\in(C_{0}^{\infty}(\mathbb{R}^N))^2$. Then
\begin{align}\label{eq:1/r v, L2}
\lim\limits_{t\rightarrow\infty}\Big\|\frac{1}{r}\nabla_{\omega}r(t)\Big\|_{L^{2}}+\Big\|\frac{1}{r}v(t)\Big\|_{L^2}=0,
\end{align}
and there exists $g\in\dot{H}_{\eta}^{1}$ such that
\begin{align}\label{eq: relatioinship between v and g}
\lim\limits_{t\rightarrow\infty}\left\|\nabla_{r,t}\Big(r^{\frac{N-1}{2}}v(t,r\omega)-g(r-t,\omega)\Big)\right\|_{L^2((0,\infty)\times S^{N-1})}=0.
\end{align}
\end{lemma}

\medskip
\begin{lemma}
Let $(v_0,v_1)$ and $g$ be as in above. Then
\begin{align*}
E(v_0,v_1)=\|\bar{g}\|_{\dot{H}_{\eta}^1(\mathbb{R}\times S^{N-1})}^2,
\end{align*}
where $$\|\bar{g}\|_{\dot{H}_{\eta}^1(\mathbb{R}\times S^{N-1})}^2=\int_{-\infty}^{\infty}\int_{S^{N-1}}\big|\partial_{\eta}g(\eta,\omega)\big|^2\ud\omega\ud\eta.$$
\end{lemma}
\begin{proof}
The energy of the equation \eqref{lwp} is given by
\begin{align*}
E_{a}(v,\partial_{t}v)=\frac{1}{2}\int_{0}^{\infty}\int_{S^{N-1}}\Big[\Big((\partial_{r}+\frac{c}{r})v\Big)^2+(\partial_{t}v)^2+\frac{1}{r^2}|\nabla_{\omega}v|^2+\frac{a}{r^{2}}|v|^2\Big]r^{N-1}\ud r.
\end{align*}
Relying on \eqref{eq:1/r v, L2}, we infer that
\begin{align*}
&\int_{0}^{\infty}\int_{S^{N-1}}\left((\partial_{r}+\frac{c}{r})v\right)^2r^{N-1}\ud \omega\ud r=\int_{0}^{\infty}\int_{S^{N-1}}\Big(\partial_{r}(r^{\frac{N-1}{2}}v)-(\frac{N-1}{2}-c)r^{\frac{N-3}{2}}v\Big)^2\ud r\\
=&\int_{0}^{\infty}\int_{S^{N-1}}\left(\partial_{r}(r^{\frac{N-1}{2}}v)\right)^2\ud r+o(1).
\end{align*}
Inserting \eqref{eq: relatioinship between v and g} into the expression of the energy, as $t\rightarrow\infty$,
\begin{align*}
E_{a}(v,\partial_{t}v)=&\frac{1}{2}\int_{0}^{\infty}\int_{S^{N-1}}\left(\partial_{r}(r^{\frac{N-1}{2}}v)\right)^2\ud \omega\ud r+\frac{1}{2}\int_{0}^{\infty}\int_{S^{N-1}}(\partial_{t}v)^2r^{N-1}\ud\omega\ud r+o(1)\\
=&\frac{1}{2}\int_{0}^{\infty}\int_{S^{N-1}}\big|\partial_{\eta}g(r-t,\omega)\big|^2\ud\omega\ud r+o(1)\\
=&\frac{1}{2}\int_{-t}^{\infty}\int_{S^{N-1}}\big|\partial_{\eta}g(\eta,\omega)\big|^2\ud\omega\ud\eta+o(1)\\
\rightarrow&\int_{-\infty}^{\infty}\int_{S^{N-1}}\big|\partial_{\eta}g(\eta,\omega)\big|^2\ud\omega\ud\eta.
\end{align*}
In view of the conservation law of the energy, implies
\begin{align*}
E_{a}(v_0,v_1)=\|\bar{g}\|_{\dot{H}_{\eta}^1(\mathbb{R}\times S^{N-1})}^2.
\end{align*}
\end{proof}

The proof of the general case and  the surjectivity follows from \cite{CKS}.

\medskip
Finally, let us introduce the energy lower bound for the linearized operator close to an approximate radial multisoliton solution. For a sequence of monotonous parameters $\{\lambda_j\}$, we denote $G_J$ by a subset of $(0,\infty)^{J}$ such that
\begin{align*}
G_{J}=\Big\{\boldsymbol{\lambda}=(\lambda_{j})_{1\leq j\leq J},0<\lambda_{J}<\lambda_{J-1}<\ldots<\lambda_1\Big\}.
\end{align*}
If $\boldsymbol{\lambda}\in G_{J}$, we denote
$$
\begin{gathered}
\gamma(\boldsymbol{\lambda})=\max_{2\leq j\leq J}\frac{\lambda_{j}}{\lambda_{j-1}}\in(0,1),\\
\mathcal{L}_{\boldsymbol{\lambda}}^{a}=-\Delta-\sum_{j=1}^{J}\frac{N+2}{N-2}W_{a\left(\lambda_{j}\right)}^{\frac{4}{N-2}}+\frac{a}{|x|^2}, \\
Z_{\boldsymbol{\lambda}}=\operatorname{span}\Big\{\left((\Lambda W_a)_{\left(\lambda_{j}\right)},0\right),\left(0,(\Lambda W_a)_{\left[\lambda_{j}\right]}\right)\Big\},
\end{gathered}
$$
see \eqref{def: scalling} for the notations $(\Lambda W)_{\left(\lambda_{j}\right)}$ and $(\Lambda W)_{\left[\lambda_{j}\right]}$. Going back to the Lemma \ref{ineq: theta N} and using the asymptotic property for the solution of the linear wave equation with inverse-square potential, we prove the following Theorem \ref{thm:multisoliton energy} as in \cite{DKM7}

\medskip
\begin{theorem}\label{thm:multisoliton energy}\emph{(\cite{DKM7})}
Assume $(N-2)\beta\geq 1$ and $(N-2)\beta$ is odd. For any $J \geq 2$, there exists $\gamma_{*}>0$ and $C>0$ with the following property. For any $\boldsymbol{\lambda}$ with $\gamma(\boldsymbol{\lambda}) \leq \gamma_{*}$, for any solution $u$ of
$$
\partial_{t}^{2}u+\mathcal{L}_{\boldsymbol{\lambda}}^{a}u=0, \quad
\vec{u}\mid_{ t=0}=\left(f, g\right) \in \mathcal{H}_{a}.
$$
one has
\begin{align*}\label{ineq: multisoliton energy}
\quad&\left\|\operatorname{\pi}_{Z_{\lambda}^{\perp}}\left(f,g\right)\right\|_{\mathcal{H}_{a}}^{2}\nonumber\\
&\leq C\Big(\sum_{\pm}\lim\limits_{t\rightarrow\pm\infty}\int_{|x|\geq|t|}\left|\nabla_{t, x}u(t, x)\right|^{2} \ud x+\gamma(\boldsymbol{\lambda})^{2 \theta_{N}}\left\|\pi_{Z_{\lambda}}\left(f,g\right)\right\|_{\mathcal{H}_a}^{2}\Big),
\end{align*}
where
\begin{align*}
\theta_{N}=\left\{
\begin{aligned}
\frac{(N-2)\beta}{2},\qquad\qquad &\text{if}~(N-6)\beta\leq 0\\
2\beta \qquad\qquad\qquad &\text{if}~(N-6)\beta>0
\end{aligned}
\right.
\end{align*}
\end{theorem}

\section{Soliton resolution along a time sequence}\label{Soliton resolution along a time sequence}

In this section we will establish the soliton resolution along a time sequence. Because the channel of energy for wave equation with inverse square potential holds in any dimension, the proof is in the same spirit of \cite{CR}.
\begin{theorem}\label{timesequence}
	Let $u$ be a radial solution to \eqref{lwp} with $T_+<\infty$ that satisfies \eqref{bh}. Then there exists $(v_0,v_1)\in \dot{H}^1_a\times L^2$, a sequence $T_+(u)$, an integer $J_0\geq1$, $J_0$ sequences of positive numbers $\{\lambda_{j,n}\}_n$, $j=1,\cdots, J_0$, $J_0$ signs $\iota_j\in\{\pm 1\}$ such that\begin{equation*}
		\lambda_{1,n}\ll \lambda_{2,n}\ll \cdots \ll \lambda_{J_0, n}\ll T_+(u)-t_n
	\end{equation*}
	and \begin{equation*}
		\vec{u}(t_n)=(v_0, v_1)+\sum_{j=1}^{J_0}\bigg(\frac{\iota_j}{\lambda_{j,n}^\frac{N-2}{2}}W_a\left(\frac{x}{\lambda_{j,n}}, 0\right)\bigg)+o(1)\ \ \mathrm{in}\ \ \dot{H}^1_a\times L^2\ \ \mathrm{as}\ n\rightarrow\infty.
	\end{equation*}
\end{theorem}

\begin{lemma}
	Let $\{u_{0,n}, u_{1,n}\}$, $\{U_L^j\}$, $\{t_{j,n}, \lambda_{j,n}\}$ and $\theta_n$ satisfy the assumptions of Subsection \ref{profile}. Let $\{\rho_n\}$, $\{\sigma_n\}$ be sequences with $0\leq \rho_n\leq \sigma_n\leq+\infty$. Then for $j\neq k$, we have
	\begin{equation*}
		\lim_{n\rightarrow\infty}\int_{\rho_n\leq |x|\leq \sigma_n}\nabla_{t,x}U_n^j(\theta_n,x)\cdot\nabla_{t,x}U_n^k(\theta_n,x)\ud x=0.
	\end{equation*}
	For $j\leq J$, it holds that
	\begin{equation*}
		\lim_{n\rightarrow\infty}\int_{\rho_n\leq |x|\leq \sigma_n}\nabla_{t,x}U_n^j(\theta_n,x)\cdot\nabla_{t,x}w_n^J(\theta_n,x)+\nabla_{t,x}U_n^j(\theta_n,x)\cdot\nabla_{t,x}r_n^J(\theta_n,x)\ud x=0
	\end{equation*}
	
\end{lemma}
For the energy-critical wave equation, the same results were proved in \cite{DKM1} for $N=3$ which was extended to even dimensions in \cite{CKS} and all other odd dimensions in \cite{CR}.

As in \cite{DKM1}, there exists $(v_0,v_1)\in \mathcal{H}_a$ such that $$(u(t), \partial_t u(t))\mapsto (v_0, v_1)\ \ \mathrm{as}\ \ t\rightarrow1^-$$
weakly in $\mathcal{H}_a$. Let $v$ be the solution to \eqref{lwp} with initial data $(v_0, v_1)$ at time $t=1$. Let $d=u-v$ be the singular part of $u$. Then $d$ is well-defined on $[t_-, 1)$ for some $t_->\mathrm{max}\{T_-(u), T_-(v)\}$. By finite speed of propagation, it follows that $\vec{d}$ is supported in the cone
$$\Big\{(t,x): t_-\leq t<1 \ \ \mathrm{and}\ \ |x|\leq 1-t\Big\}.$$
\begin{proposition}\label{noenergy}
	Let $u$ and $d$ be as above. Then
	\begin{equation*}
		\forall c_0\in(0,1), \ \ \lim_{t \rightarrow1^-}\int_{c_0(1-t)\leq|x|\leq 1-t}\Big(|\nabla d(t,x)|^2+|\partial_t d(t,x)|^2\Big)\ud x=0.
	\end{equation*}
\end{proposition}
The key step in the proof of Proposition \ref{noenergy} is to establish the following
\begin{lemma}
	Let $I_n^+=I_n\cap[\frac12, +\infty)$, and let $I_n^-=I_n\cap(-\infty, \frac12]$. Then there exists $\eta_1>0$ such that the following holds for large $n$; for all sequences $\{\theta_n\}_n$ with $\theta_n\in I_n^+$, or for all sequences $\{\theta_n\}_n$ with $\theta_n\in I_n^-$,
	\begin{equation*}
		\int_{|y|\geq |\theta_n|}\Big|\nabla_{\iota,y}(u_n-v_n)(\theta_n,y)\Big|^2\ud y\geq\eta_1.
	\end{equation*}
\end{lemma}
Theorem \ref{timesequence} can be proved in exactly the same way as in \cite{CR} once the following  proposition is established
\begin{proposition}\label{djt}
	Let $u$ be a radial solution to \eqref{lwp} with $T_+<\infty$ that satisfies \eqref{bh}. Then there exists $(v_0,v_1)\in \dot{H}^1_a\times L^2$, a sequence $T_+(u)$, an integer $J_0\geq1$, $J_0$ sequences of positive numbers $\{\lambda_{j,n}\}_n$, $j=1,\cdots, J_0$, $J_0$ signs $\iota_j\in\{\pm 1\}$ such that\begin{equation*}
		\lambda_{1,n}\ll \lambda_{2,n}\ll \cdots \ll \lambda_{J_0, n}\ll T_+(u)-t_n
	\end{equation*}
	and \begin{equation*}
		\vec{u}(t_n)=(v_0, v_1)+\sum_{j=1}^{J_0}\Big(\frac{\iota_j}{\lambda_{j,n}^\frac{N-2}{2}}W_a(\frac{x}{\lambda_{j,n}}, 0)\Big)+(w_{0,n}, 0)+o(1)\ \ \mathrm{in}\ \ \dot{H}^1_a\times L^2\ \ \mathrm{as}\ n\rightarrow\infty,
	\end{equation*}
	where \begin{equation*}
		\lim_{n\rightarrow\infty}\Big\|S(t)(w_{0,n}, 0)\Big\|_{S(\Bbb R)}=0.
	\end{equation*}
\end{proposition}
As in \cite{DKM1}, we assume  that there exist sequences $\{\tau_n\}$, $\{\lambda_n\}$ such that $\tau_n\in(0,1)$ $\tau_n\rightarrow1$ and for the singular part $d(x,t)=u(x,t)-v(x,t)$ it holds that $$\lambda_n\ll1-\tau_n,$$
and \begin{equation}\label{std}
\lim_{n\rightarrow\infty}\int_{|x|\geq \lambda_n}\Big[|\nabla d(\tau_n)|^2+(\partial_t d(\tau_n))^2+\frac{(d(\tau_n))^2}{|x|^2}\Big]\ud x=0.
\end{equation}
The proof is just a repetition of \cite{DKM1} verbatim (see Proposition 5.1) once we have established the following lemma:
\begin{lemma}\label{stdlemma}
	Assume \eqref{std}. Then
	\begin{equation}\label{stdl}
		\lim_{n\rightarrow\infty}\frac{1}{1-\tau_n}\int_{\tau_n}^1\int_{\Bbb R^N}(\partial_t d)^2\ud x\ud t=0.
	\end{equation}
\end{lemma}
\begin{proof} Since $d$ and $\partial_t d$ are compactly supported in the space variable, we can safely define
	\begin{equation*}
		z_1(t)=\int_{\Bbb R^N}(u\partial_t u-v\partial_t v)\ud x,\ \ z_2(t)=\int_{\Bbb R^N}(x\cdot\nabla u\partial_t u-x\cdot\nabla v\partial_t v)\ud x.
	\end{equation*}
	Direct computation gives that
    \begin{align*}
		z_1'(t)=&\int_{\Bbb R^N}\Big[(\partial_{t}u)^2-|\nabla u|^2-\frac{a}{|x|^2}|u|^2+|u|^\frac{2N}{N-2}\Big]\ud x\\
		&-\int_{\Bbb R^N}\Big[(\partial_{t}v)^2-|\nabla v|^2-\frac{a}{|x|^2}|u|^2+|v|^\frac{2N}{N-2}\Big]\ud x,
	\end{align*}
	which can be further rewritten as
	\begin{align*}
		z_1'(t)=&\int_{\Bbb R^N}(\partial_t d)^2\ud x-\int_{\Bbb R^N}|\nabla d|^2\ud x-a\int_{\Bbb R^N}\frac{|d|^2}{|x|^2}\ud x+\int_{\Bbb R^N}|d|^\frac{2N}{N-2}\ud x\\
		&+2\int_{\Bbb R^N}\partial_t d\partial_t v\ud x+2\int_{\Bbb R^N}\nabla d\nabla v\ud x+2a\int_{\Bbb R^N} \frac{d v}{|x|^2}\ud x\\
		&+\int_{\Bbb R^N}\left(|d+v|^\frac{2N}{N-2}-|v|^\frac{2N}{N-2}-|d|^\frac{2N}{N-2}\right)\ud x.
	\end{align*}
	Noting that $|x|\leq 1-t$ on the support of $d$, that $v$ converges in $\dot{H}^1_a\times L^2$ as $t\rightarrow1$, and that $u$ is bounded in $\dot{H}^1_a\times L^2$, we get that as $t\rightarrow 1$,
	\begin{align*}
		z'_1(t)=&\int_{\Bbb R^N}(\partial_t d)^2 \ud x-\int_{\Bbb R^N}|\nabla d|^2\ud x-a\int_{\Bbb R^N}\frac{|d|^2}{|x|^2}\ud x+\int_{\Bbb R^N}|d|^\frac{2N}{N-2}\ud x+o(1).
	\end{align*}
	A similar computation shows that \begin{align*}
		z_2'(t)=&-\frac{N}{2}\int_{\Bbb R^N} (\partial_t u)^2\ud x+\frac{N-2}{2}\int_{\Bbb R^N}\Big(|\nabla u|^2+a\frac{|u|^2}{|x|^2}-|u|^\frac{2N}{N-2}\Big)\ud x\\
		&-\left[-\frac{N}{2}\int_{\Bbb R^N} (\partial_t v)^2\ud x+\frac{N-2}{2}\int_{\Bbb R^N}\Big(|\nabla v|^2+a\frac{|v|^2}{|x|^2}-|v|^\frac{2N}{N-2}\Big)\ud x\right]
	\end{align*}
	Therefore, we have
	\begin{align*}
		z'_2(t)=-&\frac{N}{2}\int_{\Bbb R^N} (\partial_t d)^2\ud x+\frac{N-2}{2}\int_{\Bbb R^N}\Big(|\nabla d|^2+a\frac{|d|^2}{|x|^2}-|d|^\frac{2N}{N-2}\Big)\ud x.
	\end{align*}

	We define $$Z(t)=\frac{N-2}{2}z_1(t)+z_2(t),$$ then
	\begin{align*}
		Z'(t)=-\int_{\Bbb R^N} (\partial_t d)^2\ud x+o(1),\ \ \mathrm{as}\ \ t\rightarrow1^-.
	\end{align*}
	Let $\varepsilon>0$, and $m$, $n$ be two large integers with $n<m$. Integrate the preceding inequality, we get
	\begin{equation*}
		\int_{\tau_n}^{\tau_m}\int_{\Bbb R^N}(\partial_t d)^2\ud x\ud t\leq |Z(\tau_m)-Z(\tau_n)|+\varepsilon|\tau_m-\tau_n|.
	\end{equation*}
	Let $m$ go to infinity, we obtain
	\begin{equation*}
		\int_{\tau_n}^{\tau_m}\int_{\Bbb R^N}(\partial_t d)^2\ud x\ud t\leq |Z(\tau_n)|+\varepsilon(1-\tau_n).
	\end{equation*}
	As in \cite{DKM1}, we have $$\lim_{n\rightarrow\infty}\frac{|Z(\tau_n)|}{1-\tau_n}=0,$$
	which implies
	\begin{equation*}
		\lim_{n\rightarrow\infty}\frac{1}{1-\tau_n}\int_{\tau_n}^1\int_{\Bbb R^N}(\partial_t d)^2\ud x\ud t\leq\varepsilon.
	\end{equation*}
	Since $\varepsilon>0$ is arbitrary, we get \eqref{stdl}.
\end{proof}
As in \cite{DKM1}, based on the profile decomposition and rigidity of elliptic equation \cite{MMZ}, Proposition \ref{djt} follows from Lemma \ref{stdlemma} and Proposition \ref{noenergy}. Thus Theorem \ref{timesequence} is proved.

\section{ Non-radiative solutions and related estimates}
In this section, we shall choose proper parameters to impose the orthogonality conditions and consider a non-radiative solution close to a multi-soliton. In Subsection \ref{6.1}, we will use the channel of energy for the linearized equation in Section \ref{4} and the nonlinear term estimates to give a first order bound of the expansion of the solution. In Section \ref{6.2}, we will introduce the exterior scaling parameter $\iota$ of the multi-soliton and give a lower bound.

\subsection{Non-radiative solutions and orthogonality conditions}\label{6.1}

\begin{definition}\label{def: non-radiative solutions}
Let $t_{0}\in\mathbb{R}$, and let $u$ be a solution of the nonlinear wave equation with the inverse-square potential \eqref{weq}. We say that $u$ is non-radiative at $t=t_{0}$ if $u$ is defined on $\left\{|x|>\left|t-t_{0}\right|\right\}$ and
\begin{align*}
\sum\limits_{\pm}\lim\limits_{t\rightarrow\pm\infty}\int_{|x| \geq\left|t-t_{0}\right|}\left|\nabla_{t, x}u(t, x)\right|^{2}\ud x=0.
\end{align*}
We say that $u$ is weakly non-radiative, if for large $R>0, u$ is defined on $\{|x|>|t|+R\}$, and
\begin{align*}
\sum_{\pm} \lim _{t \rightarrow \pm \infty} \int_{|x| \geq|t|+R}\left|\nabla_{t, x} u(t, x)\right|^{2} \ud x=0.
\end{align*}
\end{definition}

\medskip
In this section, we assume that $(N-2)\beta\geq 3$ is odd and $u$ is the radial solution of \eqref{weq} with initial data $(u_0,u_1)$ at $t=0$ which is defined on $\{|x|>t\}$. Fix $J\geq 1$, $(\iota_{j})\in\{\pm1\}^{J}$, we denote the multi-soliton with scaling parameters $\boldsymbol{\lambda}=(\lambda_{j})_{j=1}^{J}\in G_{J}$ by $M_{a}=\sum\limits_{j=1}^{J}\iota_{j} {W_a}_{\left(\lambda_{j}\right)}$, where
\begin{align}\label{def GJ}
    G_J=\{\boldsymbol{\lambda}=(\lambda_{j})_{j=1}^{J}:0<\lambda_J<\lambda_{J-1}<\cdots<\lambda_1\}.
\end{align}
We assume that there exists $\boldsymbol{\lambda}\in G_{J}$ such that
\begin{align}\label{def: delta}
\Big\|(u_0,u_1)-(M_{a}, 0)\Big\|_{\mathcal{H}_{a}}\triangleq \delta\ll 1
\end{align}
and
\begin{align*}
\gamma:=\gamma(\boldsymbol{\lambda})=\max\limits_{2\leq j\leq J}\frac{\lambda_j}{\lambda_{j-1}}\ll 1.
\end{align*}

In the following lemma, we choose the scaling parameters to get the orthogonality condition.
\begin{lemma}\label{lem: modulation}\emph{(\cite{DKM7})}
Let $J\geq 1$. There exists a small constant $\varepsilon_{J}>0$ and a large constant $C_{J}>0$ with the following property: for all monotonically decreasing $\boldsymbol{\mu}=(\mu_j)_{j}\in(0,\infty)^{J}$ such that $\gamma(\boldsymbol{\mu})<\varepsilon$, for all $(\iota_{j})_{j}\in\{\pm 1\}^{J}$, for all $f\in\dot{H}_{a}^{1}$ such that
\begin{align*}
\Big\|f-\sum\limits_{j=1}^{J}\iota_{j}{W_{a}}_{\left(\mu_{j}\right)}\Big\|_{\dot{H}_{a}^{1}}\leq\varepsilon,
\end{align*}
then there exists a unique $\boldsymbol{\lambda}\in(0,\infty)^{J}$ such that
\begin{align*}
\max\limits_{1\leq j\leq J}\Big|\frac{\lambda_{j}}{\mu_{j}}-1\Big|\leq C_{J}\varepsilon,
\end{align*}
and for all $1\leq j\leq J$,
\begin{align*}
\int\left(\partial_{r}+\frac{c}{r}\right)\Big(f-\sum\limits_{j=1}^{J}\iota_{j}{W_{a}}_{\left(\lambda_{j}\right)}\Big)\cdot\left(\partial_{r}+\frac{c}{r}\right)\left(\Lambda {W_{a}}\right)_{\left(\lambda_{j}\right)}\ud x=0.
\end{align*}
Furthermore, the map $f\mapsto\lambda$ is of class $C^1$.
\end{lemma}

\begin{proof}~Define the mapping
\begin{align*}
&\Phi:(0,\infty)^{J}\rightarrow \mathbb{R}^{J},\qquad \Phi=(\Phi_{\ell})_{1\leq \ell\leq J},\nonumber
\end{align*}
where
\begin{align*}
\Phi_{\ell}(\boldsymbol{\lambda})&=\lambda_{\ell}-\frac{\mu_{\ell}\iota_{\ell}}{\int\left|(\partial_{r}+\frac{c}{r})(\Lambda W_{a})\right|^{2}}\int\left(\partial_{r}+\frac{c}{r}\right)\Big(f-\sum\limits_{j=1}^{J}\iota_{j}{W_{a}}_{(\lambda_{j})}\Big)\cdot\left(\partial_{r}+\frac{c}{r}\right)\left(\Lambda {W_{a}}_{(\lambda_{j})}\right)\ud x.
\end{align*}
Thus it remains to prove the following claim

\medskip
\begin{claim}\label{claim: contraction}~$\Phi$ is a contraction mapping on the compact set
\begin{align*}
D_{\eta}\triangleq\left\{(\lambda_{j})_{1\leq j\leq J}:\max\limits_{1\leq j\leq J}\Big|1-\frac{\lambda_{\ell}}{\mu_{\ell}}\Big|\leq\eta\right\}.
\end{align*}
\end{claim}

\medskip
We now prove this claim. For any $\boldsymbol{\lambda}\in \textcolor{blue}{D_{\eta}}$,
\begin{align*}
\Big|\frac{\Phi_{\ell}(\boldsymbol{\lambda})}{\mu_{\ell}}-1\Big|&\leq\Big|1-\frac{\Phi_{\ell}(\mu)}{\mu_{\ell}}\Big|+\Big|\frac{\Phi_{\ell}(\lambda)-\Phi_{\ell}(\mu)}{\mu_{\ell}}\Big|\nonumber\\
&\leq\Big\|f-\sum\limits_{j=1}^{J}\iota_{j}{W_{a}}_({\mu_j})\Big\|_{\dot{H}_{a}^{1}}+\frac{1}{\mu_{\ell}}\sum\limits_{j=1}^{J}\frac{\ud \Phi_{\ell}}{\ud\lambda_{j}}\cdot(\mu_{j}-\lambda_{j}).
\end{align*}
If $\ell\neq j$, it is not hard to see that
\begin{align*}
\frac{\ud}{\ud \lambda_{j}}\Big(\partial_{r}+\frac{c}{r}\Big){W_{a}}_{(\lambda_{j})}&=\frac{\ud}{\ud \lambda_{j}}\Big(\partial_{r}+\frac{c}{r}\Big)\bigg({\lambda_{j}^{-\frac{N-2}{2}}W_{a}\Big(\frac{|x|}{\lambda_{j}}\Big)}\bigg)\\
&=\lambda_{j}^{-1}(\Lambda W_{a})_{(\lambda_{j})}.
\end{align*}
Using \eqref{main W}, the partial derivative of $\lambda_j$ is
\begin{align*}
\left|\frac{\ud \Phi_{\ell}}{\ud\lambda_{j}}\right|&=-\frac{\mu_{\ell}\iota_{\ell}}{\int\left|(\partial_{r}+\frac{c}{r})(\Lambda W_{a})\right|^{2}\ud x}\lambda_{j}^{-1}\left|\int\Big(\partial_{r}+\frac{c}{r}\Big)(\Lambda W_a)_{(\lambda_{j})}\cdot\Big(\partial_{r}+\frac{c}{r}\Big)(\Lambda W_{a})_{(\lambda_{\ell})}\ud x\right|\\
&\lesssim\frac{\mu_{\ell}}{\lambda_{j}}(\lambda_{j}\lambda_{\ell})^{-\frac{N-2}{2}}\max\left\{\lambda_{j}^{\frac{(\beta-1)N}{2}+N-\beta}\lambda_{\ell}^{-\frac{(\beta+1)N}{2}+\beta},\lambda_{j}^{\frac{(\beta+1)N}{2}-\beta}\lambda_{\ell}^{-\frac{(\beta-1)N}{2}+N+\beta}\right\}\\
&\lesssim\frac{\mu_{\ell}}{\mu_{j}}\varepsilon^{\frac{(N-2)\beta}{2}}.
\end{align*}
Therefore,
\begin{align}\label{l noteq j}
\left|\frac{1}{\mu_{\ell}}\frac{\ud \Phi_{\ell}}{\ud\lambda_{j}}(\mu_{j}-\lambda_{j})\right|\leq\frac{(\mu_{j}-\lambda_{j})}{\mu_{j}}\varepsilon^{\frac{(N-2)\beta}{2}}\leq\eta\varepsilon^{\frac{(N-2)\beta}{2}}.
\end{align}
If $\ell=j$, we observe that
\begin{align*}
\frac{\ud}{\ud\lambda_{\ell}}\left(\Lambda{W_{a}}_{(\lambda_{\ell})}\right)=\lambda_{\ell}^{-1}(\Lambda\Lambda{W_{a}})_{(\lambda_{\ell})},
\end{align*}
where we have used the expression of $W_a$ in the last line. Hence,
\begin{align}\label{l=j}
\left|\frac{\ud\Phi_{\ell}}{\ud\lambda_{\ell}}(\boldsymbol{\lambda})\right|\leq&\Big|1-\frac{\mu_{\ell}}{\lambda_{\ell}}\Big|+\int\left|\left(\partial_{r}+\frac{c}{r}\right)\left({W_{a}}_{(\lambda_{j})}-{W_{a}}_{(\mu_{j})}\right)\right|^2\ud x\nonumber\\
+&\bigg|\frac{\mu_{\ell}}{\lambda_{\ell}\int\left|(\partial_{r}+\frac{c}{r})(\Lambda W_{a})\right|^2\ud x}\int\big(\partial_{r}+\frac{c}{r}\big)\Big(f-\sum\limits_{j=1}^{J}\iota_{j}{W_{a}}_{(\lambda_{j})}\Big)\cdot\big(\partial_{r}+\frac{c}{r}\big)(\Lambda\Lambda W_{a})_{(\lambda_{\ell})}\ud x\bigg|\nonumber\\
\lesssim&\Big|1-\frac{\mu_{\ell}}{\lambda_{\ell}}\Big|+\Big\|f-\sum\limits_{j=1}^{J}\iota_{j}{W_{a}}_{(\mu_j)}\Big\|_{\dot{H}_{a}^{1}}+\Big\|\sum\limits_{j=1}^{J}\left(\iota_{j}{W_{a}}_{(\mu_{j})}-\iota_{j}{W_{a}}_{(\lambda_{j})}\right)\Big\|_{\dot{H}_{a}^{1}}\nonumber\\
\leq&\eta+\varepsilon,
\end{align}
and direct computation shows that
\begin{align*}
&\int\Big|\big(\partial_{r}+\frac{c}{r}\big)\left({W_{a}}_{(\lambda_{j})}-{W_{a}}_{(\mu_{j})}\right)\Big|^2\ud x\\
=&\int\Big|\Big(\big(\frac{\mu_j}{\lambda_j}\big)^{\frac{N}{2}}\partial_{r}W_a\big(\frac{\mu_j}{\lambda_j}r\big)-\partial_{r}W_a(r)\Big)+\frac{c}{r}\Big(\big(\frac{\mu_j}{\lambda_j}\big)^{\frac{N-2}{2}}W_a\big(\frac{\mu_j}{\lambda_j}r\big)-W_a(r)\Big)\Big|^2r^{N-1}\ud r\\
\lesssim&\eta^2+\int\Big|\big(\partial_{r}+\frac{c}{r}\big)W_a\big(\frac{\mu_j}{\lambda_j}r\big)-\big(\partial_{r}+\frac{c}{r}\big)W_{a}(r)\Big|^2r^{N-1}\ud r\\
\lesssim&\eta^2+\int_{0}^{\infty}\Big|r-\frac{\mu_j}{\lambda_j}r\Big|^2\frac{r^{N-1}}{\big(r^{\frac{(\beta-1)N}{2}-\beta-1}+r^{-\frac{(\beta+1)N}{2}+\beta-1}\big)^2}\ud r\lesssim\eta^2.
\end{align*}

Collecting \eqref{l noteq j} and \eqref{l=j}, if $\boldsymbol{\lambda}\in B_{\eta}$, and $\eta=M_{J}\varepsilon$ for some large constant $M_J$ such that $\varepsilon\leq\varepsilon_{J}\ll M_{J}^{-1}$, then
\begin{align*}
\Big|\frac{\Phi_{\ell}(\boldsymbol{\lambda})}{\mu_{\ell}}-1\Big|\lesssim\varepsilon+\eta\varepsilon^{\frac{(N-2)\beta}{2}}+\eta(\eta+\varepsilon)\leq\eta.
\end{align*}
Thus, $\Phi$ maps $D_{\eta}$ into $D_{\eta}$. Similarly, we can prove $\Phi$ is a contraction mapping on $D_{\eta}$. By the Banach fixed point theorem, concluding the proof of Claim \ref{claim: contraction}.
\end{proof}

Due to  Lemma \ref{lem: modulation}, we can change the scaling parameters $(\lambda_j)_j$ such that
\begin{align*}
    u_{0}=M_a+h_{0},
\end{align*}
and
\begin{align}\label{def: g1}
    u_1=\sum\limits_{j=1}^{J}\alpha_{j}\left(\Lambda W_a\right)_{[\lambda_j]}+g_1.
\end{align}
Here  $(h_0,g_1)$ satisfies the following orthogonality conditions: for all $j$,
\begin{align}\label{orthogonality 1}
\int\big(\partial_{r}+\frac{c}{r}\big)h_0\cdot\big(\partial_{r}+\frac{c}{r}\big)(\Lambda W_a)_{(\lambda_j)}=0,
\end{align}
\begin{align}\label{orthogonality 2}
\int g_1\cdot(\Lambda W_{a})_{[\lambda_j]}=0.
\end{align}
Before proving the estimates of $(h_0,g_1)$, we first introduce the following lemma.

\begin{lemma}\label{Lemma 2.8}
Let $(N-2)\beta\geq1$, and $M$ be a positive constant. For all $V\in L_{\text{loc}}^{\frac{2(N+1)}{N+4}}\left(\mathbb{R}, L^{\frac{2(N+1)}{3}}\left(\mathbb{R}^{N}\right)\right)$ with
\begin{align*}
\Big\|\chi_{\{|x| \geq|t|\}} V\Big\|_{L^{\frac{2(N+1)}{N+4}}\big(\mathbb{R}, L^{\frac{2(N+1)}{3}}(\mathbb{R}^{N})\big)} \leq M,
\end{align*}
Let $u$ be the solution of
\begin{align*}
\partial_{t}^{2} u-\Delta u+V u=f_{1}+f_{2}, \quad \vec{u}\mid_{ t=0}=\left(u_{0}, u_{1}\right) \in \mathcal{H}_a,
\end{align*}
where $f_{1} \in L^{1}\left(\mathbb{R}, L^{2}\left(\mathbb{R}^{N}\right)\right)$ and $f_{2} \in W^{\prime}$, then there exists $C_{M}>0$ such that
\begin{align}\label{potential outside a wave cone}
&\left\|\chi_{\{|x| \geq|t|\}}u\right\|_{{L}^{\frac{2(N+1)}{N-2}}\left(\mathbb{R} \times \mathbb{R}^{N}\right)} +\sup _{t \in \mathbb{R}}\left\|\chi_{\{|x| \geq|t|\}} \nabla_{t, x} u(t)\right\|_{L^{2}} \nonumber\\
\leq & C_{M}\left(\left\|\left(u_{0}, u_{1}\right)\right\|_{\mathcal{H}_a}+\left\|\chi_{\{|x| \geq|t|\}} f_{1}\right\|_{L^{1}\left(\mathbb{R}, L^{2}\right)}+\left\|\chi_{\{|x| \geq|t|\}} f_{2}\right\|_{W^{\prime}}\right) .
\end{align}
\end{lemma}

The proof easily follows from the H\"{o}lder's inequality and Strichartz estimates. Therefore, a rigorous proof of Lemma \ref{Lemma 2.8} is referred to \cite{DKM7}. Then as a consequence, we obtain the following estimate.

\medskip
\begin{lemma}\label{lem:proj of initial data}
Let $u$ be the solution of \eqref{weq} defined on $\left\{(t,x)\in\mathbb{R}^{N}:|x|>|t|\right\}$, then
\begin{align}\label{projection of initial data}
\left\|\operatorname{\pi}_{Z_{\lambda}^{a}}^{\bot}\Big((u_0,u_1)-(M_{a},0)\Big)\right\|_{\mathcal{H}_{a}}\lesssim\gamma^{\frac{(N+2)\beta}{4}}+\delta^{\frac{N+1}{N-2}}\lesssim\gamma^{\frac{N\beta}{4}}+\delta^{\frac{N-2}{2}}.
\end{align}
\end{lemma}
\begin{proof}
Let $h(t)=u(t)-M_{a}$, then $h(t)$ satisfies
\begin{align*}
\left\{\begin{array}{l}
\partial_{t}^2 h+\mathcal{L}_{\lambda}^{a}h=F(h)+\mathcal{N}(h),\\
h(0)=(u_0,u_1)-(M_{a},0),
\end{array}\right.
\end{align*}
where
\begin{align*}
\mathcal{N}(h)=F(M+h)-\sum_{j=1}^{J} F\left(\iota_{j} {W_{a}}_{\left(\lambda_{j}\right)}\right)-F(h)-\frac{N+2}{N-2}\sum\limits_{j=1}^{J} {W_{a}}_{\left(\lambda_{j}\right)}^{\frac{4}{N-2}} h.
\end{align*}
Define $\tilde{h}(t)=\chi_{\{|x|\geq|t|\}}h(t)$. Thus $\tilde{h}(t)-h(t)$ satisfies
\begin{align*}
\left\{\begin{array}{l}
\big(\partial_{t}^2+\mathcal{L}_{\lambda}^{a}\big)(\tilde{h}(t)-h(t))=\chi_{\{|x|\geq|t|\}}\big(F(h)+\mathcal{N}(h)\big)\\
\big(\tilde{h}-h\big)(0)=(0,0),
\end{array}\right.
\end{align*}
assume $T>0$ and denote by $\Gamma(T)=\{(t,x),|t|\leq\min\{|x|, T\}\}$. Employing the fractional chain rule
\begin{align}\label{the norm W' for Fh}
\Big\|F(h)\chi_{\{|x|\geq|t|\}}\Big\|_{W^{\prime}((0, T))}=\left\|F(\tilde{h})\chi_{\{|x|\geq|t|\}}\right\|_{W^{\prime}((0, T))} \lesssim\|\tilde{h}\|_{{W}((0, T))}\|\tilde{h}\|_{S(\Gamma(T))}^{\frac{4}{N-2}}.
\end{align}

Recall Claim $A.5$ in \cite{DKM7}
\begin{align*}
\bigg|F\Big(\sum_{j=1}^{J} y_{j}+h\Big) &-\sum_{j=1}^{J} F\left(y_{j}\right)-F(h)-\frac{N+2}{N-2} \sum_{j=1}^{J}\left|y_{j}\right|^{\frac{4}{N-2}} h\bigg| \\
&\lesssim\sum_{j\neq k}\min \left(\left|y_{j}\right|^{\frac{4}{N-2}}\left|y_{k}\right|,\left|y_{k}\right|^{\frac{4}{N-2}}\left|y_{j}\right|\right)
+\sum_{j=1}^{J}\left|y_{j}\right|^{\frac{1}{N-2}}|h|^{\frac{N+1}{N-2}}.
\end{align*}
Given $y_j={W_a}_{(\lambda_j)}^{\frac{4}{N-2}}$, it follows that
\begin{align*}
|\mathcal{N}(h)| \lesssim \sum_{j \neq k} \min \Big({W_a}_{(\lambda_{j})}^{\frac{4}{N-2}} {W_a}_{(\lambda_{k})},{W_a}_{(\lambda_{k})}^{\frac{4}{N-2}} {W_a}_{(\lambda_{j})}\Big)+\sum_{j=1}^{J} {W_a}_{\left(\lambda_{j}\right)}^{\frac{1}{N-2}}|h|^{\frac{N+1}{N-2}} .
\end{align*}
When $j \neq k$, by Proposition \ref{prop: L1L2 bdd of Wa},
\begin{align}\label{L1L2 norm of Wa}
\left\|\chi_{\{|x| \geq|t|\}} \min \Big\{{W_a}_{\left(\lambda_{j}\right)}^{\frac{4}{N-2}} {W_a}_{\left(\lambda_{k}\right)}, {W_a}_{\left(\lambda_{j}\right)}^{\frac{4}{N-2}} {W_a}_{\left(\lambda_{k}\right)}\Big\}\right\|_{L_{t}^{1} L_{x}^{2}} \lesssim \gamma^{\frac{(N+2)\beta}{4}}.
\end{align}
Since $W_a^{\frac{1}{N-2}}\lesssim|x|^{-\frac{1+\beta}{2}}$ and $\beta>1$, $W_a^{\frac{1}{N-2}}\chi_{\{|x|>|t|\}}\in L_{t}^2(\mathbb{R},L_{x}^{\infty}(\mathbb{R}^{N}))$, we deduce
\begin{align*}
\Big\|\chi_{\Gamma(T)} {W_a}_{\left(\lambda_{j}\right)}^{\frac{1}{N-2}}|h|^{\frac{N+1}{N-2}}\Big\|_{L_{t}^{1} L^{2}} & \leq\Big\|\chi_{\Gamma(T)} {W_a}_{\left(\lambda_{j}\right)}^{\frac{1}{N-2}}\Big\|_{L_{t}^{2} L_{x}^{\infty}}\Big\|\chi_{\Gamma(T)}|h|^{\frac{N+1}{N-2}}\Big\|_{L_{t, x}^{2}} \lesssim\|\tilde{h}\|_{S(\Gamma(T))}^{\frac{N+1}{N-2}},
\end{align*}
which, combined with \eqref{the norm W' for Fh}, yields,
\begin{align*}
\|\mathcal{N}(h)\|_{L^1L^2}\lesssim\gamma^{\frac{(N+2)\beta}{4}}+\|\tilde{h}\|_{S(\Gamma(T))}^{\frac{N+1}{N-2}}.
\end{align*}

We define $h_{L}(t)$ to be the solution of linear equation
\begin{align*}
\left\{\begin{array}{l}
\partial_{t}^2h_{L}+\mathcal{L}_{\lambda}^{a}h_{L}=0, \\
h_{L}(0)=(u_0,u_1)-(M_{a},0).
\end{array}\right.
\end{align*}
It is apparent from \eqref{potential outside a wave cone} that
\begin{align*}
\Big\|\tilde{h}-h_{L}\Big\|_{S(\Gamma(T))} \lesssim \gamma^{\frac{(N+2)\beta}{4}}+\|\tilde{h}\|_{S(\Gamma(T))}^{\frac{N+1}{N-2}}+\|\tilde{h}\|_{S(\Gamma(T))}^{\frac{4}{N-2}}\|\tilde{h}\|_{W((0, T))}.
\end{align*}
Using the Strichartz estimate \eqref{ineq: Strichartz estimate}, we obtain
\begin{align*}
&\sup_{-T \leq t \leq T}\Big\|\overrightarrow{\tilde{h}}(t)-\vec{h}_{L}(t)\Big\|_{\mathcal{H}_a}+\Big\|\overrightarrow{\tilde{h}}(t)-\vec{h}_{L}(t)\Big\|_{S((0,T))\cap W((0,T))}\\
\lesssim& \gamma^{\frac{(N+2)\beta}{4}}+\|\tilde{h}\|_{S(\Gamma(T))}^{\frac{N+1}{N-2}}+\|\tilde{h}\|_{S(\Gamma(T))}^{\frac{4}{N-2}}\|\tilde{h}\|_{W((0,T))}.
\end{align*}

On account of $\|h_{L}\|_{S((0,T))\cap W((0,T))}\lesssim\delta$ and
\begin{align*}
\big\|\tilde{h}\big\|_{S((0,T))\cap W((0,T))}\lesssim\big\|(u_0-M_a,u_1)\big\|_{\mathcal{H}_a}+\gamma^{\frac{(N+2)\beta}{4}}\lesssim\delta+\gamma^{\frac{(N+2)\beta}{4}},
\end{align*}
we have
\begin{align}\label{the bdd of h subtract hL}
\sup\limits_{-T\leq t\leq T}\big\|\vec{\tilde{h}}(t)-\vec{h}_{L}(t)\big\|_{\mathcal{H}_a}\lesssim\gamma^{\frac{(N+2)\beta}{4}}+\delta^{\frac{N+1}{N-2}}.
\end{align}
whose right hand side is uniform in $T$, that is
\begin{align*}
\sup _{t \in \mathbb{R}}\left\|\vec{\tilde{h}}(t)-\vec{h}_{L}(t)\right\|_{\mathcal{H}_a}\lesssim \gamma^{\frac{(N+2)\beta}{4}}+\delta^{\frac{N+1}{N-2}}.
\end{align*}
Going back to the assume that $u$ is non-radiative, it holds that
\begin{align*}
0=&\sum\limits_{\pm}\lim\limits_{t\rightarrow\pm\infty}\int_{|x|\geq|t|}\big|\nabla_{t,x}u(t,x)\big|^2\ud x=\sum\limits_{\pm}\lim\limits_{t\rightarrow\pm\infty}\int_{|x|\geq|t|}\big|\nabla_{t,x}(h(t,x)+M_a)\big|^2\ud x\\
=&\sum\limits_{\pm}\lim\limits_{t\rightarrow\pm\infty}\int\big|\nabla_{t,x}(\tilde{h}+M_a)\big|^2\ud x\\
\lesssim&\gamma^{\frac{(N+2)\beta}{4}}+\delta^{\frac{N+1}{N-2}}.
\end{align*}
Combining \eqref{the bdd of h subtract hL} and Theorem \ref{thm:multisoliton energy}, we divide
\begin{align*}
\big\|\operatorname{\pi}_{Z_{\lambda}^{\perp}}(h_0,h_1)\big\|_{\mathcal{H}_a}\lesssim\gamma^{\theta_{N}}\delta+\gamma^{\frac{(N+2)\beta}{4}}+\delta^{\frac{N+1}{N-2}}.
\end{align*}
This completes the proof of Lemma \ref{lem:proj of initial data}.
\end{proof}
Now we are in the position to show the estimate of $(h_0,g_1)$.

\medskip
\begin{proposition}\label{prop: coefficients estimate}
\begin{align}\label{bdd: initial data}
\big\|(h_0,g_1)\big\|_{\mathcal{H}_{a}}\lesssim\gamma^{\frac{N\beta}{4}}+\delta^{\frac{N}{N-2}},
\end{align}
\begin{align}\label{bdd: delta2}
\Big|\delta^{2}-\sum\limits_{j=1}^{\infty}\alpha_{j}^2\|\Lambda W_{a}\|_{L^2}^2\Big|\lesssim\gamma^{\frac{(N-2)\beta-2}{2}}+\delta^{\frac{2(N-1)}{N-2}}.
\end{align}
\end{proposition}

\begin{proof}

According to \eqref{projection of initial data}, in view of the orthogonality condition \eqref{orthogonality 1} and the expansion \eqref{def: g1} of $u_1$, we get \eqref{bdd: initial data}. By \eqref{def: delta}, a direct computation gives
\begin{align*}
\delta^2=&\big\|(u_0-M_{a},u_1)\big\|_{\mathcal{H}_{a}}^{2}=\big\|h_0\big\|_{2}^{2}+\big\|g_1\big\|_{2}^{2}+\Big\|\sum\limits_{j=1}^{J}\alpha_{j}(\Lambda{W_{a}})_{[\lambda_{j}]}\Big\|_{2}^{2}\\
=&\big\|h_0\big\|_{2}^{2}+\big\|g_1\big\|_{2}^{2}+\sum\limits_{j=1}^{J}\alpha_{j}\big\|\Lambda{W_{a}}\big\|_{2}^{2}+\sum\limits_{j\neq k}\int\big|\alpha_{j}\alpha_{k}(\Lambda{W_{a}})_{[\lambda_{j}]}(\Lambda{W_{a}})_{[\lambda_{k}]}\big|\ud x.
\end{align*}
Using the expression of $W_{a}$ and \eqref{main W}, we obtain
\begin{align*}
\Big|\delta^2-\sum\limits_{j=1}^{J}\alpha_{j}\big\|\Lambda W_{a}\big\|_{2}^{2}\Big|\lesssim&\gamma^{\frac{(N-2)\beta}{2}-1}\sum\limits_{j=1}^{J}\alpha_{j}^2+\delta^{\frac{2(N+1)}{N-2}}+\gamma^{\frac{(N+2)\beta}{2}}\\
\lesssim&\delta^{\frac{2(N-1)}{N-2}}+\gamma^{\frac{(N-2)\beta-2}{2}},
\end{align*}
which completes the proof of \eqref{bdd: delta2}.
\end{proof}

\medskip
\subsection{Estimates about the exterior scaling parameter}\label{6.2}
In this paper, the projection space $P(R)$ \eqref{def: P(R)} is a subspace of $\mathcal{H}_{a}(R)$ spanned by $\mathcal{P}$
\begin{align*}
\mathcal{P}=\operatorname{span}\bigg\{&\Big(r^{-\frac{(N-2)\beta+(N+2)}{2}+2k_1},0\Big),\Big(0,r^{-\frac{(N-2)\beta+(N+2)}{2}+2k_2}\Big):\\
&k_1=1,2,...,\Big\lfloor\frac{(N-2)\beta}{4}+1\Big\rfloor,k_2=1,2,...,\Big\lfloor\frac{(N-2)\beta}{4}+\frac{1}{2}\Big\rfloor\bigg\},
\end{align*}
then the number of the basis in $\mathcal{P}$ is
\begin{align*}
\Big\lfloor\frac{(N-2)(\beta-1)+(N+2)}{4}\Big\rfloor+\Big\lfloor\frac{(N-2)(\beta-1)+N}{4}\Big\rfloor=k,
\end{align*}
and the norm of the elements of $\mathcal{P}$ are
\begin{align}\label{the first norm of P}
\left\|\big(r^{-\frac{(N-2)\beta+(N+2)}{2}+2k_1},0\big)\right\|_{\mathcal{H}_{a}(R)}^2=\frac{\big((N-2)\beta+(N+2)-4k_1\big)^2}{4\big((N-2)\beta-4k_1+4\big)}R^{-(N-2)\beta+4k_1-4},
\end{align}
and
\begin{align}\label{the second norm of P}
\left\|\big(0,r^{-\frac{(N-2)\beta+(N+2)}{2}+2k_2}\big)\right\|_{\mathcal{H}_{a}(R)}^2=\frac{1}{\big((N-2)\beta-4k_2+2\big)}R^{-(N-2)\beta+4k_2-2}.
\end{align}

Rewrite $\mathcal{P}=\{\Xi_{m}\}_{m\in\llbracket1,k\rrbracket}$ and make the assumption $\|\Xi_{m}\|_{\mathcal{H}_{a}(R)}\approx\frac{c_m}{R^{m-a}}$. We distinguish between two cases.

\begin{itemize}
\item We first consider the case where $m=1$, the norm is
\begin{align*}
\|\Xi_{1}\|_{\mathcal{H}_{a}(R)}^2\approx\frac{c_1^2}{R^{2-2a}}.
\end{align*}
By employing \eqref{the first norm of P} and \eqref{the second norm of P}, choosing $a=-k+\frac{N}{2}$, we observe that
\begin{itemize}
\item If $k$ is odd,
\begin{align*}
(N-2)\beta-4k_1-4=2-2a,\quad r^{-\frac{(N-2)\beta+(N+2)}{2}+2k_1}=r^{-k}.
\end{align*}
\item If $k$ is even
\begin{align*}
(N-2)\beta-4k_2+2=2-2a,\quad r^{-\frac{(N-2)\beta+(N+2)}{2}+2k_2}=r^{-k-1}.
\end{align*}
\end{itemize}

\item We next assume that $m=k$. Then $\|\Xi_k\|_{\mathcal{H}_{a}(R)}^2=\frac{c_k^2}{R^{2k-2a}}$. By \eqref{the first norm of P} and the assumption $a=-k+\frac{N}{2}$, we have the element satisfies
\begin{align*}
r^{-\frac{(N-2)\beta+(N+2)}{2}+2k_1}=r^{-(N-2)\beta},
\end{align*}
this imply $\Xi_k(r)=(r^{-(N-2)\beta},0)$.
\end{itemize}

Summarize these two cases, we denote by $\mathcal{P}=\{\Xi_m\}_{m\in\llbracket1,k\rrbracket}$ and choosing $\Xi_m$ such that
\begin{align}\label{def: the norm of Xi}
\|\Xi_m\|_{\mathcal{H}_a(R)}=\frac{c_m}{R^{m+k-\frac{N}{2}}},
\end{align}
for some constant $c_m\neq 0$.

Incorporating with scaling transformation, we have the following lemma.
\begin{lemma}\label{lem: the norm of U}
Let $U\in P(R)$ and denote by $(\theta_{m}(R))_{1\leq m\leq k}$ its coordinates in $\mathcal{P}$. Then
\begin{align*}
\|U\|_{\mathcal{H}_{a}(R)}\approx\sum\limits_{m=1}^{k}\frac{|\theta_{m}(R)|}{R^{m+k-\frac{N}{2}}}.
\end{align*}
\end{lemma}
\begin{proof}
Assume $U=(u,v)\in\mathcal{H}_{a}(R)$, it follows that $\int_{R}^{\infty}\left(|(\partial_{r}+\frac{c}{r})u|^2+|v|^2\right)r^{N-1}\ud r<\infty$. Consider
\begin{align*}
U_{R^{-1}}(x)=\left(R^{\frac{N}{2}-1}u(Rx),R^{\frac{N}{2}}v(Rx)\right).
\end{align*}
A routine computation shows that
\begin{align*}
\big\|U_{R^{-1}}\big\|_{\mathcal{H}_{a}(1)}=&\int_{1}^{\infty}\Big(\big|(\partial_{r}+\frac{c}{r})u(Rx)R^{-1}\big|^2+|v(Rx)|^2\Big)(R r)^{N-1}R\ud r\nonumber\\
=&\int_{R}^{\infty}\Big(\big|(\partial_{r}+\frac{c}{r})u(r)\big|^2+|v(r)|^2\Big)r^{N-1}\ud r\nonumber\\
=&\big\|U\big\|_{\mathcal{H}_a(R)}<\infty,
\end{align*}
thus $U_{R^{-1}}\in\mathcal{H}_{a}(1)$. According to \eqref{def: the norm of Xi} we deduce
\begin{align*}
\frac{c_m}{R^{m+k-\frac{N}{2}}}=\big\|\Xi_{m}\big\|_{\mathcal{H}_{a}(R)}=\big\|(\Xi_{m})_{R^{-1}}\big\|_{\mathcal{H}_{a}(1)},
\end{align*}
which implies $(\Xi_{m})_{R^{-1}}=\frac{1}{R^{m+k-\frac{N}{2}}}\Xi_{m}$.

From the assumption $U=\sum\limits_{m=1}^{k}\theta_{m}(R)\Xi_{m}$, it follows that
\begin{align*}
U_{R^{-1}}=\sum\limits_{m=1}^{k}\theta_{m}(R)(\Xi_{m})_{R^{-1}}=\sum\limits_{m=1}^{k}\frac{\theta_{m}(R)}{R^{m+k-\frac{N}{2}}}\Xi_m.
\end{align*}
Using $U_{R^{-1}}$ and \eqref{def: the norm of Xi}, we obtain
\begin{align*}
\big\|U\big\|_{\mathcal{H}_{a}(R)}=\big\|U_{R^{-1}}\big\|_{\mathcal{H}_{a}(1)}=\Big\|\sum\limits_{m=1}^{k}\frac{\theta_{m}(R)}{R^{m+k-\frac{N}{2}}}\Xi_{m}\Big\|_{\mathcal{H}_a(1)}\approx\sum\limits_{m=1}^{k}\frac{|\theta_{m}(R)|}{R^{m+k-\frac{N}{2}}},
\end{align*}
where the implicit constant is independent of $R>0$. This concludes the proof of Lemma \ref{lem: the norm of U}.
\end{proof}

\medskip
By the same arguments as in \cite{DKM3}, we get the following two theorems. The proofs are omitted.

\begin{theorem}\label{thm: main thm of Xi}\emph{(\cite{DKM3})}
Assume $(N-2)\beta\geq 3$ and $(N-2)\beta$ is odd. Let $u$ be a radial weakly non-radiative solution of \eqref{weq} and $\varepsilon_{0}>0$. Then there exist $m_{0}\in\llbracket1,m\rrbracket,\ell\in\mathbb{R}$ (with $\ell \neq 0$ if $m_{0}<k$) such that, for all $t_{0}\in\mathbb{R}$ and $R_{0}>0$, if $u$ is defined on
$\left\{(t,r): r>\left|t-t_{0}\right|+R_{0}\right\}$, $\left\|\vec{u}\left(t_{0}\right)\right\|_{\mathcal{H}_{a}\left(R_{0}\right)}<\varepsilon_{0}$ and
\begin{align*}
\sum_{\pm} \lim _{t \rightarrow \pm \infty} \int_{|x|>\left|t-t_{0}\right|+R_{0}}\big|\nabla_{t, x} u(t, r)\big|^{2} d x=0,
\end{align*}
then \textcolor{red}{for any} $R>R_{0}$,
\begin{align*}
\quad\big\|\vec{u}\left(t_{0}\right)-\ell\Xi_{m_{0}}\big\|_{\mathcal{H}_{a}(R)}\leq C\max \bigg\{\Big(\frac{R_{0}}{R}\Big)^{\left(m_{0}+k-\frac{N}{2}\right) \frac{N+2}{N-2}},\Big(\frac{R_{0}}{R}\Big)^{m_{0}+\frac{N}{2}-k}\bigg\}.
\end{align*}
\end{theorem}

\begin{theorem}\label{thm: main thm of Xi 2}
Assume $(N-2)\beta\geq 3$ and $(N-2)\beta$ is odd. Let $u$ be a radial non-radiative solution of \eqref{weq}. If $m_0=k$, then $u$ is a stationary solution.
\end{theorem}

\begin{proposition}\label{exterior profile estimate}
For $m_0\leq k-1$, there exist a constant $C>0$ such that if $u$ is a radial solution of \eqref{weq} defined on $\left\{(t,x)\in\mathbb{R}^{N}:|x|>t\right\}$, then
\begin{align*}
|\ell|\leq C\delta^{\frac{2N}{N-2}}\lambda_{1}^{m_0+k-\frac{N}{2}}.
\end{align*}
\end{proposition}
\begin{proof}
By scaling, we find that for any $j\in\llbracket1, J\rrbracket$
\begin{align*}
\big\|{W_a}_{(\lambda_j)}\big\|_{\dot{H}_a^1(R)}^2=&\int_{R}^{\infty}\Big|\big(\partial_{r}+\frac{c}{r}\big){W_a}_{(\lambda_j)}\Big|^2r^{N-1}\ud r=\big\|W_a\big\|_{\dot{H}_a^1(\frac{R}{\lambda_j})}^2.
\end{align*}
It follows from the definition of $\{\lambda_j\}_{1\leq j\leq J}$ that $\frac{R}{\lambda_j}\geq\frac{R}{\lambda_1}\gg1$, this means exactly that
\begin{align*}
W_a(r)\sim r^{-\frac{(N-2)(1+\beta)}{2}}.
\end{align*}
Then
\begin{align*}
\big\|{W_a}_{(\lambda_j)}\big\|_{\dot{H}_a^1(R)}=&\big\|W_{a}\big\|_{\dot{H}_a^1(\frac{R}{\lambda_j})}\\
=&\Big(\int_{R/{\lambda_j}}^{\infty}\Big(\big(\partial_{r}+\frac{c}{r}\big)r^{-\frac{(N-2)(1+\beta)}{2}}\Big)^2r^{N-1}\ud r\Big)^{\frac{1}{2}}=\Big(\frac{R}{\lambda_j}\Big)^{-\frac{(N-2)\beta}{2}}.
\end{align*}
For $(N-2)\beta\geq3$, it follows from $\|\vec{u}(0)-(M_a,0)\|_{\mathcal{H_a}(R)}\lesssim\delta$ that
\begin{align}\label{the bdd of the initial data}
\big\|\vec{u}(0)\big\|_{\mathcal{H}_{a}(R)}\lesssim&\big\|\vec{u}(0)-(M_a,0)\big\|_{\mathcal{H}_{a}(R)}+\left\|\sum\iota_j{W_a}_{(\lambda_j)}\right\|_{\dot{H}_a^1(R)}\nonumber\\
\lesssim&\delta+\Big(\frac{\lambda_1}{R}\Big)^{\frac{(N-2)\beta}{2}}=\delta+\Big(\frac{\lambda_1}{R}\Big)^{k-\frac{1}{2}}.
\end{align}
Let $\varepsilon_0>0$ such that $\|\vec{u}(0)\|_{\mathcal{H}_{a}(\lambda_1)}<\varepsilon_0$. Fix $B>0$ large enough, due to the fact that $\|\vec{u}(0)-(M_a,0)\|_{\mathcal{H}_{a}(R)}=\delta\leq\varepsilon_J\ll 1$, we have that $\|\vec{u}_0\|_{\mathcal{H}_{a}(R\lambda_1)}\leq\varepsilon_0$. For all $R>B\lambda_1$, from Theorem \ref{thm: main thm of Xi}, we get that for $m_0\leq k-1$,
\begin{align*}
\left|\big\|\vec{u}(0)\big\|_{\mathcal{H}_{a}(R)}-\frac{\ell c_{m_0}}{R^{m_0+k-\frac{N}{2}}}\right|\leq C\max\bigg\{\Big(\frac{B\lambda_1}{R}\Big)^{(m_0+k-\frac{N}{2})\frac{N+2}{N-2}},\Big(\frac{\lambda_1}{R}\Big)^{m_0-k+\frac{N}{2}}\bigg\},
\end{align*}
for all $R\geq B\lambda$, thus it follows from \eqref{the bdd of the initial data} that
\begin{align*}
\frac{\ell}{R^{m_0+k-\frac{N}{2}}}\lesssim&\delta+\Big(\frac{\lambda_1}{R}\Big)^{k-\frac{1}{2}}+\max\bigg\{\Big(\frac{\lambda_1}{R}\Big)^{(m_0+k-\frac{N}{2})\frac{N+2}{N-2}},\Big(\frac{\lambda_1}{R}\Big)^{m_0-k+\frac{N}{2}}\bigg\}\\
\lesssim&\delta+\Big(\frac{\lambda_1}{R}\Big)^{a_{m_0}},
\end{align*}
where $a_{m_0}=\min\left\{(m_0+k-\frac{N}{2})\frac{N+2}{N-2},m_0-k+\frac{N}{2}\right\}$.

Choose $R$ such that $\left(\frac{\lambda_1}{R}\right)^{a_{m_0}}=\delta$, then $R=\lambda_1\delta^{-\frac{1}{a_{m_0}}}$, and we have
\begin{align*}
|\ell|\lesssim R^{m_0+k-\frac{N}{2}}\delta=\left(\lambda_1\delta^{-\frac{1}{a_{m_0}}}\right)^{m_0+k-\frac{N}{2}}\delta\leq{\lambda_1}^{m_0+k-\frac{N}{2}}\delta^{1-\frac{m_0+k-\frac{N}{2}}{a_{m_0}}}.
\end{align*}
Owing to
\begin{align*}
\min\limits_{1\leq m_0\leq k-1}1-\frac{m_0+k-\frac{N}{2}}{a_{m_0}}=\frac{2N}{N-2},
\end{align*}
this means exactly that
\begin{align*}
|\ell|\lesssim{\lambda_1}^{m_0+k-\frac{N}{2}}\delta^{\frac{2N}{N-2}}.
\end{align*}
\end{proof}

\section{Differential inequalities of parameters}\label{section 7}
We now suppose that $u$ is a global solution of \eqref{weq} such that
\begin{align*}
\limsup _{t \rightarrow+\infty}\big\|\vec{u}(t)\big\|_{\mathcal{H}_{a}}<\infty.
\end{align*}
Let $v_{L}$ be the unique solution of the free wave equation \eqref{lwp} such that for any $A\in\mathbb{R}$
\begin{align*}
\lim _{t \rightarrow+\infty} \int_{|x| \geq A+|t|}\big|\nabla_{t, x}\left(u-v_{L}\right)(t, x)\big|^{2} d x=0.
\end{align*}
The existence of $v_L$ can be proved as in \cite{CR}.

For $J \geq 1, \iota_j \in\{\pm 1\}^{J},(f, g) \in \mathcal{H}_{a}$, we define its distance to the multi-soliton as
\begin{align*}
d_{J, \iota}(f, g)=\inf_{\lambda \in G_{J}}\bigg\{\Big\|(f, g)-\sum_{j=1}^{J} \iota_{j}\left({W_{a}}_{\left(\lambda_{j}\right)}, 0\right)\Big\|_{\mathcal{H}_{a}}+\gamma(\boldsymbol{\lambda})\bigg\},
\end{align*}

In Section \ref{Soliton resolution along a time sequence}, there exist $J\geq 1$, $\iota \in\{\pm 1\}^{J}$ and a time sequence $\{t_n\}_n$ such that
\begin{align*}
\lim_{n\rightarrow\infty}d_{J,\iota}\big(\vec{u}(t_{n})-\vec{v}_{L}(t_{n})\big)=0.
\end{align*}
In order to prove the soliton resolution for all time, we assume by contradiction that there exist a sequence $\{\tilde{t}_n\}$ satisfying $\tilde{t}_n<t_n$  for any $n$, and  a small number $\varepsilon>0$ such that if $t\in\left(\tilde{t}_{n}, t_{n}\right]$, there holds
\begin{align}\label{distance d t}
d_{J, \iota}\big(\vec{u}(t)-\vec{v}_{L}(t)\big)<\varepsilon_{0},
\end{align}
and
\begin{align}\label{distance d tn}
d_{J,\iota}\big(\vec{u}(\tilde{t}_{n})-\vec{v}_{L}(\tilde{t}_{n})\big)=\varepsilon_{0}.
\end{align}

Since  $h(t)=u(t)-v_{L}(t)-M_{a}(t)$, according to Lemma \ref{lem: modulation}, for $t \in\left[\tilde{t}_{n}, t_{n}\right]$ there exists $\boldsymbol{\lambda}(t)\in G_{J}$ such that the orthogonal conditions
\begin{align}\label{eq: orthogonal condition 1}
\int\big(\partial_{r}+\frac{c}{r}\big)h(t)\cdot\big(\partial_{r}+\frac{c}{r}\big)(\Lambda W_{a})_{(\lambda_{j}(t))}\ud x=0.
\end{align}
hold. Moreover, we also have
\begin{align}\label{def: dJ, gamma}
d_{J, \iota}\big(\vec{u}(t)-v_{L}(t)\big)\approx \big\|\vec{u}(t)-\vec{v}_{L}(t)-(M_{a}(t), 0)\big\|_{\mathcal{H}_{a}}+\gamma(\lambda).
\end{align}
By the definition of $\delta(t)$ \eqref{def: delta}, there holds
$$
 \delta(t)=\sqrt{\big\|h(t)\big\|_{\dot{H}_{a}^{1}}^{2}+\big\|\partial_{t}\left(u-v_{L}\right)(t)\big\|_{L^{2}}^{2}} .
$$
Similar to \eqref{def: g1}, we expand $\partial_{t} U=\partial_{t} u-\partial_{t} v_{L}$ as
\begin{align}\label{the expansion of partial t U}
\partial_{t} U(t)=\sum_{j=1}^{J}\alpha_{j}(t)\iota_{j}\Lambda{W_{a}}_{\left[\lambda_{j}(t)\right]}+g_{1}(t).
\end{align}
where $g_1$ satisfies the orthogonality conditions
\begin{align*}
\int g_{1}(t) \Lambda W_{\left[\lambda_{j}(t)\right]}=0,\quad \forall j \in\llbracket1, J\rrbracket.
\end{align*}
We also define
\begin{align}\label{def: beta j t}
\beta_{j}(t)=&-\iota_{j} \int(\Lambda W_a)_{[\lambda_{j}(t)]} \partial_{t} U(t)\ud x\nonumber\\
=&-\iota_{j}\sum\limits_{k\neq j}\iota_{k}\alpha_{k}(t)\int(\Lambda W_{a})_{[\lambda_{j}(t)]}(\Lambda W_{a})_{[\lambda_{k}(t)]}\ud x-\iota_{j}^{2}\alpha_{j}(t)\big\|\Lambda W_{a}\big\|_{2}^{2}.
\end{align}

\subsection{Non-radiative profile decomposition}
There exists $1=j_{1}<j_{2}<\ldots<j_{K+1}=J+1$ with $\llbracket1, J\rrbracket=\cup_{k=1}^{K}\llbracket j_{k}, j_{k+1}-1\rrbracket$ satisfying the following relationship
\begin{itemize}
\item For all $k\in \llbracket1, K\rrbracket$
\begin{align}\label{eq: in one term}
\nu_{j}=\lim _{n \rightarrow \infty} \frac{\lambda_{j}\left(s_{n}\right)}{\lambda_{j_{k}}\left(s_{n}\right)}>0,\quad \forall j \in \llbracket j_{k}, j_{k+1}-1\rrbracket.
\end{align}
\item For all $k\in \llbracket1, K-1\rrbracket$
\begin{align}\label{eq: not in a term}
\lim _{n \rightarrow \infty} \frac{\lambda_{j_{k+1}}\left(s_{n}\right)}{\lambda_{j_{k}}\left(s_{n}\right)}=0.
\end{align}
\end{itemize}
Thus we divide the interval $\llbracket1, J\rrbracket$ into $K$ sub-intervals.

We note that $\nu_{j_{k}}=1$. Employing the same arguments as in \cite{DKM7}, we can prove

\begin{lemma}\label{lem: nonradiative profile}
For all $k \in\llbracket1, K\rrbracket$, there exists $\left(V_{0}^{k}, V_{1}^{k}\right)\in\mathcal{H}^{a}_{rad}$ such that the solution $V^{k}$ of \eqref{weq} with initial data $\left(V_{0}^{k}, V_{1}^{k}\right)$ is defined on $\{|x|>|t|\}$ and is non-radiative. Furthermore,
Assume in the k-th sub-interval $J^{k}=$ $j_{k+1}-j_{k}, \iota^{k}=\left(\iota_{j_{k}}, \ldots, \iota_{j_{k+1}-1}\right)$ and
\begin{align*}
V_{n}^{k}(t, x)=\frac{1}{\lambda_{j_{k}}^{\frac{N-2}{2}}\left(s_{n}\right)} V^{k}\Big(\frac{t}{\lambda_{j_{k}}\left(s_{n}\right)}, \frac{x}{\lambda_{j_{k}}\left(s_{n}\right)}\Big),
\end{align*}
we have (extracting subsequences if necessary),
\begin{align*}
\lim _{n \rightarrow \infty}\bigg\|\vec{u}\left(s_{n}\right)-\vec{v}_{L}\left(s_{n}\right)-\sum_{k=1}^{K} \vec{V}_{n}^{k}(0)\bigg\|_{\mathcal{H}_{a}}=0,
\end{align*}
and
\begin{align*}
\quad d_{J^{k}, \iota^{k}}\left(V_{0}^{k}, V_{1}^{k}\right) \leq C \varepsilon_{0}.
\end{align*}
More precisely, after extracting a subsequence, $(V_0^{k},V_1^k)$ can be decomposed into
\begin{align}\label{eq: initial date}
\left\{\begin{array}{l}
V_{0}^{k}=\sum\limits_{j=j_{k}}^{j_{k+1}-1} \iota_{j} {W_{a}}_{\left(\nu_{j}\right)}+\bar{h}_{0}^{k}, \\
V_{1}^{k}=\sum\limits_{j=j_{k}}^{j_{k+1}-1} \iota_{j} \bar{\alpha}_{j}(\Lambda W_{a})_{\left[\nu_{j}\right]}+\bar{g}_{1}^{k},
\end{array}\right.
\end{align}
where
\begin{align}\label{eq: weak limit}
\begin{array}{ll}
\bar{h}_{0}^{k}=&\underset{n\rightarrow\infty}{w-\lim}\lambda_{j_{k}}^{\frac{N-2}{2}}(s_{n})h(s_{n}, \lambda_{j_{k}}\left(s_{n}\right) \cdot),\\
\bar{\alpha}_{j}=&\lim\limits_{n\rightarrow\infty}\alpha_{j}\left(s_{n}\right),\\
\bar{g}_{1}^{k}=&\underset{n\rightarrow\infty}{w-\lim}\lambda_{j_{k}}^{\frac{N}{2}}\left(s_{n}\right)g_{1}\left(s_{n}, \lambda_{j_{k}}\left(s_{n}\right) \cdot\right).
\end{array}
\end{align}
Furthermore, we have that
\begin{align}\label{JE Wa}
JE(W_{a},0)=\sum\limits_{k=1}^{K} E\Big(\vec{V}^{k}(0)\Big).
\end{align}
where the energy is defined in \eqref{energy}. Note that the limits \eqref{eq: not in a term} and \eqref{eq: weak limit} imply for any $j$ in the same sub-interval, we have orthogonality conditions as before
\begin{align*}
\int\Big(\partial_{r}+\frac{c}{r}\Big)\bar{h}_{0}^{k}\cdot\Big(\partial_{r}+\frac{c}{r}\Big)(\Lambda W_{a})_{\left(\nu_{j}\right)}\ud x=\int \bar{g}_{1}^{k} \cdot\left(\partial_{r}+\frac{c}{r}\right)(\Lambda W_{a})_{\left[\nu_{j}\right]}\ud x=0.
\end{align*}
\end{lemma}

\subsection{Estimates of the parameters and their derivatives}
Using the profile decomposition in Lemma \ref{lem: nonradiative profile} and the estimates we have obtained in prop \ref{prop: coefficients estimate}, we can prove the following lemma by contradiction and we omit the proof because it is similar to \cite{DKM7}.
\begin{lemma}\label{lem: estimate of the parameters}
For $(N-2)\beta>2$. There exists a constant $C>0$, depending only on $J$ and $N$ and $\beta$, such that under the preceding assumptions,
\begin{align}
\big\|\left(h, g_{1}\right)\big\|_{\mathcal{H}_{a}} \leq& o_{n}(1)+C\Big(\gamma^{\frac{N}{4}\beta}+\delta^{\frac{N}{N-2}}\Big),\quad \forall t \in\left[\tilde{t}_{n}, t_{n}\right],\label{ineq: h0,g1}\\
\Big|\delta^2-\sum\limits_{j=1}^{J}\alpha_{j}^2\|\Lambda W_{a}\|_{2}^{2}\Big|\leq& o_{n}(1)+C\Big(\gamma^{\frac{(N-2)\beta-2}{2}}+\delta^{\frac{2(N-1)}{N-2}}\Big),\quad \forall t \in\left[\tilde{t}_{n}, t_{n}\right],\label{the bdd of alpha and delta}\\
\Big|\beta_{j}+\alpha_{j}\|\Lambda W\|_{L^{2}}^{2}\Big| \leq& o_{n}(1)+C\Big(\gamma^{\frac{(N-2)\beta-2}{4}}+\delta^{\frac{N}{N-2}}\Big). \quad \forall t \in\left[\tilde{t}_{n}, t_{n}\right].\label{the bdd of beta and alpha}
\end{align}
Here $o_{n}(1)$ represents a term that goes to 0 as $n$ goes to infinity, uniformly with respect to $t \in\left[\tilde{t}_{n}, t_{n}\right]$.
\end{lemma}

\begin{lemma}\label{lem: expansion of energy}
For $(N-2)\beta>2$, then
\begin{align}
\big\|(h(t),g_{1}(t))\big\|_{\mathcal{H}_{a}}\lesssim& \gamma^{\frac{N\beta}{4}}+o_{n}(1),\label{the estimate of h,g}\\
\delta\lesssim& \gamma^{\frac{(N-2)\beta}{4}}+o_{n}(1), \label{the estimate of delta}\\
\bigg|\frac{1}{2}\delta^2-\kappa'_{1}\sum\limits_{1\leq j\leq J-1}\iota_{j}\iota_{j+1}\Big(\frac{\lambda_{j+1}}{\lambda_{j}}\Big)^{\frac{(N-2)\beta}{2}}\bigg|\leq&\gamma^{\frac{(N-2)\beta-2}{2}}+o_{n}(1),\label{the bdd of delta and kappa1}
\end{align}

where $o_{n}(1)\rightarrow 0$ as $n \rightarrow \infty$, uniformly with respect to $t\in\left[\tilde{t}_{n},t_{n}\right]$, and

\begin{align*}
\kappa'_{1}=\big(N(N-2)\beta^2\big)^{\frac{N-2}{2}}\int_{\Bbb R^N}\frac{1}{|x|^{\frac{(\beta+1)(N-2)}{2}}}{W_a}^{\frac{N+2}{N-2}}(x)\ud x=\frac{1}{\|\Lambda W_a\|_{L^2}^2}\kappa_{1}.
\end{align*}
\end{lemma}
\begin{proof}
Recall that the energy of \eqref{weq} is given by
\begin{align}\label{def: erengy}
E_{a}[\vec{u}]=\int\frac{1}{2}|\nabla u|^2+\frac{1}{2}|\partial_{t}u|^2+\frac{1}{2}\frac{a}{|x|^2}|u|^2-\frac{N-2}{2N}|u|^{\frac{2N}{N-2}},
\end{align}
where $\vec{u}=(u,\partial_{t}u)$. Then we expand the energy \eqref{def: erengy} of $\vec{u}-\vec{v_{L}}$,
\begin{align*}
E_{a}(\vec{u}-\vec{v_L})=&E_{a}\Big(\sum\limits_{j=1}^{J}\iota_{j}{W_a}_{(\lambda_j)}+h_0,\sum\limits_{j=1}^{J}\alpha_{j}\iota_{j}\left(\Lambda W_a\right)_{[\lambda_j]}+g_1\Big)\\
=&\frac{1}{2}\int_{\Bbb R^N}\Big|\sum\limits_{j=1}^{J}\iota_j\nabla{W_a}_{(\lambda_j)}+\nabla h_0\Big|^2+\Big|\sum\limits_{j=1}^{J}\alpha_j\iota_j(\Lambda W_a)_{[\lambda_j]}+g_1\Big|^2\ud x\\
&+\frac{1}{2}\int_{\Bbb R^N}\frac{a}{|x|^2}\Big|\sum\limits_{j=1}^{J}\iota_j{W_a}_{(\lambda_j)}+h_0\Big|^2\ud x-\frac{N-2}{2N}\int_{\Bbb R^N}\Big|\sum\limits_{j=1}^{J}\iota_j{W_a}_{(\lambda_j)}+h_0\Big|^{\frac{2N}{N-2}}\ud x.
\end{align*}
Going back to the orthogonality \eqref{orthogonality 2}, integration by parts and using the equation
\begin{align*}
-\Delta {W_a}_{(\lambda_j)}+\frac{a}{|x|^2}{W_a}_{(\lambda_j)}=|{W_a}_{(\lambda_j)}|^{\frac{4}{N-2}}{W_a}_{(\lambda_j)},
\end{align*}
this energy can be further rewritten as
\begin{align*}
E_a(\vec{u}-\vec{v_L})=&\frac{J}{2}\big\|W_a\big\|_{\dot{H}^{1}_a}^2+\big\|h_0\big\|_{\dot{H}^{1}_a}^2+\frac{1}{2}\sum\limits_{j=1}^{J}\alpha_j^2\big\|\Lambda W_a\big\|_{L^2}^2+\big\|g_1\big\|_{L^2}^2\\
&+\sum\limits_{j=1}^{J}\iota_{j}\int_{\Bbb R^N}\Big(|{W_a}_{(\lambda_j)}|^{\frac{4}{N-2}}{W_a}_{(\lambda_j)}\Big)h_0\ud x+\sum\limits_{1\leq j,k\leq J}\int_{\Bbb R^N}\Big(|{W_a}_{(\lambda_j)}|^{\frac{4}{N-2}}{W_a}_{(\lambda_j)}\Big){W_a}_{(\lambda_k)}\ud x\\
&+\sum\limits_{1\leq j,k\leq J}\alpha_j\alpha_k\int_{\Bbb R^N}(\Lambda{W_a})_{[\lambda_j]}(\Lambda W_a)_{[\lambda_k]}\ud x-\frac{N-2}{2N}\int_{\Bbb R^N}\Big|\sum\limits_{j=1}^{J}\iota_j{W_a}_{(\lambda_j)}+h_0\Big|^{\frac{2N}{N-2}}\ud x.
\end{align*}
According to \eqref{JE Wa},
\begin{align*}
\lim\limits_{t\rightarrow\infty}E_{a}(\vec{u}-\vec{v_L})=\sum\limits_{k=1}^{K}\lim\limits_{t\rightarrow\infty}E_{a}(\vec{V}^{k}(0))=J E_a(W_a,0).
\end{align*}
Plugging \eqref{nonlinear estimate for y1 to yJ} into the energy term and taking $y_j=\iota_j {W_a}_{(\lambda_j)}$, it follows that
\begin{align*}
&\bigg|\frac{J}{2}\big\|W_a\big\|_{\dot{H}_{a}^1}^{2}+\frac{1}{2}\sum\limits_{j=1}^{J}\alpha_{j}^{2}\big\|\Lambda W_a\big\|_{L^2}^2-J E_{a}(W_a,0)-\frac{N-2}{2N}J\int{W_a}^{\frac{2N}{N-2}}\ud x\\
&+\sum\limits_{1\leq j<k\leq J}\iota_j\iota_k\int_{\Bbb R^N} {W_a}_{(\lambda_j)}^{\frac{N+2}{N-2}}{W_a}_{(\lambda_k)}\ud x\bigg|\\
\lesssim&
\sum_{1\leq j<k\leq J}\int_{\Bbb R^N}\bigg(\min\Big\{{W_a}_{(\lambda_{j})}^{\frac{4}{N-2}}{W_a}_{(\lambda_{k})}^{2}, {W_a}_{(\lambda_{k})}^{\frac{4}{N-2}} {W_a}_{(\lambda_{j})}^{2}\Big\}+\min\Big\{{W_a}_{(\lambda_{j})}^{\frac{N+2}{N-2}} {W_a}_{(\lambda_{k})}, {W_a}_{(\lambda_{k})}^{\frac{N+2}{N-2}} {W_a}_{(\lambda_{j})}\Big\}\bigg)\ud x \\
&+\big\|g_1\big\|_{L^2}^2+\big\|h_0\big\|_{\dot{H}_{a}^1}^2+\gamma^{\frac{(N-2)\beta-2}{2}}\delta^2+o_{n}(1),
\end{align*}
where we have used the fact $\int|(\Lambda{W_a})_{[\lambda_j]}(\Lambda{W_a})_{[\lambda_k]}|\lesssim\gamma^{\frac{(N-2)\beta-2}{2}}$.
On one hand, by the inequality \eqref{ineq: Wa N N-2} we get
\begin{align*}
&\int_{\Bbb R^N}\min\Big\{W_{(\lambda_{j})}^{\frac{4}{N-2}}{W_a}_{(\lambda_{k})}^{2}, {W_a}_{(\lambda_{k})}^{\frac{4}{N-2}} {W_a}_{(\lambda_{j})}^{2}\Big\}\ud x+\int_{\Bbb R^N}\min\Big\{{W_a}_{(\lambda_{j})}^{\frac{N+2}{N-2}} {W_a}_{(\lambda_{k})}, {W_a}_{(\lambda_{k})}^{\frac{N+2}{N-2}} {W_a}_{(\lambda_{j})}\Big\}\ud x\\
\lesssim& \int_{\Bbb R^N}{W_a}_{(\lambda_j)}^{\frac{N}{N-2}}{W_a}_{(\lambda_k)}^{\frac{N}{N-2}}\ud x\lesssim\gamma^{\frac{N\beta}{2}}.
\end{align*}
On the other hand, using the estimate \eqref{ineq: h0,g1} on $(h_0,g_1)$, we obtain
\begin{align}\label{the bdd of alpha and Wa}
\bigg|\frac{1}{2}\sum\limits_{j=1}^{J}\alpha_{j}^2\big\|\Lambda W_a\big\|_{L^2}^2+\sum\limits_{1\leq j<k\leq J}\iota_j\iota_k&\int_{\Bbb R^N} {W_a}_{(\lambda_j)}^{\frac{N+2}{N-2}}{W_a}_{(\lambda_k)}\ud x\bigg|\nonumber\\
\lesssim&\gamma^{\frac{N\beta}{2}}+\delta^{\frac{2N}{N-2}}+\gamma^{\frac{(N-2)\beta-2}{2}}\delta^2+o_{n}(1).
\end{align}

Now it remains to estimate $\int {W_a}_{(\lambda_j)}^{\frac{N+2}{N-2}}{W_a}_{(\lambda_k)}\ud x$.
For $j<k$, by the relationship of $\{\lambda_j\}_{1\leq j\leq J}$, we have $\lambda_k<\lambda_j$ and $\frac{\lambda_j}{\lambda_k}\leq\gamma^{-(k-j)}$. Thus
\begin{align*}
\int_{\Bbb R^N}{W_a}_{(\lambda_j)}^{\frac{N+2}{N-2}}{W_a}_{(\lambda_k)}\ud x=&\Big(\frac{\lambda_j}{\lambda_k}\Big)^{\frac{N-2}{2}}\int_{\Bbb R^N}{W_a}^{\frac{N+2}{N-2}}(x){W_a}\Big(\frac{\lambda_j x}{\lambda_k}\Big)\ud x.
\end{align*}
According to \eqref{eq: the formula of Wa},
\begin{align*}
\bigg|W_a(x)-\frac{\big(N(N-2)\beta^2\big)^{\frac{N-2}{4}}}{|x|^{\frac{(\beta+1)(N-2)}{2}}}\bigg|= O\bigg(\frac{1}{|x|^{\frac{3}{2}\beta+(N-2)}}\bigg),
\end{align*}
which leads to
\begin{align}\label{the bdd of Wa N}
\int_{\Bbb R^N}{W_a}_{(\lambda_j)}^{\frac{N+2}{N-2}}{W_a}_{(\lambda_k)}\ud x=&\Big(\frac{\lambda_k}{\lambda_j}\Big)^{\frac{\beta(N-2)}{2}}\int_{\Bbb R^N}\frac{(N(N-2)\beta^2)^{\frac{N-2}{4}}}{|x|^{\frac{(\beta+1)(N-2)}{2}}}{W_a}^{\frac{N+2}{N-2}}(x)\ud x\nonumber\\
&+O\bigg(\Big(\frac{\lambda_k}{\lambda_j}\Big)^{\frac{3\beta+(N-2)}{2}}\int_{\Bbb R^N}\frac{1}{|x|^{\frac{3}{2}\beta+(N-2)}}{W_a}^{\frac{N+2}{N-2}}(x)\ud x\bigg).
\end{align}
Therefore, if $j<k-1$, we get the bound
\begin{align*}
\int_{\Bbb R^N}{W_a}_{(\lambda_j)}^{\frac{N+2}{N-2}}{W_a}_{(\lambda_k)}\ud x\lesssim\gamma^{(N-2)\beta}.
\end{align*}
Combining \eqref{the bdd of alpha and delta} and \eqref{the bdd of alpha and Wa} we obtain
\begin{align*}
\delta^2\lesssim\gamma^{\frac{(N-2)\beta}{2}}+\gamma^{\frac{(N-2)\beta-2}{2}}\delta^2+\delta^{\frac{2(N-1)}{N-2}}+o_{n}(1),
\end{align*}
which yields $\delta\lesssim \gamma^{\frac{(N-2)\beta}{4}}+o_{n}(1)$. Hence
\begin{align*}
\bigg|\frac{1}{2}\delta^2-\big(N(N-2)\beta^2\big)^{\frac{N-2}{2}}\int_{\Bbb R^N}\frac{1}{|x|^{\frac{(\beta+1)(N-2)}{2}}}{W_a}^{\frac{N+2}{N-2}}(x)\ud x& \sum\limits_{1\leq j\leq J+1}\iota_j\iota_{j+1}\Big(\frac{\lambda_{j+1}}{\lambda_{j}}\Big)^{\frac{(N-2)\beta}{2}}\bigg|\\
&\lesssim\gamma^{\frac{(N-2)\beta-2}{2}}+o_{n}(1).
\end{align*}
Denote
\begin{align*}
\kappa'_{1}=\big(N(N-2)\beta^2\big)^{\frac{N-2}{2}}\int_{\Bbb R^N}\frac{1}{|x|^{\frac{(\beta+1)(N-2)}{2}}}{W_a}^{\frac{N+2}{N-2}}(x)\ud x,
\end{align*}
then we obtain \eqref{the bdd of delta and kappa1}.
\end{proof}

We set $U(t)=u(t)-v_{L}(t)$, and
\begin{align*}
h(t)=U(t)-M_a(t)=U(t)-\sum\limits_{j=1}^{J}\iota_{j}{W_{a}}_{(\lambda_{j}(t))}.
\end{align*}
If we denote the nonlinear term $F(u)=|u|^{\frac{4}{N-2}}u$, then $\left(h(t), \partial_{t} U(t)\right)$ satisfies the following system for $t\in\left[\tilde{t}_{n}, t_{n}\right]$,
\begin{align}\label{eq: function of h}
\left\{\begin{array}{l}
\frac{\partial h}{\partial t}=\frac{\partial U}{\partial t}+\sum\limits_{j=1}^{J}\iota_{j}\lambda_{j}^{\prime}(t)(\Lambda W_{a})_{[\lambda_{j}(t)]}\\
\frac{\partial}{\partial t}\left(\frac{\partial U}{\partial t}\right)-\Delta h+\frac{a}{|x|^2}h=F(U)-\sum\limits_{j=1}^{J} F\left(\iota_{j} {W_{a}}_{(\lambda_{j})}\right)+N\left(h, v_{L}\right),
\end{array}\right.
\end{align}
where
\begin{align*}
N\left(h, v_{L}\right)\triangleq F\left(M+h+v_{L}\right)-F(M+h),
\end{align*}
satisfies
\begin{align}\label{eq: estimate of sigma}
\big|N\left(h, v_{L}\right)\big| \lesssim\big|v_{L}(t)\big|^{\frac{N+2}{N-2}}+\sum_{j=1}^{J}\Big({W_{a}}_{\left(\lambda_{j}\right)}^{\frac{4}{N-2}}+|h(t)|^{\frac{4}{N-2}}\Big)\big|v_{L}(t)\big|.
\end{align}
Employing the orthogonality condition, we obtain the estimate on the first order derivative of $\lambda_j$.

\begin{lemma}\label{lem: lambeda first deravitive estimate}
\begin{align}\label{the bdd of alpha and gamma}
\left|\lambda_{j}^{\prime}+\alpha_{j}\right|\lesssim\gamma^{\frac{(N-1)\beta}{4}}+o_{n}(1).
\end{align}
where $o_{n}(1)$ represents a term goes to 0 as $n \rightarrow \infty$, uniformly with respect to $t \in\left[\tilde{t}_{n}, t_{n}\right]$.
\end{lemma}

It is simple to show by the orthogonality condition \eqref{eq: orthogonal condition 1} and \eqref{2.13}.

\medskip
\begin{lemma}.\label{lem: lambda second deravitive estimate}
For $(N-2)\beta>2$, for all $j \in \llbracket1, J\rrbracket$,
\begin{align}\label{eq: lambda second deravitive estimate}
\bigg|\lambda_{j} \beta_{j}^{\prime}+\kappa_{0}\Big(\iota_{j} \iota_{j+1}\Big(\frac{\lambda_{j+1}}{\lambda_{j}}\Big)^{\frac{(N-2)\beta}{2}}-\iota_{j} \iota_{j-1}\Big(\frac{\lambda_{j}}{\lambda_{j-1}}\Big)^{\frac{(N-2)\beta}{2}}\Big)\bigg| \lesssim \gamma^{\frac{N\beta}{2}}+o_{n}(1),
\end{align}
where
\begin{align*}
\kappa_{0}=\frac{(N(N-2))^{\frac{N-2}{4}}\beta^{\frac{N}{2}}}{2}\int\frac{1}{|x|^{\frac{(\beta+1)(N-2)}{2}}}W_{a}^{\frac{N+2}{N-2}}(x)\ud x.
\end{align*}
\end{lemma}
\begin{proof}
According to \eqref{def: beta j t}, $\beta_{j}(t)=-\iota_{j}\int(\Lambda W_{a})_{[\lambda_{j}(t)]}\partial_{t}U(t)\ud x$. Differentiating in $t$ gives
\begin{align}
\lambda_{j}\beta^{\prime}_{j}(t)=&-\iota_{j}\int\lambda_{j}^{-\frac{N}{2}}\lambda^{\prime}_{j}\Big(-\frac{N}{2}+\frac{x}{\lambda_{j}}\cdot\nabla\Big)(\Lambda W_{a})(\frac{x}{\lambda_{j}})\partial_{t}U(t)\ud t\nonumber\\
&-\iota_{j}\lambda_{j}\int(\Lambda W_{a})_{[\lambda_{j}]}\Big(\Delta h-\frac{a}{|x|^2}h+F(h)-\sum\limits_{j=1}^{J}F(\iota {W_{a}}_{\lambda_{j}})+\sigma(h,v_{L})\Big)\ud x\nonumber\\
=&-\iota_{j}\int\lambda^{\prime}_{j}(\Lambda_{0}\Lambda W_{a})_{[\lambda_{j}]}\partial_{t}U(t)\ud t\label{first term of sum}\\
&-\iota_{j}\int(\Lambda W_{a})_{(\lambda_{j})}\Big(\Delta h-\frac{a}{|x|^2}h+F(h)-\sum\limits_{j=1}^{J}F(\iota {W_{a}}_{\lambda_{j}})+\sigma(h,v_{L})\Big)\ud x\label{second term of sum}.
\end{align}
Our first goal is to show the bound of the first term \eqref{first term of sum} of this summation. Using the expansion \eqref{the expansion of partial t U} of $\partial_{t}U$, we obtain
\begin{align*}
\int(\Lambda_0\Lambda W_a)_{[\lambda_j]}\partial_{t}U\ud x=&\int(\Lambda_0\Lambda W_a)_{[\lambda_j]}g_1\ud x+\iota_j\alpha_j\int(\Lambda_0\Lambda W_a)_{[\lambda_j]}(\Lambda W_a)_{[\lambda_j]}\ud x\\
&+\sum\limits_{k\neq j}\iota_k\alpha_k\int(\Lambda_0\Lambda W_a)_{[\lambda_j]}(\Lambda W_a)_{[\lambda_k]}\ud x.
\end{align*}
Hence, \eqref{the estimate of h,g} gives
\begin{align*}
\int(\Lambda_0\Lambda W_a)_{[\lambda_j]}g_1\ud x\lesssim\big\|\Lambda_0\Lambda W_a\big\|_{L^2}\big\|g_1\big\|_{L^2}\lesssim\gamma^{\frac{N\beta}{4}}+o_{n}(1).
\end{align*}
Integration by parts shows that the second term equals $0$. Using \eqref{the bdd of Lambda Wa and Lambda Wa} we get
\begin{align*}
\int(\Lambda_0\Lambda W_a)_{[\lambda_j]}(\Lambda W_a)_{[\lambda_k]}\ud x\lesssim\gamma^{\frac{(N-2)\beta}{2}-1}.
\end{align*}
By the estimate of $\lambda'$ in Lemma \ref{lem: lambeda first deravitive estimate} and the assumption $(N-2)\beta\geq 1$, we obtain
\begin{align}\label{eq: lambda beta' first order}
\left|\int\lambda^{\prime}_{j}(\Lambda_{0}\Lambda W_{a})_{[\lambda_{j}]}\partial_{t}U(t)\ud t\right|\lesssim\gamma^{\frac{(3N-5)\beta}{4}-1}+o_{n}(1)\leq\gamma^{\frac{(N-2)\beta}{2}-1}+o_{n}(1).
\end{align}

It remains to estimate \eqref{second term of sum}. Plugging \eqref{eq: function of h} into \eqref{second term of sum} gives
\begin{align}
&\lambda_{j}\int(\Lambda W_{a})_{[\lambda_{j}]}\partial_{t}^2U\ud x\nonumber\\\label{eq:0}
=&-\int(\Lambda W_{a})_{[\lambda_{j}]}\Big(-\Delta+\frac{a}{|x|^2}-\frac{N+2}{N-2}{W_{a}}_{(\lambda_{j})}^{\frac{4}{N-2}}\Big)h\ud x\\\label{eq:1}
&+\int(\Lambda W_{a})_{(\lambda_{j})}\sigma(h,v_{L})\ud x\\\label{eq:2}
&+\int(\Lambda W_{a})_{(\lambda_{j})}\Big(F(\iota_{j}{W_{a}}_{(\lambda_{j})}+h)-F(\iota_{j}{W_{a}}_{(\lambda_{j})})-\frac{N+2}{N-2}{W_{a}}_{(\lambda_{j})}^{\frac{4}{N-2}}h\Big)\ud x\\\label{eq:3}
&+\int(\Lambda W_{a})_{(\lambda_{j})}\Big(F(M+h)-F(M)+F(\iota_j {W_{a}}_{(\lambda_{j})})-F(\iota_j {W_{a}}_{(\lambda_{j})}+h)\Big)\ud x\\\label{eq:4}
&+\int(\Lambda W_{a})_{(\lambda_{j})}\Big(F(M)-\sum\limits_{j=1}^{J}F(\iota_{k}{W_{a}}_{(\lambda_{k})})\Big)
\ud x.
\end{align}
Notice that the operator $\mathcal{L}_{W(\lambda_{j})}^{a}=-\Delta+\frac{a}{|x|^2}-\frac{N+2}{N-2}{W_{a}}_{(\lambda_{j})}^{\frac{4}{N-2}}$ is self-adjoint. It follows that
\begin{align*}
\eqref{eq:0}=\int(\Lambda W_{a})_{[\lambda_{j}]}\mathcal{L}_{W(\lambda_{j})}^{a}h\ud x=\int\mathcal{L}_{W(\lambda_{j})}^{a}(\Lambda W_{a})_{[\lambda_{j}]}h\ud x=0.
\end{align*}

From \eqref{eq: estimate of sigma} and H\"{o}lder's inequality, we have
\begin{align*}
|\eqref{eq:1}|\lesssim&\int(\Lambda W_{a})_{(\lambda_{j})}|v_{L}|^{\frac{N+2}{N-2}}\ud x+\sum\limits_{k=1}^{J}\int(\Lambda W_{a})_{(\lambda_{j})}{W_{a}}_{(\lambda_{k})}^{\frac{4}{N-2}}|v_{L}(t)|\ud x\\
&+J\int(\Lambda W_{a})_{(\lambda_{j})}|h(t)|^{\frac{4}{N-2}}|v_{L}(t)|\ud x\\
\lesssim&\|v_{L}\|_{L^{\frac{2N}{N-2}}}^{\frac{N+2}{N-2}}+\|v_{L}\|_{L^{\frac{2N}{N-2}}}.
\end{align*}
Since $v_{L}$ is a solution to \eqref{lwp}, we have $\lim\limits_{t\rightarrow\infty}\|v_{L}(t)\|_{L^{\frac{2N}{N-2}}}=0$, which implies that $|\eqref{eq:1}|=o_n(1)$.

According to the definition of nonlinear term $F(h)$, \eqref{eq: 2 nonlinear estimate} and
\begin{align}\label{eq: estimate of Lambda Wa}
\big|(\Lambda W_{a})_{(\lambda_{j})}\big|=\lambda_{j}^{-\frac{N-2}{2}}\bigg|\Big(\frac{N-2}{2}W_{a}+x\cdot\nabla W_{a}\Big)(\frac{x}{\lambda_{j}})\bigg|\lesssim\big|{W_{a}}_{(\lambda_{j})}\big|,
\end{align}
then by \eqref{eq: 2 nonlinear estimate} and Sobolev's inequality, we obtain
\begin{align}\label{the bdd of 2}
\eqref{eq:2}\lesssim&\int\big|{W_{a}}_{(\lambda_{j})}\big|h^2\chi_{\{|h|\leq{W_{a}}_{(\lambda_{j})}\}}{W_{a}}_{(\lambda_{j})}^{\frac{6-N}{N-2}}\ud x+\int\big|{W_{a}}_{(\lambda_{j})}\big|h^{\frac{N+2}{N-2}}\chi_{\{|h|>{W_{a}}_{(\lambda_{j})}\}}\ud x\nonumber\\
\lesssim&\int h^2{W_{a}}_{(\lambda_{j})}^{\frac{4}{N-2}}(x)\ud x+\int\big|h\big|^{\frac{2N}{N-2}}\ud x\nonumber\\
\lesssim&\gamma^{\frac{N\beta}{2}}+o_{n}(1).
\end{align}

Recalling \eqref{eq: estimate of F2}, the estimates of $\eqref{eq:3}$ are divided into two cases.
\begin{itemize}
\item If $N\geq6$.
We choose $a=\iota_{j}{W_{a}}_{(\lambda_j)}$, $b=\sum\limits_{k\neq j}\iota_{k}{W_{a}}_{(\lambda_{k})}$ and $c=h$ in \eqref{eq: estimate of F2}. From \eqref{eq: estimate of Lambda Wa} and H\"{o}lder's inequality, we obtain the estimate
\begin{align*}
&\bigg|\int(\Lambda W_{a})_{(\lambda_{j})}\Big[F(M+h)-F(M)+F(\iota_j {W_{a}}_{(\lambda_{j})})-F(\iota_j {W_{a}}_{(\lambda_{j})}+h)\Big]\ud x\bigg|\\
\lesssim&\sum\limits_{j\neq k}\int {W_{a}}_{(\lambda_{j})}^{\frac{N+2}{2(N-2)}}|h|{W_{a}}_{(\lambda_{k})}^{\frac{N+2}{2(N-2)}}\ud x\lesssim\|h\|_{L^{\frac{2N}{N-2}}}\sum\limits_{k\neq j}\Big(\int {W_{a}}_{(\lambda_{k})}^{\frac{N}{N-2}}{W_{a}}_{(\lambda_{j})}^{\frac{N}{N-2}}\Big)^{\frac{N+2}{2N}}\\
\lesssim&\gamma^{\frac{(N+1)\beta}{2}}+o_{n}(1).
\end{align*}
\item If $N\leq 5$, taking $a=\iota_{j}{W_{a}}_{(\lambda_j)}$, $b=\sum\limits_{k\neq j}\iota_{k}{W_{a}}_{(\lambda_{k})}$ and $c=h$, in \eqref{eq: estimate of F2} and by \eqref{2.12}, we have
\begin{align*}
&\bigg|\int(\Lambda W_{a})_{(\lambda_{j})}\Big[F(M+h)-F(M)+F(\iota_j {W_{a}}_{(\lambda_{j})})-F(\iota_j {W_{a}}_{(\lambda_{j})}+h)\Big]\ud x\bigg|\\
\lesssim&\bigg|\int(\Lambda W_{a})_{(\lambda_j)}\Big[\sum\limits_{k\neq j}{W_{a}}_{(\lambda_{k})}|h|\big({W_{a}}_{(\lambda_{j})}+\sum\limits_{k\neq j}{W_{a}}_{(\lambda_{k})}+|h|\big)^{\frac{6-N}{N-2}}\Big]\ud x\bigg|
\end{align*}
\begin{align*}
\lesssim&\int {W_{a}}_{(\lambda_{j})}\bigg[\sum\limits_{k\neq j}{W_{a}}_{(\lambda_{k})}|h|\big({W_{a}}_{(\lambda_{j})}^{\frac{6-N}{N-2}}+\sum\limits_{k\neq j}{W_{a}}_{(\lambda_{k})}^{\frac{6-N}{N-2}}+|h|^{\frac{6-N}{N-2}}\big)\bigg]\ud x\\
\lesssim&\sum\limits_{k\neq j}\Big[\int{W_{a}}_{(\lambda_{j})}^{\frac{4}{N-2}}|h|{W_{a}}_{(\lambda_{k})}+\int{W_{a}}_{(\lambda_j)}{W_a}_{(\lambda_k)}^{\frac{4}{N-2}}|h|+\int{W_a}_{(\lambda_j)}{W_a}_{(\lambda_k)}|h|^{\frac{4}{N-2}}\Big]\\
\lesssim&\big\|h\big\|_{L^{\frac{2N}{N-2}}}\sum\limits_{1\leq k,l\leq J,k\neq l}\big\|{W_a}_{(\lambda_{\ell})}{W_a}_{(\lambda_k)}^{\frac{4}{N-2}}\big\|_{L^{\frac{2N}{N+2}}}+\big\|h\big\|_{\frac{2N}{N-2}}^{\frac{4}{N-2}}\sum\limits_{1\leq k,l\leq J,k\neq l}\Big(\int{W_a}_{(\lambda_k)}^{\frac{N}{N-2}}{W_a}_{(\lambda_{\ell})}^{\frac{N}{N-2}}\ud x\Big)^{\frac{N-2}{N}}\\
\leq&\gamma^{\frac{(3N-4)\beta}{4}}+\gamma^{\frac{N^2\beta}{2(N-2)}-\beta}+o_{n}(1)\lesssim\gamma^{\frac{(3N-4)\beta}{4}}+o_n(1).
\end{align*}
\end{itemize}

Since $\gamma\ll1$, collecting the above estimates, we get
\begin{align}\label{the bound of 3}
|\eqref{eq:3}|\leq\max\Big\{\gamma^{\frac{(N+1)\beta}{2}},\gamma^{\frac{(3N-4)\beta}{4}}\Big\}+o_{n}(1).
\end{align}
As a consequence of \eqref{the bdd of 2} and \eqref{the bound of 3}, for $(N-2)\beta>2$, we have
\begin{align}\label{the bdd of the easy term}
|\eqref{eq:0}|+|\eqref{eq:1}|+|\eqref{eq:2}|+|\eqref{eq:3}|\lesssim\gamma^{\frac{N\beta}{2}}+o_{n}(1).
\end{align}

We finally return to the leading term,
\begin{align*}
\eqref{eq:4}
=&\int(\Lambda W_{a})_{(\lambda_{j})}\Big(F\big(\sum\limits_{k\neq j}\iota_{k}{W_{a}}_{(\lambda_k)}+\iota_{j}{W_{a}}_{(\lambda_j)}\big)-F(\iota_{j}{W_{a}}_{(\lambda_j)})\Big)\ud x\\
=&\int(\Lambda W_{a})_{(\lambda_{j})}\Big(F\big(\sum\limits_{k\neq j}\iota_{k}{W_{a}}_{(\lambda_k)}+\iota_{j}{W_{a}}_{(\lambda_j)}\big)-F(\iota_{j}{W_{a}}_{(\lambda_j)})-F'(\iota_{j}{W_{a}}_{(\lambda_j)})\sum\limits_{k\neq j}\iota_{k}{W_{a}}_{(\lambda_k)}\Big)\ud x\\
&+\int(\Lambda W_{a})_{(\lambda_j)}F'(\iota_j{W_a}_{(\lambda_j)})\sum\limits_{k\neq j}\iota_k{W_a}_{(\lambda_k)}\ud x.
\end{align*}
By \eqref{eq: estimate of Lambda Wa}, letting $y=x/\lambda_{j}$, we have
\begin{align*}
&\int(\Lambda W_{a})_{(\lambda_j)}F'(\iota_j{W_a}_{(\lambda_j)})\sum\limits_{k\neq j}\iota_k{W_a}_{(\lambda_k)}\ud x\lesssim\sum\limits_{k\neq j}\int(\Lambda W_a)_{(\lambda_j)}{W_a}_{(\lambda_j)}^{\frac{4}{N-2}}{W_{a}}_{(\lambda_k)}\ud x\\
=&\sum\limits_{k\neq j}\Big(\frac{\lambda_j}{\lambda_k}\Big)^{\frac{N-2}{2}}\int{W_a}^{\frac{N+2}{N-2}}(y)W_{a}\Big(\frac{\lambda_j}{\lambda_k}y\Big)\ud y\lesssim\sum\limits_{k\neq j}\Big(\frac{\lambda_k}{\lambda_j}\Big)^{\frac{(N-2)\beta}{2}}\int y^{-\frac{(\beta+1)(N-2)}{2}}W_a^{\frac{N+2}{N-2}}(y)\ud y.
\end{align*}
Therefore, it suffices to consider the case that $k=j\pm1$. To conclude the proof, it remains to prove
\begin{align}\label{eq: the second derivative of U}
&\bigg|\int_{\mathbb{R}^N}\Big(F(M)-\sum\limits_{k=1}^{J}F\big(\iota_k {W_a}_{(\lambda_k)}\big)\Big)(\Lambda{W_a})_{(\lambda_j)}\ud x\nonumber\\
&-\kappa_0\Big(\iota_{j+1}\Big(\frac{\lambda_{j+1}}{\lambda_j}\Big)^{\frac{N-2}{2}}-\iota_{j-1}\Big(\frac{\lambda_j}{\lambda_{j-1}}\Big)^{\frac{N-2}{2}}\Big)\bigg|\lesssim\gamma^{\frac{N\beta}{2}}.
\end{align}

In order to prove this claim, it is equivalent to prove the following inequalities:
\begin{itemize}
\item For $1\leq j\leq J$,
\begin{align}\label{eq: the 1 bdd}
&\qquad\bigg|\int_{\mathbb{R}^{N}}(\Lambda{W_a})_{(\lambda_j)}\Big(F(M)-\sum\limits_{k=1}^{J}F(\iota_k{W_a}_{(\lambda_k)})\Big)\ud x\nonumber\\
-&\frac{N+2}{N-2}\int_{\mathbb{R}^N}(\Lambda{W_a})_{(\lambda_j)}\Big(\iota_{j+1}{W_a}_{(\lambda_j)}^{\frac{4}{N-2}}{W_a}_{(\lambda_{j+1})}+\iota_{j-1}{W_a}_{(\lambda_j)}^{\frac{4}{N-2}}{W_a}_{(\lambda_{j-1})}\Big)\ud x\bigg|\lesssim\gamma^{\frac{N\beta}{2}}.
\end{align}
\item For $1\leq j\leq J-1$,
\begin{align}\label{eq: the 2 bdd}
&\qquad\bigg|\int_{\mathbb{R}^{N}}{W_a}_{(\lambda_j)}^{\frac{4}{N-2}}{W_a}_{(\lambda_{j+1})}(\Lambda W_a)_{(\lambda_j)}\nonumber\\
-&\Big(\frac{\lambda_{j+1}}{\lambda_j}\Big)^{\frac{N-2}{2}}\frac{N^{\frac{N-2}{4}}(N-2)^{\frac{N+6}{4}}\beta^{\frac{N}{2}}}{2(N+2)}\int|x|^{-\frac{(\beta+1)(N-2)}{2}}W_{a}^{\frac{N+2}{N-2}}(x)\ud x\bigg|\lesssim\gamma^{\frac{N\beta}{2}}.
\end{align}
\item For $2\leq j\leq J$,
\begin{align}\label{eq: the 3 bdd}
&\qquad\bigg|\int_{\mathbb{R}^{N}}{W_a}_{(\lambda_j)}^{\frac{4}{N-2}}{W_a}_{(\lambda_{j-1})}(\Lambda W_a)_{(\lambda_j)}\nonumber\\
+&\Big(\frac{\lambda_{j}}{\lambda_{j-1}}\Big)^{\frac{N-2}{2}}\frac{N^{\frac{N-2}{4}}(N-2)^{\frac{N+6}{4}}\beta^{\frac{N}{2}}}{2(N+2)}\int|x|^{-\frac{(\beta+1)(N-2)}{2}}W_{a}^{\frac{N+2}{N-2}}(x)\ud x\bigg|\lesssim\gamma^{\frac{N\beta}{2}}.
\end{align}
\end{itemize}

Because $W_{a}(x)\sim|x|^{\frac{(\beta-1)(N-2)}{2}}(1+|x|^{2\beta})^{-\frac{N-2}{2}}$, ${W_{a}}_{(\lambda_{j})}<{W_{a}}_{(\lambda_{k})}$ is equivalent to
\begin{align*}
&\lambda_{j}^{-\frac{(N-2)\beta}{2}}|x|^{\frac{(\beta-1)(N-2)}{2}}\Big(\lambda_{j}^{\frac{\beta}{2}}+|x|^{\beta}\lambda_{j}^{-\frac{\beta}{2}}\Big)^{-(N-2)}\lambda_{j}^{\frac{(N-2)\beta}{2}}\\
\lesssim&\lambda_{k}^{-\frac{(N-2)\beta}{2}}|x|^{\frac{(\beta-1)(N-2)}{2}}\Big(\lambda_{k}^{\frac{\beta}{2}}+|x|^{\beta}\lambda_{k}^{-\frac{\beta}{2}}\Big)^{-(N-2)}\lambda_{k}^{\frac{(N-2)\beta}{2}}.
\end{align*}
The fact $(\beta-1)(N-2)>0$ means that
\begin{align*}
|x|^{\beta}\Big(\lambda_{k}^{-\frac{\beta}{2}}-\lambda_{j}^{-\frac{\beta}{2}}\Big)\lesssim\lambda_j^{\frac{\beta}{2}}-\lambda_{k}^{\frac{\beta}{2}},
\end{align*}
which can be rewritten as
\begin{align}\label{relationship of lambda k and lambda j}
\left\{\begin{array}{l}
|x|^{\beta}<(\lambda_j\lambda_k)^{\frac{\beta}{2}},\qquad if \quad\lambda_j>\lambda_k,\\
|x|^{\beta}>(\lambda_j\lambda_k)^{\frac{\beta}{2}},\qquad if \quad\lambda_j<\lambda_k.
\end{array}\right.
\end{align}
Then for any $k,j\in\llbracket1, J\rrbracket$, if $|x|\in\Big(\lambda_{k}^{\frac{1}{2}}\lambda_{k+1}^{\frac{1}{2}},\lambda_{k}^{\frac{1}{2}}\lambda_{k-1}^{\frac{1}{2}}\Big)$, there exists a constant $C$ such that ${W_a}_{(\lambda_j)}\lesssim{W_{a}}_{(\lambda_{k})}$.

{\bf Proof of \eqref{eq: the 1 bdd}}. Substituting \eqref{eq: the 1 bdd} into \eqref{relationship of lambda k and lambda j}, we obtain
\begin{align*}
&\int_{\mathbb{R}^{N}}(\Lambda{W_a})_{(\lambda_j)}\Big(F(M_a)-\sum\limits_{k=1}^{J}F(\iota_k{W_a}_{(\lambda_k)})\Big)\ud x\\
&-\frac{N+2}{N-2}\int_{\mathbb{R}^N}(\Lambda{W_a})_{(\lambda_j)}\Big(\iota_{j+1}{W_a}_{(\lambda_j)}^{\frac{4}{N-2}}{W_a}_{(\lambda_{j+1})}+\iota_{j-1}{W_a}_{(\lambda_j)}^{\frac{4}{N-2}}{W_a}_{(\lambda_{j-1})}\Big)\ud x\\
\triangleq&\int_{\mathbb{R}^{N}}P_{j}(x)\ud x=\sum\limits_{k=1}^{J}\int_{\sqrt{\lambda_{k}\lambda_{k+1}}\leq|x|\leq\sqrt{\lambda_k\lambda_{k-1}}}P_{j}(x)\ud x.
\end{align*}

If $j\neq k$, we have
\begin{align*}
\Big|F(M)-F\big(\iota_k{W_a}_{(\lambda_k)}\big)\Big|=\Big|F\Big(\sum\limits_{\ell=1}^{J}\iota_{\ell}{W_a}_{(\lambda_{\ell})}\Big)-F\big(\iota_k{W_a}_{(\lambda_k)}\big)\Big|
\lesssim\sum\limits_{\ell\neq k}{W_a}_{(\lambda_{\ell})}{W_a}_{(\lambda_k)}^{\frac{4}{N-2}}.
\end{align*}

We define the domain $B_k=\big\{x:\sqrt{\lambda_{k}\lambda_{k+1}}\leq|x|\leq\sqrt{\lambda_k\lambda_{k-1}}\big\}$. Since $j\neq k$, \eqref{relationship of lambda k and lambda j} implies ${W_a}_{(\lambda_{\ell})}\leq {W_a}_{(\lambda_k)}$ on $B_k$. Thus,
\begin{align*}
&\int_{B_k}|P_{j}(x)|\ud x\nonumber\\
\leq&\int_{B_k}{W_a}_{(\lambda_j)}\Big|F(M)-F(\iota_k{W_a}_{(\lambda_k)})-\sum\limits_{k\neq \ell}F(\iota_{\ell}{W_a}_{(\lambda_{\ell})})\\
&-\frac{N+2}{N-2}\Big(\iota_{j+1}{W_a}_{(\lambda_{j+1})}+\iota_{j-1}{W_a}_{(\lambda_{j-1})}\Big){W_a}_{(\lambda_j)}^{\frac{4}{N-2}}\Big|\ud x\\
\lesssim&\sum\limits_{\ell\neq k}\int_{B}{W_a}_{(\lambda_j)}{W_a}_{(\lambda_k)}^{\frac{4}{N-2}}{W_a}_{(\lambda_{\ell})}{W_a}_{(\lambda_j)}\ud x+\sum\limits_{l\neq k}\int_{B}{W_a}_{(\lambda_j)}{W_a}_{(\lambda_{\ell})}^{\frac{N+2}{N-2}}{W_a}_{(\lambda_j)}\ud x\\
&+\int_{B_k}{W_a}_{(\lambda_j)}^{\frac{N+2}{N-2}}\big({W_a}_{(\lambda_{j+1})}+{W_a}_{(\lambda_{j-1})}\big)\ud x\\
\lesssim&\sum\limits_{\ell\neq k}\int_{B_k}{W_a}_{(\lambda_j)}{W_a}_{(\lambda_{\ell})}^{\frac{N}{N-2}}{W_a}_{(\lambda_k)}^{\frac{N}{N-2}}\ud x+\int_{B_k}{W_a}_{(\lambda_j)}^{\frac{N+2}{N-2}}\big({W_a}_{(\lambda_{j+1})}+{W_a}_{(\lambda_{j-1})}\big)\ud x.
\end{align*}

By \eqref{eq: the formula of Wa}, we have
\begin{align*}
&\int{W_a}_{(\lambda)}^{\frac{N}{N-2}}{W_a}_{(\mu)}^{\frac{N}{N-2}}\ud x\nonumber\\
=&(\lambda\mu)^{-\frac{N}{2}}\int_{0}^{\infty}\min\Big\{\Big(\frac{r}{\lambda}\Big)^{\frac{N(\beta-1)}{2}},\Big(\frac{\lambda}{r}\Big)^{\frac{N(\beta+1)}{2}}\Big\}\min\Big\{\Big(\frac{r}{\mu}\Big)^{\frac{N(\beta-1)}{2}},\Big(\frac{\mu}{r}\Big)^{\frac{N(\beta+1)}{2}}\Big\}r^{N-1}\ud r\nonumber\\
\lesssim&\Big(\frac{\lambda}{\mu}\Big)^{\frac{N\beta}{2}}.
\end{align*}
Thus
\begin{align}\label{the first estimate of 5.38}
\int_{B_k}\big|P_j(x)\big|\ud x\lesssim \gamma^{\frac{N\beta}{2}}.
\end{align}

If $k=j$, the domain of integration is $B_j=\big\{x:\sqrt{\lambda_{j}\lambda_{j+1}}\leq|x|\leq\sqrt{\lambda_j\lambda_{j-1}}\big\}$. By \eqref{eq: 2 nonlinear estimate}, writing
\begin{align*}
E_j\triangleq\Big\{x\in\mathbb{R}^N:\Big|\sum\limits_{\ell\neq j}\iota_{\ell}{W_a}_{(\lambda_{\ell})}(x)\Big|\leq{W_a}_{(\lambda_j)}(x)\Big\},
\end{align*}
then on the domain $B_j$, we choose $a=\iota_j{W_a}_{(\lambda_j)}$, $b=\sum\limits_{\ell\neq j}\iota_{\ell}{W_a}_{(\lambda_{\ell})}(x)$, and get
\begin{align*}
&\bigg|F(M)-F(\iota_j{W_a}_{(\lambda_j)})-\frac{N+2}{N-2}{W_a}_{(\lambda_j)}^{\frac{4}{N-2}}\sum\limits_{\ell\neq j}\iota_{\ell}{W_a}_{(\lambda_{\ell})}\bigg|\\
\lesssim&\chi_{E_j}{W_a^2}_{(\lambda_{\ell})}{W_a}_{(\lambda_j)}^{\frac{6-N}{N-2}}+\chi_{\mathbb{R}^{N}\setminus E_j}\Big(\sum\limits_{\ell\neq j}\iota_j{W_a}_{(\lambda_{\ell})}\Big)^{\frac{N+2}{N-2}}\lesssim\sum\limits_{\ell\neq j}{W_a}_{(\lambda_{\ell})}^2{W_a}_{(\lambda_j)}^{\frac{6-N}{N-2}}.
\end{align*}
As a consequence of \eqref{eq: estimate of Lambda Wa}, it holds
\begin{align*}
\int_{B_j}|P_j(x)|\ud x\leq&\sum\limits_{\ell\neq j}\int_{B_j}{W_a}_{(\lambda_{\ell})}^2{W_a}_{(\lambda_j)}^{\frac{6-N}{N-2}}\ud x+\sum\limits_{\ell\neq\{j,j\pm1\}}\int_{B_j}{W_a}_{(\lambda_{\ell})}{W_a}_{(\lambda_j)}^{\frac{N+2}{N-2}}\ud x\\
\leq&\sum\limits_{\ell\neq j}\int_{B_j}{W_a}_{(\lambda_{\ell})}^{\frac{N}{N-2}}{W_a}_{(\lambda_j)}^{\frac{N}{N-2}}\ud x+\sum\limits_{\ell\neq\{j,j\pm 1\}}\int_{B_j}{W_a}_{(\lambda_{\ell})}\cdot\big(-\Delta+\frac{a}{|x|^2}\big){W_a}_{(\lambda_j)}\ud x.
\end{align*}

If $\ell\neq\{j,j\pm1\}$,\eqref{2.10} gives rise to
\begin{align*}
&\int_{B_j}{W_a}_{(\lambda_{\ell})}\cdot\big(-\Delta+\frac{a}{|x|^2}\big){W_a}_{(\lambda_j)}\ud x\\
=&\int_{B_j}\big(\partial_r+\frac{c}{r}\big){W_a}_{(\lambda_{\ell})}\cdot\big(\partial_r+\frac{c}{r}\big){W_a}_{(\lambda_j)}\ud x\\
\lesssim&\min\Big\{\Big(\frac{\lambda_j}{\lambda_{\ell}}\Big)^{\frac{(N-2)\beta}{2}},\Big(\frac{\lambda_{\ell}}{\lambda_j}\Big)^{\frac{(N-2)\beta}{2}}\Big\}\lesssim\gamma^{(N-2)\beta}.
\end{align*}
Collecting the above estimates, we get
\begin{align*}
\int_{\sqrt{\lambda_{j}\lambda_{j+1}}\leq|x|\leq\sqrt{\lambda_j\lambda_{j-1}}}|P_j(x)|\ud x\lesssim\gamma^{\frac{N\beta}{2}}.
\end{align*}
This estimate together with \eqref{the first estimate of 5.38} yields the desired estimate \eqref{eq: the 1 bdd}.

{\bf Proof of \eqref{eq: the 2 bdd}}.
Let $y=\frac{x}{\lambda_j}$, from the definition of $W_a$, we have
\begin{align*}
&\int(\Lambda W_a)_{(\lambda_j)}(x){W_a}_{(\lambda_j)}^{\frac{4}{N-2}}(x){W_a}_{(\lambda_{j+1})}(x)\ud x\\
=&\big(N(N-2)\beta^2\big)^{\frac{N-2}{4}}\Big(\frac{\lambda_j}{\lambda_{j+1}}\Big)^{\frac{(N-2)\beta}{2}}\int(\Lambda W_a)(y){W_a}^{\frac{4}{N-2}}(y)|y|^{-\frac{(\beta+1)(N-2)}{2}}\ud y+O(\gamma^{\frac{N\beta}{2}}).
\end{align*}

Noting that $\Lambda W_a(x)=(x\cdot\nabla+\frac{N-2}{2}){W_a}(x)$ and
\begin{align}\label{Delta 1}
&\Delta\left(|x|^{-\frac{(N-2)(\beta+1)}{2}+2}\right)\nonumber\\
=&\Big(2-\frac{(N-2)(\beta+1)}{2}\Big)\Big[-\frac{(N-2)(\beta+1)}{2}|x|^{-\frac{(N-2)(\beta+1)}{2}-1}-\frac{\sum x_i^2}{|x|}+N|x|^{-\frac{(N-2)(\beta+1)}{2}}\Big]\nonumber\\
=&\Big(2-\frac{(N-2)(\beta+1)}{2}\Big)\frac{(N+2)-\beta(N-2)}{2}|x|^{-\frac{(N-2)(\beta+1)}{2}}.
\end{align}
We easily show that by \eqref{Delta 1}
\begin{align*}
&\int{W_a}^{\frac{4}{N-2}}(x)x\cdot\nabla W_a(x)|x|^{-\frac{(N-2)(\beta+1)}{2}}\ud x\\
=&\frac{N-2}{(N+2)\big(2-\frac{(N-2)(\beta+1)}{2}\big)}\int\nabla\big({W_a}^{\frac{N+2}{N-2}}(x)\big)\nabla\Big(|x|^{2-\frac{(N-2)(\beta+1)}{2}}\Big)\ud x\\
=&-\frac{N-2}{(N+2)\big(2-\frac{(N-2)(\beta+1)}{2}\big)}\int\Delta\Big(|x|^{2-\frac{(N-2)(\beta+1)}{2}}\Big){W_a}^{\frac{N+2}{N-2}}\ud x\\
=&\Big(-\frac{N-2}{2}+\frac{(N-2)^2\beta}{2(N+2)}\Big)\int{W_a}^{\frac{N+2}{N-2}}(x)|x|^{-\frac{(N-2)(\beta+1)}{2}}\ud x.
\end{align*}
Then we obtain
\begin{align*}
\int(\Lambda W_a)(x){W_a}^{\frac{4}{N-2}}(x)|x|^{-\frac{(N-2)(\beta+1)}{2}}\ud x=\frac{(N-2)^2\beta}{2(N+2)}\int{W_a}^{\frac{N+2}{N-2}}(x)|x|^{-\frac{(N-2)(\beta+1)}{2}}\ud x
\end{align*}
and \eqref{eq: the 2 bdd} follows.

{\bf Proof of \eqref{eq: the 3 bdd}}. For $2\leq j\leq J$, it is easy to see that
\begin{align*}
W_a\Big(\frac{\lambda_j}{\lambda_{j-1}}x\Big)\sim\big(N(N-2)\beta^2\big)^{\frac{N-2}{4}}\Big(\frac{\lambda_j}{\lambda_{j-1}}x\Big)^{\frac{(\beta-1)(N-2)}{2}}.
\end{align*}
Therefore,
\begin{align*}
&\int{W_a}_{(\lambda_j)}^{\frac{4}{N-2}}(x)(\Lambda W_a)_{(\lambda_j)}(x){W_a}_{(\lambda_{j-1})}(x)\ud x\\
=&\Big(\frac{\lambda_j}{\lambda_{j-1}}\Big)^{\frac{N-2}{2}}\int W_a^{\frac{4}{N-2}}\Big(\frac{x}{\lambda_j}\Big)(\Lambda W_a)\Big(\frac{x}{\lambda_j}\Big)W_a\Big(\frac{x}{\lambda_{j-1}}\Big)\ud x\\
=&\Big(\frac{\lambda_j}{\lambda_{j-1}}\Big)^{\frac{(N-2)\beta}{2}}\left((N(N-2)\beta^2\right)^{\frac{N-2}{4}}\int{W_a}^{\frac{4}{N-2}}(x)(\Lambda W_a)(x)|x|^{\frac{(\beta-1)(N-2)}{2}}\ud x+O(\gamma^{\frac{N\beta}{2}}).
\end{align*}
A similar calculation to equation \eqref{Delta 1} yields
\begin{align*}
\int{W_a}^{\frac{4}{N-2}}(x)(\Lambda W_a)(x)|x|^{\frac{(\beta-1)(N-2)}{2}}\ud x=-\frac{(N-2)^2\beta}{2(N+2)}\int{W_a}^{\frac{N+2}{N-2}}(x)|x|^{\frac{(\beta-1)(N-2)}{2}}\ud x,
\end{align*}
we have
\begin{align}\label{before the translate}
&\int{W_a}_{(\lambda_j)}^{\frac{4}{N-2}}(x){W_a}_{(\lambda_{j-1})}(x)(\Lambda W_a)_{(\lambda_j)}(x)\ud x\nonumber\\
=&-\Big(\frac{\lambda_j}{\lambda_{j-1}}\Big)^{\frac{(N-2)\beta}{2}}\frac{N^{\frac{N-2}{4}}(N-2)^{\frac{N+6}{4}}\beta^{\frac{N}{2}}}{2(N+2)}\int{W_a}^{\frac{N+2}{N-2}}(x)|x|^{\frac{(\beta-1)(N-2)}{2}}\ud x+O(\gamma^{\frac{N\beta}{2}}).
\end{align}
Assume $\sigma=\frac{1}{r}$, then direct computation shows
\begin{align*}
W_a^{\frac{N+2}{N-2}}(r)=&\big(N(N-2)\beta^2\big)^{\frac{N+2}{4}}|x|^{\frac{(\beta-1)(N+2)}{2}}\big(1+|x|^{2\beta}\big)^{-\frac{N+2}{2}}\\
=&\big(N(N-2)\beta^2\big)^{\frac{N+2}{4}}\sigma^{\frac{(\beta-1)(N+2)}{2}}\big(1+\sigma^{-2\beta}\big)^{-\frac{N+2}{2}}\\
=&\big(N(N-2)\beta^2\big)^{\frac{N+2}{4}}\sigma^{\frac{(\beta-1)(N+2)}{2}}\big(1+\sigma^{2\beta}\big)^{-\frac{N+2}{2}}\sigma^{N+2}\\
=&W_a^{\frac{N+2}{N-2}}(\sigma)\sigma^{N+2},
\end{align*}
and \eqref{before the translate} can be written as
\begin{align*}
\int{W_a}^{\frac{N+2}{N-2}}(r)r^{\frac{(\beta-1)(N-2)}{2}}r^{N-1}\ud r=&\int W_a^{\frac{N+2}{N-2}}(\sigma)\sigma^{\frac{(\beta-1)(N-2)}{2}}\sigma\ud \sigma\\
=&\int W_a^{\frac{N+2}{N-2}}(\sigma)\sigma^{-\frac{(\beta+1)(N-2)}{2}}\sigma^{N-1}\ud \sigma.
\end{align*}
This completes the proof of \eqref{eq: the 3 bdd} and \eqref{eq: the second derivative of U} follows.

Combining \eqref{eq: lambda beta' first order}, \eqref{the bdd of the easy term} and the fact that $\frac{(N-1)\beta-2}{2}<\frac{N\beta}{2}$, $\gamma\ll 1$, we reach \eqref{eq: lambda second deravitive estimate}.
\end{proof}

\begin{proposition}\label{prop: conclusion}
For $(N-2)\beta>2$. For all large $n$ and for all $t \in\left[\tilde{t}_{n}, t_{n}\right]$,
\begin{align}\label{the conclusion of delta}
&\delta\lesssim\gamma^{\frac{(N-2)\beta}{4}}+o_{n}(1)
\end{align}
\begin{align}\label{the conclusion of beta lambda}
\Big|\beta_{j}-\|\Lambda W_{a}\|_{2}^{2}\lambda_{j}^{\prime}\Big|\lesssim\gamma^{\frac{(N-2)\beta-2}{4}}+o_{n}(1)
\end{align}
\begin{align}\label{the conclusion of beta 2}
\bigg|\frac{1}{2}\sum\limits_{j=1}^{J}\beta_{j}^{2}-\kappa_{1}\sum\limits_{1\leq j\leq J-1}\iota_{j}\iota_{j+1}\Big(\frac{\lambda_{j+1}}{\lambda_{j}}\Big)^{\frac{(N-2)\beta}{2}}\bigg|\lesssim\gamma^{\frac{(N-2)\beta-2}{2}}+o_{n}(1)
\end{align}
\begin{align}\label{kappa1}
\left|\lambda_{j}\beta^{\prime}_{j}+\kappa_{0}\Big(\iota_{j}\iota_{j+1}\Big(\frac{\lambda_{j+1}}{\lambda_{j}}\Big)^{\frac{(N-2)\beta}{2}}-\iota_{j}\iota_{j-1}\Big(\frac{\lambda_{j}}{\lambda_{j-1}}\Big)^{\frac{(N-2)\beta}{2}}\Big)\right|\lesssim\gamma^{\frac{N\beta}{2}}+o_{n}(1)
\end{align}
where $o_{n}(1)$ goes to 0 as $n \rightarrow \infty$, uniformly with respect to $n$, $t \in\left[\tilde{t}_{n}, t_{n}\right]$, and
\begin{align*}
\kappa_{0}=&\frac{N(N-2)^{\frac{N-2}{4}}\beta^{\frac{N}{2}}}{2}\int|x|^{-\frac{(\beta+1)(N-2)}{2}}W_{a}^{\frac{N+2}{N-2}}(x)\ud x,\\
\kappa_{1}=&\|\Lambda W_{a}\|_{L^{2}}^{2}(N(N-2)\beta^2)^{\frac{N-2}{4}}\int|x|^{-\frac{(\beta+1)(N-2)}{2}} W^{\frac{N+2}{N-2}}(x)\ud x.
\end{align*}
\end{proposition}
\begin{proof}
\eqref{the conclusion of delta} follows from \eqref{the estimate of delta}. Due to \eqref{the bdd of beta and alpha}, \eqref{the bdd of alpha and gamma} and \eqref{the estimate of delta} we have
\begin{align*}
\left|\beta_j-\lambda_{j}'\|\Lambda W_a\|_{L^2}^2\right|\lesssim&\left|\beta_j+\alpha_{j}\|\Lambda W_{a}\|_{L^2}^2\right|+\left|\alpha_{j}+\lambda_{j}'\right|\|\Lambda W_a\|_{L^2}^2\\
\lesssim&\gamma^{\frac{(N-2)\beta-2}{4}}+\delta^{\frac{N}{N-2}}+\gamma^{\frac{N\beta}{4}}\lesssim\gamma^{\frac{(N-2)\beta-2}{4}}.
\end{align*}
Combining \eqref{the bdd of alpha and delta}, \eqref{the bdd of beta and alpha} \eqref{the bdd of delta and kappa1} and the estimate $|\alpha_j|+|\beta_j|\lesssim\gamma^{\frac{(N-2)\beta-2}{4}}$, for \eqref{the conclusion of beta lambda} it follows that
\begin{align}\label{the bdd alpha2 beta 2}
&\left|\beta_j^2-\|\Lambda W_a\|_{L^2}^4\alpha_j^2\right|=\left|(\beta_j-\|\Lambda W_{a}\|_{L^2}^2\alpha_j)(\beta_j+\|\Lambda W_{a}\|_{L^2}^2\alpha_j)\right|\nonumber\\
\lesssim&\gamma^{\frac{(N-2)\beta-2}{4}}\gamma^{\frac{(N-2)\beta-2}{2}}+o_{n}(1)\lesssim\gamma^{\frac{3(N-2)\beta}{4}-1}+o_{n}(1).
\end{align}
Replacing $\beta$ with $\alpha$, according to \eqref{the bdd alpha2 beta 2}, \eqref{the bdd of alpha and delta} and \eqref{the bdd of delta and kappa1} we get \eqref{the conclusion of beta 2}. Furthermore,
\begin{align*}
&\bigg|\frac{1}{2}\sum\limits_{j=1}^{J}\beta_{j}^2-\kappa_{1}\sum\limits_{j=1}^{J-1}\iota_{j}\iota_{j+1}\Big(\frac{\lambda_{j+1}}{\lambda_j}\Big)^{\frac{(N-2)\beta}{2}}\bigg|\\
=&\frac{1}{2}\bigg|\sum\limits_{j=1}^{J}\beta_{j}^2-\|\Lambda W_a\|_{L^2}^2\kappa_{1}'\sum\limits_{j=1}^{J-1}\iota_{j}\iota_{j+1}\Big(\frac{\lambda_{j+1}}{\lambda_j}\Big)^{\frac{(N-2)\beta}{2}}\bigg|\\
\lesssim&\frac{1}{2}\bigg|\|\Lambda W_a\|_{L^2}^4\sum\limits_{j=1}^{J}\alpha_j^2-\sum\limits_{j=1}^{J}\beta_j^2\bigg|+\frac{1}{2}\bigg|\delta^2-\|\Lambda W_a\|_{L^2}^2\sum\limits_{j=1}^{J}\alpha_j^2\bigg|\|\Lambda W_a\|_{L^2}^2\\
&+\bigg|\frac{1}{2}\delta^2-\kappa_1'\sum\limits_{j=1}^{J}\iota_j\iota_{j+1}\Big(\frac{\lambda_{j+1}}{\lambda_j}\Big)^{\frac{(N-2)\beta}{2}}\bigg|\big\|\Lambda W_a\big\|_{L^2}^2\\
\lesssim&\gamma^{\frac{(N-2)\beta-2}{2}}+o_{n}(1),
\end{align*}
which completes the proof of \eqref{kappa1}.
\end{proof}

\section{Proof of Theorem \ref{main thm}}
In this section, we will eliminate the $o_{n}(1)$ terms in the estimates of the scaling parameters and choose $\tilde{J}$ as the unique index in $\llbracket1, J\rrbracket$ such that the exterior profiles are not equal to $\pm W_a$. We first prove the lower bounds of the exterior scaling parameter $\ell$. Second, using the conditions we got in Section \ref{section 7} to prove the key estimates about the coefficients in Proposition \ref{the main estimate without o1} and get the contradiction. Finally we conclude the proof of Theorem \ref{main thm}.

Recalling from Section \ref{section 7} the definitions of $t_{n}$, $\tilde{t}_{n}$, $J$, and for $j\in\llbracket1, J\rrbracket$, $\iota_{j}$, $\alpha_{j}(t)$, $\beta_{j}(t)$, $\lambda_{j}(t)$, after extraction of subsequences, the following weak limits exists in $\mathcal{H}$,
\begin{align*}
\left(\widetilde{U}_{0}^{j}, \widetilde{U}_{1}^{j}\right)=\underset{n \rightarrow \infty}{\mathrm{w}-\lim }\left(\lambda_{j}\left(\tilde{t}_{n}\right)^{\frac{N}{2}-1} U\left(\tilde{t}_{n}, \lambda_{j}\left(\tilde{t}_{n}\right) \cdot\right), \lambda_{j}\left(\tilde{t}_{n}\right)^{\frac{N}{2}} \partial_{t} U\left(\tilde{t}_{n}, \lambda_{j}\left(\tilde{t}_{n}\right) \cdot\right)\right)
\end{align*}
where $U=u-v_{L}$ as before. We note that there exists $j \in\llbracket1, J\rrbracket$ such that $\left(\tilde{U}_{0}^{j}, \tilde{U}_{1}^{j}\right)\neq\left(\iota_{j} W_a, 0\right)$.
If not, for all $k\in\llbracket1, K-1\rrbracket$, denote $\mu_{k,n}=\lambda_{j_k}(s_n)$, after extracting subsequence, by the limits \eqref{eq: initial date}, \eqref{eq: weak limit}, we obtain for $n\rightarrow\infty$
\begin{align}
\left(\mu^{\frac{N-2}{2}}_{k,n}(u-v_{L})(s_n,\mu_{k,n}\cdot),\mu^{\frac{N}{2}}_{k,n}\partial_{t}(u-v_L)(s_n,\mu_{k,n}\cdot)\right)\rightharpoonup\left(V_0^k,V_1^k\right).
\end{align}

By the assumption \eqref{distance d t},
\begin{align}
\bigg\|\left(V_0^k,V_1^k\right)-\sum\limits_{j=j_k}^{j_{k+1}-1}\iota_j({W_a}_{(\lambda_j)},0)\bigg\|_{\mathcal{H}_a}\lesssim\epsilon_0.
\end{align}
And for $j\in\llbracket j_k,j_{k+1}-2\rrbracket$, \eqref{def GJ} and \eqref{def: dJ, gamma} imply
\begin{align}
\frac{\nu_{j+1}}{\nu_j}\leq\lim\limits_{n\rightarrow\infty}\gamma(\lambda(s_n))\leq \epsilon_0,
\end{align}
which means $j_{k+1}=j_k+1$. According to the partition of the interval $\llbracket 1,J\rrbracket$ in \eqref{eq: not in a term}, there is only one element in each sub-interval such that for all $j\in \llbracket 1,J\rrbracket$, $
\lim\limits_{n\rightarrow\infty}\frac{\lambda_{j+1}(\tilde{t}_n)}{\lambda_j(\tilde{t}_n)}=0$, which implies $\lim\limits_{n\rightarrow\infty}\gamma(\tilde{t}_n)=0$. Hence, \eqref{the conclusion of delta} yields $\lim\limits_{n\rightarrow\infty}\delta(\tilde{t}_n)=0$, which is a contradiction with the definition of $\delta$ and $\tilde{t}_n$. We assume that $\bar{J}$ is the smallest index of $\llbracket 1,J\rrbracket$ such that $(\tilde{U}^{\bar{J}}_0,\tilde{U}^{\bar{J}}_1)$ is not the ground state, namely,  $\bar{J}$ is the unique index such that
\begin{align*}
\left(\tilde{U}_{0}^{j}, \tilde{U}_{1}^{j}\right)=&\left(\iota_{j} W_a, 0\right), \quad\forall j \in\llbracket 1,\bar{J}-1\rrbracket,\\
\left(\tilde{U}_{0}^{\bar{J}}, \tilde{U}_{1}^{\bar{J}}\right) \neq&\left(\iota_{\bar{J}}{W_a}, 0\right),\quad\forall j \in\llbracket \bar{J},J\rrbracket.
\end{align*}
For simplify, we denote 
\begin{align*}
\lambda_{\bar{J}, n}\triangleq\lambda_{\bar{J}}\left(\tilde{t}_{n}\right), \quad \tilde{\gamma}(t)\triangleq\gamma\left(\left(\lambda_{j}(t)\right)_{\bar{J} \leq j \leq J}\right)=\max _{\bar{J} \leq j \leq J-1} \frac{\lambda_{j+1}(t)}{\lambda_{j}(t)} .
\end{align*}

We introduce the following lemma. Readers can refer to \cite{DKM7} for the proof.

\begin{lemma}\emph{(\cite{DKM7})}\label{lem: the limits of t, lambda}
\begin{align*}
\lim _{n \rightarrow \infty} \frac{t_{n}-\tilde{t}_{n}}{\lambda_{\widetilde{J}, n}}=+\infty.
\end{align*}
\end{lemma}

\begin{lemma}\emph{(\cite{DKM7})}\label{the tilde gamma bdd}Let $T>0$ and
\begin{align*}
t_{n}^{\prime}=\tilde{t}_{n}+T \lambda_{\tilde{J}, n}.
\end{align*}
Then for $(N-2)\beta>2$, for large $n$, for all $t \in\left[\tilde{t}_{n}, t_{n}^{\prime}\right]$ and for all $j\in\llbracket\tilde{J},J\rrbracket$,
\begin{align}\label{the tilde bdd of beta lambda}
\Big|\beta_{j}-\|\Lambda W_a\|_{L^{2}}^{2} \lambda_{j}^{\prime}\Big| \leq C\tilde{\gamma}^{\frac{(N-2)\beta-2}{4}},
\end{align}
\begin{align}\label{the tilde bdd of the second derivative}
\bigg|\frac{1}{2} \sum_{j=\tilde{J}}^{\text {J }} \beta_{j}^{2}-\kappa_{1} \sum_{\tilde{J} \leq j \leq J-1} \iota_{j} \iota_{j+1}\Big(\frac{\lambda_{j+1}}{\lambda_{j}}\Big)^{\frac{(N-2)\beta}{2}}\bigg| \leq C \tilde{\gamma}^{\frac{(N-2)\beta-2}{2}},
\end{align}
and for all $t \in\left[\tilde{t}_{n}, t_{n}^{\prime}\right]$,
\begin{align}\label{the tilde bdd of the beta2}
\bigg|\lambda_{j} \beta_{j}^{\prime}+\kappa_{0}\Big(\iota_{j} \iota_{j+1}\Big(\frac{\lambda_{j+1}}{\lambda_{j}}\Big)^{\frac{(N-2)\beta}{2}}-\iota_{j} \iota_{j-1}\Big(\frac{\lambda_{j}}{\lambda_{j-1}}\Big)^{\frac{(N-2)\beta}{2}}\Big)\bigg| \leq C \tilde{\gamma}^{\frac{N\beta}{2}}.
\end{align}
\end{lemma}
\begin{proof}
Since we have assumed $t'_n=\tilde{t}_{n}+T\lambda_{\tilde{J},n}$, by Lemma \ref{lem: the limits of t, lambda} and restricting $t$ to $[\tilde{t}_{n},t'_{n}]$, we can replace $\gamma(t)+o_{n}(1)$ by $\tilde{\gamma}(t)$, the following claim can be proved by exploiting the method in \cite{DKM7}.
\begin{itemize}
\item
\begin{align}\label{the equation 1 of cliam}
\lim\limits_{n\rightarrow\infty}\max\limits_{\tilde{t}_{n}\leq t\leq t'_{n}}|\beta_j(t)|+\frac{\lambda_{j+1}(t)}{\lambda_{j}(t)}=0,
\end{align}
\item
\begin{align}\label{the equation 2 of claim}
\liminf\limits_{n\rightarrow\infty}\min\limits_{\tilde{t}_n\leq t\leq t'_n}\tilde{\gamma}(t)>0.
\end{align}
\end{itemize}
 Thus we can restrict the indices $j$ to $\llbracket\tilde{J},J\rrbracket$. According to Proposition \ref{prop: conclusion}, the proof is complete.
\end{proof}

As a consequence of Proposition \ref{exterior profile estimate}, we prove a lower bound on one of the scaling parameters $\lambda_j$.

\begin{lemma}\label{lem: the lower bdd of lambda j}
Let $\ell$ be defined by Theorem \ref{thm: main thm of Xi} and $m_{0}=\frac{(N-2)\beta+1}{2}$  . Then if $\varepsilon_{0}=\varepsilon_{0}(T)$ is chosen small enough, there exists a constant $C>0$ such that for large $n$,
\begin{align}\label{the lower bdd of lambda j}
\forall t \in\left[\tilde{t}_{n}, t_{n}^{\prime}\right], \quad|\ell| \leq C\bigg(\frac{\lambda_{\widetilde{J}}(t)}{\lambda_{\widetilde{J}}\left(\tilde{t}_{n}\right)}\bigg)^{(m_{0}-\frac{1}{2})/\beta} \delta(t)^{\frac{2N}{N-2}}.
\end{align}
\end{lemma}

\begin{proof}
We prove this inequality by contradiction. Assuming (after extracting a subsequence) that there exists a large constant $M$ and a sequence $\left\{\tilde{s}_{n}\right\}_{n}$ with $\tilde{s}_{n} \in\left[\tilde{t}_{n}, t_{n}^{\prime}\right]$ and for all $n$
\begin{align}\label{the contradiction assumption}
|\ell| \geq M\bigg(\frac{\lambda_{\tilde{J}}\left(\tilde{s}_{n}\right)}{\lambda_{\tilde{J}}\left(\tilde{t}_{n}\right)}\bigg)^{(m_{0}-\frac{1}{2})/{\beta}} \delta\left(\tilde{s}_{n}\right)^{\frac{2N}{N-2}}.
\end{align}
As in the proof in \cite{DKM7} and by the non-radiative profile decomposition Lemma \ref{lem: nonradiative profile}, we can find a constant $\tilde{\lambda}\in[0,\infty\}$ such that
\begin{align}\label{def: limits of tilde lambda}
\lim\limits_{n\rightarrow\infty}\frac{\lambda_{\tilde{J}}(\tilde{t}_n)}{\lambda_{\tilde{J}}(\tilde{s}_n)}=\tilde{\lambda}.
\end{align}

Let $m_0=\frac{(N-2)\beta+1}{2}$. By Theorem \ref{thm: main thm of Xi}, for large $R$,
\begin{align*}
\left\|\left(\tilde{\lambda}^{\frac{N-2}{2}} V_{0}^{\widetilde{J}}(\tilde{\lambda} \cdot), \tilde{\lambda}^{\frac{N}{2}} V_{1}^{\tilde{J}}(\tilde{\lambda} \cdot)\right)-\ell \Xi_{p_{0}}\right\|_{\mathcal{H}_a(R)}
\lesssim \max \bigg\{\frac{1}{R^{(2m_{0}-\frac{N}{2})\frac{N+2}{N-2}}}, \frac{1}{R^\frac{N}{2}}\bigg\},
\end{align*}
and after scaling, for large $R$,
\begin{align}\label{the contradiction}
\left\|\left(V_{0}^{\tilde{J}}, V_{1}^{\tilde{J}}\right)-\tilde{\lambda}^{(m_{0}-\frac{1}{2})/\beta} \ell \Xi_{p_{0}}\right\|_{\mathcal{H}_a(R)} \lesssim \max\bigg\{\frac{1}{R^{(2m_{0}-\frac{N}{2})\frac{N+2}{N-2}}}, \frac{1}{R^\frac{N}{2}}\bigg\}.
\end{align}

Replacing $|\iota|$ by $\tilde{\lambda}^{(m_0-\frac{1}{2})/\beta}|\iota|$ in Proposition \eqref{exterior profile estimate} gives
\begin{align*}
\tilde{\lambda}^{(m_0-\frac{1}{2})/\beta}|\ell|\leq 2C \delta(\tilde{s}_n)^{\frac{2N}{N-2}}.
\end{align*}
On the one hand, using \eqref{def: limits of tilde lambda}, for large $n$, we obtain
\begin{align*}
\bigg(\frac{\lambda_{\widetilde{J}}\left(\tilde{t}_{n}\right)}{\lambda_{\widetilde{J}}\left(\tilde{s}_{n}\right)}\bigg)^{(m_0-\frac{1}{2})/\beta}|\ell| \leq 2 C \delta\left(\tilde{s}_{n}\right)^{\frac{2N}{N-2}}.
\end{align*}
On the other hand, by \eqref{the contradiction assumption} we deduce for large $n$
\begin{align*}
\bigg(\frac{\lambda_{\widetilde{J}}\left(\tilde{t}_{n}\right)}{\lambda_{\widetilde{J}}\left(\tilde{s}_{n}\right)}\bigg)^{(m_0-\frac{1}{2})/\beta}|\ell| \geq M\delta\left(\tilde{s}_{n}\right)^{\frac{2N}{N-2}}.
\end{align*}
Since we can choose $M$ arbitrarily, let $M=3C$, we obtain that $\delta\left(\tilde{s}_{n}\right)=0$ for large $n$, which contradicts \eqref{the contradiction}.
\end{proof}

Using Proposition \ref{prop: conclusion},   the following proposition can be obtained as in \cite{DKM7}.
\begin{proposition}\emph{(\cite{DKM7})}\label{the main estimate without o1}
Let $(N-2)\beta>2$, $C>0$, $J_{0}\geq 2$, $a>0$. There exists $\varepsilon=\varepsilon\left(C, J_{0}, a\right)>0$, such that for all $L>0$, there exists $T^{*}=T^{*}\left(L, C, J_{0}, a\right)$ with the following property. For all $T>0$, for all $C^{1}$ functions
\begin{align*}
\boldsymbol{\lambda} &=\left(\lambda_{j}\right)_{j}:[0, T] \rightarrow G_{J_{0}} \\
\boldsymbol{\beta} &=\left(\beta_{j}\right)_{j}:[0, T] \rightarrow \mathbb{R}^{J_{0}},
\end{align*}
satisfying, for all $t \in[0, T]$, for any $j$,
\begin{align*}
\gamma(\boldsymbol{\lambda})\triangleq \gamma \leq \varepsilon,
\end{align*}
\begin{align*}
 \Big|\beta_{j}-\|\Lambda W_a\|_{L^{2}}^{2} \lambda_{j}^{\prime}\Big| \leq C \gamma^{\frac{(N-2)\beta-2}{4}},
\end{align*}
\begin{align*}
\bigg|\frac{1}{2} \sum_{j=1}^{J_{0}} \beta_{j}^{2}-\kappa_{1} \sum_{1 \leq j \leq J_{0}-1} \iota_{j} \iota_{j+1}\Big(\frac{\lambda_{j+1}}{\lambda_{j}}\Big)^{\frac{(N-2)\beta}{2}}\bigg| \leq C \gamma^{\frac{(N-2)\beta-2}{2}},
\end{align*}
\begin{align*}
\bigg|\lambda_{j} \beta_{j}^{\prime}+\kappa_{0}\Big(\iota_{j} \iota_{j+1}\Big(\frac{\lambda_{j+1}}{\lambda_{j}}\Big)^{\frac{(N-2)\beta}{2}}-\iota_{j} \iota_{j-1}\Big(\frac{\lambda_{j}}{\lambda_{j-1}}\Big)^{\frac{(N-2)\beta}{2}}\Big)\bigg| \leq C \gamma^{\frac{N\beta}{2}},
\end{align*}
\begin{align*}
L \leq C \gamma^{\frac{N-2}{2}}\left(\frac{\lambda_{1}}{\lambda_{1}(0)}\right)^{a},
\end{align*}
we have
\begin{align*}
T \leq T^{*} \lambda_{1}(0).
\end{align*}
\end{proposition}

\medskip
{\it Proof of Theorem \ref{main thm}} By \eqref{the estimate of delta}, for all $t \in\left[\tilde{t}_{n}, t_{n}^{\prime}\right]$, it holds that
\begin{align*}
\delta(t)^{\frac{2N}{N-2}} \lesssim \gamma(t)^{\frac{N\beta}{2}}+o_{n}(1),
\end{align*}
which, combined with \eqref{the equation 1 of cliam} and \eqref{the equation 2 of claim}, implies that for all $t \in\left[\tilde{t}_{n}, t_{n}^{\prime}\right]$,
\begin{align*}
\delta(t)^{\frac{2N}{N-2}} \lesssim \tilde{\gamma}(t)^{\frac{N\beta}{2}}.
\end{align*}

According to Lemma \ref{lem: the lower bdd of lambda j}, there exists a constant $C>0$ such that for any $t \in\left[\tilde{t}_{n}, t_{n}^{\prime}\right]$
\begin{align}\label{the lower bdd of lambda j by gamma}
|\ell| \leq C\Big(\frac{\lambda_{\tilde{J}}(t)}{\lambda_{\tilde{J}}\left(\tilde{t}_{n}\right)}\Big)^{(m_{0}-\frac{1}{2})/{\beta}} \tilde{\gamma}(t)^{\frac{N\beta}{2}}.
\end{align}
By \eqref{the tilde bdd of beta lambda}, \eqref{the tilde bdd of the second derivative}, \eqref{the tilde bdd of the beta2} of Lemma \ref{the tilde gamma bdd}, the parameters $\left(\beta_{j}\right)_{\tilde{J} \leq j \leq J}$ and $\left(\lambda_{j}\right)_{\tilde{J} \leq j \leq J}$ satisfy the assumptions of Proposition \ref{the main estimate without o1} for times $t \in\left[\tilde{t}_{n}, t_{n}^{\prime}\right]$. Hence, Proposition \ref{the main estimate without o1} yields that for large $n$
\begin{align*}
t_{n}^{\prime}-\tilde{t}_{n} \leq T_{*} \lambda_{\tilde{J}}\left(\tilde{t}_{n}\right).
\end{align*}
Thus the assumption $t'_{n}-\tilde{t}_{n}=T\lambda_{\tilde{J}}(t_n)$ implies that $T\leq T_{*}$, for a constant $T_{*}$ depending only on the solution $u$ and the parameters $\ell, m_{0}$. Since $T$ can be taken arbitrarily large, we obtain a contradiction and complete the proof of Theorem \ref{main thm}.


\begin{thebibliography}{10}
    \bibitem{AS}M.~Abramowitz, I.~A.~Stegun,
    Handbook of Mathematical Functions with Formulas, Graphs, and Mathematical Tables.
    \emph{National Bureau of Standards Applied Mathematics Serie} (1964), Vol. 55, MR0167642

    \bibitem{BG}H. Bahouri, P. Gerard, High frequency approximation of solutions to critical nonlinear wave equations, \emph{Amer. J. Math}, \textbf{121} (1999), no. 1, 131-175.

    \bibitem{BCLPZ} A. Bulut, M. Czubak, D. Li, N. Pavlovic, X. Zhang, Stability and unconditional uniqueness of solutions for energy critical wave equations in high dimensions.
    \emph{Communications in Partial Differential Equations} \textbf{38} (2013) 575-607.

    \bibitem{BPS}N. Burq, F. Planchon, J. Stalker, A. Tahvildar-Zadeh, Strichartz estimates for the wave and Schr\"odinger equations with inverse-square potential, \emph{J. Func. Anal.}, \textbf{203}(2) (2003) 519-549.

    \bibitem{CDKM}C. ~Collot,  T.~Duyckaerts, C.~E.~Kenig and F.~Merle, Soliton resolution for the radial quadratic wave equation in six space dimensions, \emph{arxiv:2201.01848}.


    \bibitem{CKS}R. ~Cote, C. E. ~Kenig, W. ~Schlag, Energy partition for the linear radial wave equation, \emph{Math. Ann.} \textbf{358} (2014), 573-607.

    \bibitem{Cote}R. Cote, On the soliton resolution for equivariant wave maps to the sphere, \emph{Comm. Pure Appl. Math.} \textbf{68} (2015), no. 11, 1946–2004.


     \bibitem{CKLS}R. ~Cote, C. E. ~Kenig, A. ~Lawrie, W. ~Schlag, Profiles for the radial focusing 4d energy-critical wave equation, \emph{Comm. Math. Phy.} \textbf{357} (2018), 943-1008.

         \bibitem{DL}Du, X. Li, high dimensions

    \bibitem{DKM} T.~Duyckaerts, C.~E.~Kenig and F.~Merle,
    Classification of radial solutions of the focusing, energy-critical wave equation.
    \emph{Cambridge J. Math.} \textbf{1} (2013), 75--144.

    \bibitem{DKM1} T.~Duyckaerts, C.~E.~Kenig and F.~Merle,
    Universality of blow-up profile for small radial type II blow-up solutions of the energy-critical wave equation. \emph{J. Eur. Math. Soc.} \textbf{13} (2011), 533-599.

   \bibitem{DKM6} T.~Duyckaerts, C.~E.~Kenig and F.~Merle,
Universality of blow-up profile for small type II blow-up solutions of the energy-critical wave equation: the nonradial case. \emph{J. Eur. Math. Soc.} \textbf{13} (2011), 533-599.

\bibitem{DKM8} T.~Duyckaerts, C.~E.~Kenig and F.~Merle, Profiles of bounded radial solutions of the focusing , energy-critical wave equation, \emph{Geom. Funct. Anal.} \textbf{22}(2012), 639-698.

    \bibitem{DKM5} T.~Duyckaerts, C.~Kenig, F.~Merle,
    Scattering profile for global solutions of the energy-critical wave equation,
    \emph{J. Eur. Math. Soc.} \textbf{21} (2019), no. 7, 2117-2162.
    \bibitem{DKM3} T.~Duyckaerts, C.~E.~Kenig and F.~Merle, Decay estimates for nonradiative solutions of the energy-critical focusing wave equation. \emph{J. Geom. Anal.} \textbf{31} (2021), no. 7, 7036–7074.
    \bibitem{DKM2} T.~Duyckaerts, C.~E.~Kenig and F.~Merle, Exterior energy bounds for the critical wave equation close to the ground state. \emph{Comm. Math. Phys.} \textbf{379} (2020), no. 3, 1113–1175.
    \bibitem{DKM7} T.~Duyckaerts, C.~E.~Kenig and F.~Merle, Soliton resolution for the radial critical wave equation in all odd space dimensions, arxiv:1912.07664, to appear in \emph{Acta. Math.}


    \bibitem{DJKM} T.~Duyckaerts, H.~Jia, C.~E.~Kenig and F.~Merle,
    Soliton resolution along a sequence of times for the focusing energy critical wave equation, \emph{Geom. Funct. Anal.} \textbf{27} (2017), no. 4, 798--862.
    \bibitem{DJKM2} T.~Duyckaerts, H.~Jia, C.~E.~Kenig and F.~Merle, Universality of blow up profile for small blow up solutions to the energy critical wave map equation. \emph{Int. Math. Res. Not.} 2018, no. 22, 6961–7025.
    \bibitem{DKMM}T. Duyckaerts, C. Kenig, Y. Martel, F. Merle, Soliton resolution for critical co-rotational wave maps and radial cubic wave equation, arxiv:2103.01293.

    \bibitem{DY}T. Duyckaerts, J. Yang, Scattering to a stationary solution for the superquintic radial wave equation outside an obstacle, arxiv:1910.00811


    \bibitem{Eck}W. Eckhaus, The long-time behaviour for perturbed wave equations and related problems, in \emph{Trends in Applications of Pure Mathematics  to Mechanics (Bad Honnef, 1985)}, Springer, 1986, 168-194.
    \bibitem{ES}W. Eckhauss, P. Schuur, The emergence of solitons of the Korteweg-de Vries equation from arbitrary initial conditions, \emph{Math. Methods App. Sci.} \textbf{5} (1983), 97-116.

    \bibitem{GV} J. Ginibre, G. Velo, Generalized Strichartz inequalities for the wave equation, \emph{J. Func. Anal.} \textbf{133} (1995) 50-68.

    \bibitem{JL1}J. Jendrej, A. Lawrie, Two-bubble dynamics for threshold solutions to the wave maps equation. \emph{Invent. Math.} \textbf{213} (2018), no. 3, 1249–1325.

    \bibitem{JL2}J. Jendrej, A. Lawrie, An asymptotic expansion of two-bubble wave maps, arxiv:2003.05829.

    \bibitem{JL3}J. Jendrej, A. Lawrie, Uniqueness of two-bubble wave maps, arxiv:2003.05835.

    \bibitem{JL4}J. Jendrej, A. Lawrie, Continuous time soliton resolution for two-bubble equivariant wave maps, arxiv:2010.12506.

    \bibitem{JL5}J. Jendrej, A. Lawrie, Soliton resolution for equivariant wave maps, arxiv:2106.10738.

    \bibitem{JBSX1}H. Jia, B. Liu, S. Wilhelm, G. Xu, Generic and non-generic behavior of solutions to defocusing energy critical wave equation with potential in the radial case. Int. Math. Res. Not. IMRN 2017, no. 19, 5977–6035.
    \bibitem{JBX1}H. Jia, B. Liu, G. Xu,  Long time dynamics of defocusing energy critical 3+1 dimensional wave equation with potential in the radial case. \emph{Comm. Math. Phys.} \textbf{339} (2015), no. 2, 353–384.

    \bibitem{JBX2}H. Jia, B. Liu, G. Xu, Global center stable manifold for the defocusing energy critical wave equation with potential. \emph{Amer. J. Math.} \textbf{142} (2020), no. 5, 1497–1557.


    \bibitem{JK}H. Jia, C. Kenig, Asymptotic decomposition for semilinear wave and quivariant wave map equations, \emph{Amer. J. Math}, \textbf{139} (2017), 1521-1603.
    \bibitem{KSWW}H. Kahlf, U. W. Schmincke, J. Walter, R. W\"ust, On the spectral theory of Schr\"odinger and Dirac operators with strongly singular potentials, \emph{Spectral Theory and Differential } \textbf{448}, Springer, Berlin, 1975, 182-226.

    \bibitem{KLLS}C.~E.~Kenig, A.~Lawrie, B.~Liu and W.~Schlag,
    Channels of energy for the linear radial wave equation.
    \emph{Advances in Math.} \textbf{285} (2015), 877--936.
    \bibitem{KLS}C. E. Kenig, A. Lawrie, W. Schlag, Relaxation of wave maps exterior to a ball to harmonic maps for all data. \emph{Geom. Funct. Anal.} \textbf{24} (2014), no. 2, 610–647.
    \bibitem{Kenig}C. Kenig, The method of energy channels for nonlinear wave equtions, \emph{Disc. Cont. Dyn. Sys. }  \textbf{39} (2019), no. 12, 6979–6993.
    \bibitem{Kenig2}C. Kenig, Asymptotic simplification for solutions of the energy critical nonlinear wave equation. \emph{J. Math. Phys.} 62 (2021), no. 1, Paper No. 011502.

    \bibitem{KMVZZ} R.~Killip, C.~Miao, M.~Visan, J.~Zhang, J.~Zheng,
    The energy-critical NLS with inverse-square potential,
    \emph{Discrete Contin. Dyn. Syst.} \textbf{37} (2017), no. 7, 3831-3866.

        \bibitem{K} M. D. Kruskal, The birth of the soliton, In \emph{Nonlinear evolution equations solvable by the spectral transofrm}, \textbf{26} of \emph{Res. Notes in Math.} Pitman, Boston, Mass., 1978, pp. 1-8.


    \bibitem{S} S.~Schechter,
    On the inversion of certain matrices.
    \emph{Math. Tables Aids Comput.} \textbf{13} (1959), 73--77.



    \bibitem{MMZ}C. Miao, J. Murphy, J. Zheng, The energy-critical nonliniear wave equation with an inverse-square potential, \emph{Ann. I. H. Poincar\'e-AN}  \textbf{37} (2020) 417-456.
    \bibitem{MZZ}C. Miao, J. Zhang, J. Zheng, Strichartz estimates for wave equation with inverse square potential, \emph{Comm. Cont. Math.} \textbf{15} (2013) 1350026.
    \bibitem{PST}F. Planchon, J. Stalker, A. Tahvildar-Zadeh,Lp estimates for the wave equation with the inverse-square potential.Discrete \emph{Contin. Dyn. Syst.} \textbf{9} (2003), 427–442.
    \bibitem{CR}C. Rodriguez, Profiles for the radial focusing energy-critical wave equation in odd dimensions. \emph{Adv. Differential Equations} \textbf{21} (2016), no. 5-6, 505–570.
    \bibitem{Sch}P. Schuur, Asymptotic Analysis of Soliton Problems, Springer-Verlag, Berlin, 1986.

    \bibitem{VZ}J. L. Vasquez, E. Zuazua, The Hardy inequality and the asymptotic behaviour of the heat equation with an inverse-square potential, \emph{J. Func. Anal.} \textbf{173} (2000) 103-153.



\end{thebibliography}
\end{document}